\documentclass[12pt]{article}
\usepackage{amsfonts}
\usepackage{amsthm,amsmath,amssymb}
\usepackage{graphicx}
\usepackage{setspace}
\usepackage{amsfonts,amstext}

\usepackage{authblk}
\usepackage{natbib}
\usepackage{xcolor}
\usepackage{titletoc}
\usepackage{booktabs}
\usepackage{fullpage}
\usepackage{hyperref}
\hypersetup{
	colorlinks,
	linkcolor={red!50!black},
	citecolor={blue!50!black},
	urlcolor={blue!80!black}
}

\newtheorem{lemma}{\textbf{Lemma}}[section]
\theoremstyle{definition}
\newtheorem{example}{\textbf{Example}}[section]
\newtheorem{proposition}{\textbf{Proposition}}[section]
\theoremstyle{plain}
\newtheorem{assumption}{\textbf{Assumption}}[section]

\theoremstyle{definition}
\newtheorem{theorem}{\textbf{Theorem}}[section]
\newtheorem{corollary}{\textbf{Corollary}}[section]
\newtheorem{remark}{Remark}[section]
\theoremstyle{plain}
\newtheorem{definition}{\textbf{Definition}}[section]

\title{Towards a General Large Sample Theory for Regularized Estimators 
\thanks{We would like to thank Peter Bickel, Ivan Canay, Xiaohong Chen, Max Farrell, Eric Gautier, Joel Horowitz, Jack Porter as well as participants in numerous seminars and conferences for their helpful comments. Usual disclaimer applies.}}

\author{Michael Jansson\thanks{UC Berkeley, Dept. of Economics. E-mail: mjansson@berkeley.edu.}~~and Demian Pouzo\thanks{UC Berkeley, Dept. of Economics. E-mail: dpouzo@berkeley.edu (corresponding author).}}

\begin{document}
	\maketitle

\begin{abstract}
	We present a general framework for studying regularized estimators; such estimators are pervasive in estimation problems wherein ``plug-in" type estimators are either ill-defined or ill-behaved. Within this framework,  we derive, under primitive conditions, consistency and a generalization of the asymptotic linearity property. We also provide data-driven methods for choosing tuning parameters that, under some conditions, achieve the aforementioned properties. We illustrate the scope of our approach by presenting a wide range of applications. 
\end{abstract}

\startcontents[section1]
\renewcommand\contentsname{Table of Contents}
\printcontents[section1]{ }{1}{\section*{\contentsname}}

\setlength{\abovedisplayskip}{2.5pt}
\setlength{\belowdisplayskip}{2.5pt}

%
%

\section{Introduction}

It was noted as early as Stein \cite{Stein1956} that in many complex models, the parameter mapping, $\psi$, that links the probability distribution generating the data, $P$, to some parameter space may be ill-behaved or even ill-defined when evaluated at the empirical distribution. The widespread solution in these cases is to regularize the problem. Regularization procedures are ubiquitous in statistics and elsewhere, examples of these include kernel-based estimators; series-based estimators and penalization-based estimators among many others.\footnote{Examples of regularizations are so ubiquitous that providing a thorough review is outside the scope of the paper; see e.g. \cite{BickelLi06}, \cite{buhlmann2011statistics}, \cite{HardleLinton1994} and \cite{Chen2007} for excellent reviews of several regularization methods.} Even though there has been an enormous amount of work in statistics and other sciences studying the properties of these procedures, they are viewed, by and large, as separate and unrelated. In particular, results like consistency or large sample distribution theory, when they exists, they have only been derived in a case-by-case basis;  to our knowledge, there is no general theory or systematic approach. The goal of this paper is to fill this gap by providing the basis for an unifying large sample theory for regularized estimators that will allow us to make systematic progress in studying their large sample properties. 

 Our point of departure is the general conceptual framework put forward by Bickel and Li (\cite{BickelLi06}), wherein the authors propose a general definition of regularization. According to their framework, a regularization can be viewed sequence of parameter mappings, $(\psi_{k})_{k=1}^{\infty}$ that replaces the original parameter mapping, $\psi$, each element is well-behaved, and its limit coincides with the original mapping. The index of this sequence (denoted by $k$) represents what is often referred as the tuning (or regularization) parameter; e.g. it is the (inverse of the) bandwidth for kernels, the number of terms in a series expansion, or the (inverse of the) scale parameter in penalizations. While Bickel and Li's framework encompasses many examples and applications, it is unclear what type of asymptotic properties can be obtained in such a general framework. We provide two set of general theorems under intuitive conditions that establish large sample properties for regularized estimators. One set of results establishes consistency and rate of convergence, and a data-driven method for choosing the tuning parameter that achieves these rates. Another set of results provide foundations for large sample distribution theory by deriving a generalization of the classical asymptotic linearity property.   
 

 Our approach to obtain consistency and convergence rate results is akin to the one used in the standard large sample theory for ``plug-in" estimators, in the sense that it relies on continuity of the mapping used for estimation (see \cite{Wolfowitz1957}, \cite{DonohoLiu1991}). The key difference is that in plug-in estimation this mapping is $\psi$, but for regularized estimators the natural mapping is the (sequence of) \emph{regularized} parameter mappings, $(\psi_{k})_{k=1}^{\infty}$; this difference --- in particular, the fact that we have a sequence of mappings --- introduces nuances that are not present in the standard ``plug-in" estimation case. We show that the key component of the convergence rate is the modulus of continuity of the regularized mapping, which, typically, will deteriorate as one moves further into the sequence of regularized mappings, thus yielding a generalized version of the well-known ``noise-bias" trade-off. While this result, by itself, does not constitute a big leap from Bickel and Li's framework, we use the underlying insights to propose a data-driven method to choose the tuning parameter that under some conditions yields convergence rates proportional to the ``oracle" ones, i.e., those implied by the choice that balances the ``noise-bias" trade-off. This method is an extension of the Lepski method as presented in \cite{PereverzevSchock2006} for ill-posed inverse problems.\footnote{Similar versions has been used in several particular applications. Closest to our examples are the work: by \cite{Pouzo2017} for regularized M-estimators; by \cite{ChenChristensen2015} in non-parametric IV regressions; by \cite{GineNickl2008} for estimation of the integrated square density; by \cite{gaillac2019adaptive} in a random coefficient model; by \cite{lepski1997} for estimation of a function at a point.}

 Our second set of results are concerned with obtaining a type of asymptotic linear representation for regularized estimators. The property of asymptotic linearity is well-known in the literature and is the cornerstone of large sample distribution theory. This property states that the estimator, once centered at the true parameter, is equal to a sample average of a mean zero function --- referred as the influence function --- plus an asymptotically negligible term. 
 
 In parametric models, asymptotic linearity is typically satisfied by commonly used estimators like the ``plug-in" estimator. In more complex settings such semi-/non-parametric models, however, this is not longer true. In such cases, there are no estimators satisfying this property, because, for instance, the efficiency bound of the parameter of interest is infinite, or more generally, the parameter is not root-n estimable. 
 For these situations, asymptotic representations analogous to asymptotic linearity have been obtained in specific examples for specific regularizations, but, to our knowledge, there is no general approach. 
 	This is specially problematic as there is no systematic method for properly standardizing the estimator in situations where the parameter is not root-n estimable.\footnote{For density and regression estimation problems there is a large literature, especially for particular functionals like evaluation at a point; e.g. see \cite{EggermontLariccia2001} Vol I and II for references and results. In more general contexts such as M-estimation and GMM-based models, to our knowledge, the literature is much more sparse with only a few papers allowing for slower than root-n parameters in particular settings. Closest to ours are the papers by \cite{ChenLiao2014}, \cite{chen2014sieve} in the context of M-estimation models with series/sieve-based estimators; \cite{Newey1994} in a two-stage moment model using kernel-based estimators; \cite{ChenPouzo2015} in conditional moment models with sieve-based estimators; \cite{cattaneo2013optimal} in partitioning estimators of the conditional expectation	function and its derivatives.} Our goal is to propose a systematic approach by considering a generalization of asymptotic linearity that relaxes certain features of the standard property but still provides a useful asymptotic characterization of the estimator. This property, which is already present in many examples and we refer to as Generalized Asymptotic Linearity (GAL for short), relaxes the standard one in two dimensions: It allows for the location term to be different from the true parameter, and it allows for this centering and the influence function to vary with the sample size. Each of these relaxations attempts to capture different nuances that already exists in the many scattered examples in the literature. Our results, which we now describe, will shed more light on the role and necessity of each.

  We provide sufficient conditions for regularized estimators to satisfy GAL. Analogously to the theory of asymptotic linearity for plug-in estimators, our results rely
  on a notion of differentiability, but contrary to plug-in estimators, it relies on differentiability of each element in the sequence of regularized mappings, $(\psi_{k})_{k=1}^{\infty}$ not on differentiability of the original mapping $\psi$. 
  
  As a consequence of this approach, GAL for regularized estimators exhibits two simplified features. First, the location term is given by $\psi_{k}(P)$ which can be interpreted as a psuedo-true parameter. 
  The second simplified feature concerns the influence function and its dependence on the sample size. As in the location term, the dependence on the sample size of the influence function arises only through the dependence of the tuning parameter, $k$, on the sample size. Thus, the relevant object is a sequence of influence functions, each related to the derivative of the elements in $(\psi_{k})_{k=1}^{\infty}$. We view this quantity as the natural departure from the traditional influence function as it is the sequence of regularized mappings, $(\psi_{k})_{k=1}^{\infty}$, an not the original mapping, $\psi$, the one used for constructing the estimator.  This last feature allows us to propose a natural and systematic way of standardizing the estimator regardless of whether root-n consistency holds. To explain this, we first note that in situations where asymptotic linearity holds, the proper standardization is given by square root of the sample size divided by the standard error of the value of the influence function. Under GAL the standarization of the regularized estimators turns out to be analogous except that in this case the influence function is indexed by the tuning parameter which at the same time depends on the sample size. Whether the standardization is root-n or slower depends on the behavior of the standard error of the value influence function as we move further into the sequence of regularized mappings (i.e., as $k$ diverges). 
   
   Throughout the paper we present examples not to break new ground but to illustrate our assumptions and results, and their scope.

\bigskip

\textbf{Notation.} 
The term ``wpa1-$P$" is short for with probability approaching 1 under $P$, so for a generic sequence of IID random variables $(Z_{n})_{n}$ with $Z_{n} \sim P$, the phrase ``$Z_{n} \in A$ wpa1-$P$" formally means $P(Z_{n} \notin A) = o(1)$. 
For any random variables $(X,Y)$ we use $p_{X}$ and $p_{XY}$ to denote the pdf (w.r.t. Lebesgue) corresponding to $X$ and $X,Y$ resp. 
 For any linear normed spaces $(A,||.||_{A})$ and $(B,||.||_{B})$, let $A^{\ast}$ be the dual of $A$, and for any continuous, homogeneous of degree 1 function $f : (A,||.||_{A}) \mapsto (B,||.||_{B})$, $||f||_{\ast} = \sup_{a \in A \colon ||a||_{A} \ne 1} ||f(a)||_{B}$. For a Euclidean set $S$, we use $L^{p}(S)$ to denotes the set of $L^{p}$ functions with respect to Lebesgue. For any other measure $\mu$, we use $L^{p}(S,\mu)$ or $L^{p}(\mu)$. The norm $||.||$ denotes the Euclidean norm and when applied to matrices it corresponds to the operator norm. For any matrix $A$, let $e_{min}(A)$ denote the minimal eigenvalue.  The symbol $\precsim$ denotes less or equal up to universal constants; $\succsim$ is defined analogously.

	\section{Setup}
	\label{sec:setup} 


Let $\mathbb{Z} \subseteq \mathbb{R}^{d}$ and let  $\boldsymbol{z} \equiv (z_{1},z_{2},...) \in \mathbb{Z}^{\infty}$ denote a sequence of IID data drawn from some $P \in \mathcal{P}(\mathbb{Z}) \subset ca(\mathbb{Z})$, where $\mathcal{P}(\mathbb{Z})$ is the set of Borel probability measures over $\mathbb{Z}$ and $ca(\mathbb{Z})$ is the space of signed Borel measures of finite variation. For each $P \in \mathcal{P}(\mathbb{Z})$, let $\mathbf{P}$ be the induced probability over $\mathbb{Z}^{\infty}$. A \textbf{model} is defined as a subset of $\mathcal{P}(\mathbb{Z})$; and it will typically be denoted as $\mathcal{M}$.


\begin{remark}
	Since we only consider IID random variables, it is enough to define a model as a family of probabilities over marginal probabilities. For richer data structures, one would have to define the model as a family of probabilities over $(Z_{1},Z_{2},...)$. See Appendix \ref{app:timeseries} for a discussion about how to extend our results to general stationary models. $\triangle$
\end{remark}

A \textbf{parameter on model $\mathcal{M}$} is a mapping $\psi : \mathcal{M} \rightarrow \Theta$ with $(\Theta,||.||_{\Theta})$ being a normed space.\footnote{If the mapping does not point-identified an element of $\Theta$, i.e., $\psi$ is one-to-many, our results go through with minimal changes that account for the fact that $\psi(P)$ is a set in $\Theta$.} 

%

For the results in this paper, we need to endow  $\mathcal{M}$ with some topology. For the results in Section \ref{sec:consistent} it suffices to work with a distance, $d$, under which the empirical distribution (defined below) converges to $P$. For the results in Section \ref{sec:ALR} and beyond, however, it is convenient to have more structure on the distance function, and thus, we work with a distance of the form
\begin{align*}
      ||P-Q||_{\mathcal{S}} \equiv \sup_{f \in \mathcal{S}} \left| \int f(z) P(dz) - \int f(z) Q(dz) \right|
\end{align*} 
where $\mathcal{S}$ is some class of Borel measurable and uniformly bounded functions (bounded by one). For instance, the total variation norm can be viewed as taking $\mathcal{S}$ as the class of indicator functions over Borel sets, and its denoted directly as $||.||_{TV}$; the weak topology over $\mathcal{P}(\mathbb{Z})$ is metricized by taking $\mathcal{S}=LB$ --- the space of bounded Lipschitz functions --- and its norm is denoted directly as $||.||_{LB}$; see \cite{VdV-W1996} for a more thorough discussion.

\subsection{Regularization}
\label{sec:regularization}

 Let $\mathcal{D} \subseteq \mathcal{P}(\mathbb{Z})$ be the set of all discretely supported probability distributions. Let $P_{n} \in \mathcal{D}$ be the \textbf{empirical distribution}, where $P_{n}(A) = n^{-1} \sum_{i=1}^{n} 1\{  Z_{i} \in A   \}$ for any $A \subseteq \mathbb{Z}$. 
  As illustrated by our examples, in many situations --- especially in non-/semi-parametric models --- the parameter mapping might be either ill-defined (e.g., if $P_{n} \notin \mathcal{M}$) or ill-behaved when evaluated at the empirical distribution $P_{n}$, so it has to be regularized.

The following definition of regularization is based on the first part of the definition in \cite{BickelLi06} p. 7. To state it, we define a \textbf{tuning set} as any subset of $\mathbb{R}_{+}$ that is unbounded from above, and the \textbf{approximation error} function as $k \mapsto B_{k}(P) \equiv ||\psi_{k}(P) - \psi(P)||_{\Theta}$.

\begin{definition}\label{def:regular}
	Given a model $\mathcal{M}$, a regularization of the parameter mapping $\psi$ is a sequence $\boldsymbol{\psi} \equiv (\psi_{k})_{k \in \mathbb{K}}$ such that $\mathbb{K}$ is a tuning set and 
	\begin{enumerate}
		\item For any $k \in \mathbb{K}$, $\psi_{k} : \mathbb{D}_{\psi} \subseteq ca(\mathbb{Z}) \rightarrow \Theta$ where $\mathbb{D}_{\psi} \supseteq \mathcal{M} \cup \mathcal{D}$.
		\item For any $P \in \mathcal{M}$, $\lim_{k \rightarrow \infty} B_{k}(P) = 0 $.
	\end{enumerate} 
\end{definition}

Condition 1 ensures that $\psi_{k}(P_{n})$ and $\psi_{k}(P)$ are well-defined and that they are singletons for all $k \in \mathbb{K}$. Condition 2 ensures that, in the limit, the regularization approximates the original parameter mapping; the limit is warranted as the tuning set $\mathbb{K}$ is unbounded from above. In many applications the tuning set is given by $\mathbb{N}$ but there are applications such as kernel-based estimators, where it is more natural to use a (uncountable) subset of $\mathbb{R}_{+}$.

For each $k \in \mathbb{K}$, the implied estimator is given by $\psi_{k}(P_{n})$ which --- like the ``plug-in" estimator --- is permutation invariant. While, this restriction still encompasses a wide array of commonly used methods, it does rule out some estimation methods, notably those that rely on non-trivial sample-splitting procedures. We  briefly discuss how to extend our framework to these cases in Appendix \ref{app:split}. 

Conditions 1 and 2 are not enough to obtain ``nice" asymptotic properties of the regularized estimator such as consistency and asymptotic normality. In analogy to the standard asymptotic theory for ``plug-in" estimators, these properties will be obtained by essentially imposing different degrees of smoothness on the regularization. 

\subsection{Examples}

The following examples complement those in \cite{BickelLi06} to illustrate that the Definition \ref{def:regular} encompasses a wide array of commonly used methods.

\begin{example}[Non-Parametric IV Regression (NPIV)]\label{exa:NPIV}
		This example studies a popular regression model used in economics called the Non-parametric Instrumental Variable (IV) model, that belongs to the class of ill-posed inverse problems; see \cite{DarollesFanFlorensRenault2011,HallHorowitz2005,AiChen2003,ai2007estimation,NeweyPowell2003,Florens2003,BCK2007} among others. The model is given by
		\begin{align}\label{eqn:NPIV-id}
		E[Y - h(W) \mid X] = 0,
		\end{align}  
		where $h$ is such that $E[|h(W)|^{2}]<\infty$, $Y$ is the outcome variable, $W$ is the endogenous regressor and $X$ is the IV. We show how our method encompasses commonly used regularizations schemes such as sieves-based and penalized-based ones.
	
	For a given subspace of $L^{2}([0,1],p_{W})$, $\Theta$, the model $\mathcal{M}$ is defined as the class of probabilities over $Z = (Y,W,X) \in \mathbb{R} \times [0,1]^{2}$ with pdf with respect to Lebesgue, $p$, such that:\footnote{This restriction is mild and can be changed to accommodate discrete variables simply by requiring pdf's with respect to the counting measure.} (1) $p_{X} = p_{W} = U(0,1)$, $E[|Y|^{2}]<\infty$ and $||p_{XW}||_{L^{\infty}}<\infty$; and (2) there exists a unique $h \in \Theta$ that satisfies \ref{eqn:NPIV-id}. The restriction (1) can be relaxed and is made for simplicity so we can focus on the objects of interest that are $h$ and $P$; it implies that $L^{2}([0,1],p_{X}) = L^{2}([0,1],p_{W}) = L^{2}([0,1])$ which simplifies the derivations.\footnote{To restrict the support to $[0,1]$ is common in the literature (e.g. \cite{HallHorowitz2005}). At this level of generality, one can always re-define $h$ as $h \circ F_{W}^{-1}$ so that $p_{W} = U(0,1)$; of course this will affect the smoothness properties of $h$. The restriction $p_{X} = U(0,1)$ is really about  $p_{X}$ being known, since in that case, one can always take $F_{X}(X)$ as the instrument.} The restriction (2) is what defines an IV non-parametric model. It implies that for any $P \in \mathcal{M}$, $r_{P}(\cdot) \equiv \int y P_{YX}(dy,\cdot)$ is well-defined and belongs to the range of the operator $T_{P} : \Theta \subseteq L^{2}([0,1]) \rightarrow L^{2}([0,1])$ given by $T_{P}[h](\cdot) =  \int h(w) p_{WX}(w,\cdot) dw $ for any $h \in L^{2}([0,1])$.\footnote{Alternatively, we can define $T_{P}[h](X)= \int h(w) p(w|X)dw$ and $r_{P}(X)= \int y p(y|X)dy$. Depending on the type of the regularization one has at hand, it is more convenient to use one or the other.}	Thus, for any $P \in \mathcal{M}$, $\psi(P)$ is the (unique) solution of $r_{P} = T_{P}[h]$.

		To illustrate our method, we consider the estimation of a linear functional of $\psi(P)$  of the form $\gamma(P) \equiv \int \pi(w) \psi(P)(w) dw$ for some $\pi \in L^{2}([0,1])$, which by the Riesz representation theorem covers any linear bounded functional on $L^{2}([0,1])$.

		It is well-known that the estimation problem needs to be regularized. First, we need to regularize the ``first stage parameters" --- the operator $T_{P}$ and $r_{P}$;
		 second, given the regularization of $T_{P}$ and $r_{P}$, the inverse problem for finding $\psi(P)$ typically needs to be regularized; e.g. when $T_{P}$ is compact or when $\psi(P)$ is not a singleton.  
		
	    By setting $\mathbb{K}=\mathbb{N}$, the regularization of the ``first stage" is given by a sequence of mappings $(T_{k,P},r_{k,P})_{k \in \mathbb{N}}$ such that, for any $k \in \mathbb{N}$, $T_{k,P} : \Theta  \rightarrow L^{2}([0,1])$ and $r_{k,P} \in L^{2}([0,1])$. The ``second stage" regularization is summarized by an operator $\mathcal{R}_{k,P} : L^{2}([0,1]) \rightarrow L^{2}([0,1])$ for which  
	      \begin{align}\label{eqn:NPIV-reg}
	      \psi_{k}(P) = \mathcal{R}_{k,P} [T^{\ast}_{k,P} [r_{k,P}]],~\forall P \in lin(\mathcal{M} \cup \mathcal{D}).
	      \end{align} 
	  	We assume that the regularization structure $(T_{k,P},r_{k,P},\mathcal{R}_{k,P})_{k \in \mathbb{N}}$ is such that: (1) $\lim_{k \rightarrow \infty}||\mathcal{R}_{k,P}[T^{\ast}_{k,P}[g]] - (T^{\ast}_{P} T_{P})^{-1}T^{\ast}_{P}[g]||_{L^{2}([0,1])} = 0$ pointwise over $g \in L^{2}([0,1])$; (2) $\lim_{k \rightarrow \infty}||\mathcal{R}_{k,P}[T^{\ast}_{k,P}[r_{k,P} - r_{P}]]||_{L^{2}([0,1])} =0$. 
	  	 We relegate a more thorough discussion and particular examples of the regularization to Appendix \ref{app:T-reg-suff}. For now, it suffices to note that the first stage regularization encompasses  commonly used regularizations such as the Kernel-based (e.g., \cite{DarollesFanFlorensRenault2011}, \cite{HallHorowitz2005}) and the Series-Based (e.g., \cite{AiChen2003} and \cite{NeweyPowell2003}) regularizations, and the second stage regularization encompasses commonly used regularizations such as Tikhonov-/Penalization-based regularization (e.g., \cite{DarollesFanFlorensRenault2011} and \cite{HallHorowitz2005}) and Series-based regularization (e.g., \cite{AiChen2003} and \cite{NeweyPowell2003}). For these combinations, conditions (1)-(2) haven been verified, under primitive conditions, in the literature; e.g. see \cite{EnglEtAl1996} Ch. 3-4. 
	  	
	   It is easy to see that under conditions (1)-(2), the expression in \ref{eqn:NPIV-reg} is in fact a regularization for $\psi(P)$ with $\mathbb{D}_{\psi} \supseteq  \mathcal{M} \cup \mathcal{D}$ being a linear subspace specified in expression \ref{eqn:NPIV-Domain} in  Appendix \ref{app:T-reg-suff}. From this result, it also follows that $\{\gamma_{k}(P) \equiv \int \pi(w) \psi_{k}(P)(w) dw\}_{k \in \mathbb{N}}$ is a regularization for $\gamma(P)$ (in this case, $\Theta = \mathbb{R}$). $\triangle$
\end{example}

The next is not an example but rather a commonly used estimation technique that also fits in our framework.

\begin{example}[Regularized M-Estimators]
	Given some model $\mathcal{M}$, the parameter mapping is defined as 
	\begin{align*}
	\psi(P) = \arg\min_{\theta \in \Theta} E_{P}[\phi(Z,\theta)],~\forall P \in \mathcal{M},
	\end{align*}
	where $\Theta$ and $\phi : \mathbb{Z} \times \Theta \rightarrow \mathbb{R}_{+}$ are primitives of the problem and are such that the argmin is non-empty for any $P \in \mathcal{M}$.

	We impose the following assumptions over $(\mathcal{M},\Theta,\phi)$: $\Theta$ is a subspace of $L^{q}$ where $L^{q} \equiv L^{q}(\mathbb{Z},\mu)$ for any $q \in [1,\infty)$ and some finite measure $\mu$, and for $q=\infty$, $L^{\infty} = \mathbb{C}(\mathbb{Z},\mathbb{R})$;\footnote{The class $\mathbb{C}(\mathbb{Z},\mathbb{R})$ is the class of continuous and uniformly bounded real-valued functions on $\mathbb{Z}$.} and $\theta \mapsto E_{P}[|\phi(Z,\theta)|]$ bounded and continuous, for all $P \in \mathcal{M}$.
	
	The regularization is lifted from \cite{Pouzo2017} and is defined using $\mathbb{K} = \mathbb{N}$ by: a sequence of nested linear subspaces of $L^{q}$, $(\Theta_{k})_{k \in \mathbb{N}}$, such that $dim(\Theta_{k}) = k$ and the union is dense in $\Theta$; a vanishing real-valued sequence $(\lambda_{k})_{k \in \mathbb{N}}$ with $\lambda_{k} \in (0,1]$ and a lower-semi compact function $Pen : L^{q} \rightarrow \mathbb{R}_{+}$ such that, for each $k \in \mathbb{N}$\footnote{A lower-semi compact function is one with compact lower contour sets.}
	\begin{align*}
		\psi_{k}(P) \equiv \arg\min_{\theta \in \Theta_{k}} E_{P}[\phi(Z,\theta)] + \lambda_{k} Pen(\theta)
	\end{align*}
	is a singleton for any $P \in \mathcal{M} \cup \mathcal{D}$.
	
 It is clear that condition 1 in Definition \ref{def:regular} holds; we now show by contradiction that Condition 2 also holds. Suppose that there exists a $\epsilon>0$ such that $||\psi_{k}(P)-\psi(P)||_{\Theta} \geq \epsilon$ for all $k$ large. Let $\Pi_{k} \psi(P)$ be the projection of $\psi(P)$ onto $\Theta_{k}$; for sufficiently large $k$,  $||\Pi_{k} \psi(P)-\psi(P)||_{\Theta} \leq \epsilon$. Then, by optimality of $\psi_{k}(P)$ and some algebra, for large $k$,  $\inf_{\theta \in \Theta \colon ||\theta-\psi(P)||_{\Theta} \geq \epsilon} E_{P}[\phi(Z,\theta)] \leq E_{P}[\phi(Z,\psi(P))] - \{ E_{P}[\phi(Z,\psi(P))-\phi(Z,\Pi_{k}\psi(P))] + \lambda_{k} Pen(\Pi_{k}\psi(P))\}$. By continuity of $E_{P}[\phi(Z,\cdot)]$, $\lambda_{k} \downarrow 0$ and convergence of $\Pi_{k}\psi(P)$ to $\psi(P)$ the term in the curly bracket vanishes as $k$ diverges, leading to the contradiction $\inf_{\theta \in \Theta \colon ||\theta-\psi(P)||_{\Theta} \geq \epsilon} E_{P}[\phi(Z,\theta)] \leq E_{P}[\phi(Z,\psi(P))]$.  
 $\triangle$
\end{example}

As these examples, those below and those in \cite{BickelLi06} illustrate, the formulation of regularization is very flexible. Indeed, Definition \ref{def:regular} is sufficiently mild
that one may wonder whether any results can be obtained at the present level of generality. In what follows, we establish some useful asymptotic properties by imposing additional smoothness restrictions on the regularization.

\section{Consistency for Regularized Estimators} 
\label{sec:consistent}

%

We say a regularization is \textbf{consistent} if for each $k \in \mathbb{K}$, $\psi_{k}(P_{n}) = \psi_{k}(P) + o_{P}(1)$. Using Definition \ref{def:regular} it is straightforward to show that for consistent regularizations there exists a tuning parameter sequence $(k_{n})_{n}$ such that $\psi_{k_{n}}(P_{n}) = \psi(P) + o_{P}(1)$. This claim, however, is silent about how to choose the tuning parameter sequence and what is the corresponding rate. This feature does not seem to be a shortcoming of the claim, but rather a manifestation of the fact that the consistency requirement is very mild. In other words, it would appear that some strengthening of this requirement is needed if we want obtain stronger conclusions. In this section and the next one, our goal is to give general but ``easy-to-interpret" conditions that will enable us to obtain convergence rates and a data-driven choice of tuning parameter sequence that achieves such rates, thus providing a general guideline for establishing asymptotic concentration results for regularized estimators.

In order to do this, we take as given that the empirical distribution converges to the truth at rate, $(r_{n})_{n}$ --- a diverging positive real-valued sequence --- under a distance $d$, i.e., $d(P_{n},P) = O_{P}(r^{-1}_{n})$. Such results can typically be obtained from empirical process theory and are not the focus of this paper. Under this condition, from the classical results of \cite{Wolfowitz1957}, continuity of the regularization (with respect to $d$) presents itself as a natural condition to establish consistency of the  regularization. To define the notion of continuity formally  we say a function $f : \mathbb{R}_{+} \rightarrow \mathbb{R}_{+}$ is a modulus of continuity if $f$ is continuous, non-decreasing and such that $f(t)=0$ iff $t=0$. 

\begin{definition}[Continuous Regularization]\label{def:reg-cont}
	A regularization $\boldsymbol{\psi}$ of $\psi$ is continuous at $P \in \mathbb{D}_{\psi}$ with respect to $d$, if there exists a family of modulus of continuity $(\delta_{k})_{k \in \mathbb{R}}$ such that for any $k \in \mathbb{K}$
	 \begin{align}
		|| \psi_{k}(P') - \psi_{k}(P) ||_{\Theta} \leq \delta_{k}(d(P',P))
	\end{align}
	for any $P' \in \mathbb{D}_{\psi}$. 
\end{definition}

The definition is equivalent to the standard ``$\delta/\epsilon$"-definition of continuity because the modulus of continuity of $\psi_{k}$, $\delta_{k}$, can converge to $0$ arbitrary slowly. Admittedly, this condition is not necessary to obtain our results as it requires continuity to hold for \emph{any} deviation of $P$, but the notion of consistent regularization requires continuity only along those deviations taken by $P_{n}$ for all $n$ large enough.  However, we still view this condition as a reasonable ``initial step" to obtain consistency results. Moreover, the definition does not impose any \emph{uniform} bounds on $\delta_{k}$ across different $k \in \mathbb{K}$. While such restriction would simplify the proofs considerably, it is too strong for many applications. Recall that the regularization is introduced precisely due to the poor behavior of $\psi$ at $P$.

As the lemma \ref{lem:unif-psik-Pn-P} in Appendix \ref{app:consistent} shows, when the regularization is continuous, the ``sampling error" term, $||\psi_{k}(P_{n}) - \psi_{k}(P)||_{\Theta}$ is of order $\delta_{k}(r^{-1}_{n})$ in probability. From this result and the fact that $\limsup_{n \rightarrow \infty} \delta_{k}(r^{-1}_{n}) = 0$ for every $k \in \mathbb{K}$, a simple diagonalization argument establishes existence of a tuning parameter sequence, $(k_{n})_{n}$ for which  $(\psi_{k_{n}}(P_{n}))_{n \in \mathbb{N}}$ is consistent. The next theorem formalizes this claim.

\begin{theorem}[Consistency of Regularized Estimators]\label{thm:consistent}
	Suppose a regularization, $\boldsymbol{\psi}$, is continuous (at $P$) with respect to $d$ such that $d(P_{n},P) = O_{P}(r^{-1}_{n})$. Then there exists a $(k_{n})_{n \in \mathbb{N}}$ in $\mathbb{K}$ such that 
	\begin{align*}
	B_{k_{n}}(P) = o(1)~and~\delta_{k_{n}}(d(P_{n},P)) = o_{P}(1),
	\end{align*}
	and
	\begin{align*}
	||\psi_{k_{n}}(P_{n}) - \psi(P) ||_{\Theta} = o_{P}(1).
	\end{align*}
\end{theorem}

\begin{proof}
	See Appendix \ref{app:consistent}.
\end{proof}

%

Although constructive, the consistency result outlined above suffers from the
potential problem that the associated tuning parameter sequence makes no attempt to ensure that the magnitude of the estimation error is in some sense minimal. Indeed, because the tuning parameter sequence has been designed to satisfy only the minimal requirement that $B_{k_{n}}(P) + \delta_{k_{n}}(d(P_{n},P)) = o_{P}(1)$; the only property that
can be claimed on the part of the regularized estimator is consistency. The goal of the next section is to improve on this as well as to construct a data-driven choice of tuning parameter.

\section{Data-driven Choice of Tuning Parameter}
\label{sec:rate-choice}


 Let $\mathbb{K}_{n} $ be the (user-specified) set over which the tuning parameter is chosen; it is assumed to be a finite subset of $\mathbb{K}$.\footnote{In Appendix \ref{app:rate-choice-Gn-gral} we extend the main theorem of this section to the case where $\mathbb{K}_{n}$ is any closed set of $\mathbb{K}$, not necessarily finite.} Also, in this section, we strengthen the consistency of $P_{n}$ to  $d(P_{n},P) = o_{P}(r^{-1}_{n})$ (remark \ref{rem:rate-rn} below discusses the reason behind this choice). 


By Theorem \ref{thm:consistent} and the triangle inequality it is easy to see that the distance between the regularized estimator, $\psi_{k_{n}}(P_{n})$, and the true parameter is bounded by the sum of two terms: the ``sampling error", $\delta_{k_{n}}(r^{-1}_{n})$ and the ``approximation error", $B_{k_{n}}(P)$, which generalizes the well-known ``noise-bias" trade-off present in many applications. Thus, this observation suggests the following criterion to construct tuning sequence $(k_{n})_{n}$ that yields a consistent estimators: 
\begin{align*}
\arg\min_{k \in \mathbb{K}_{n}} \{ \delta_{k}(r^{-1}_{n}) + B_{k}(P) \},
\end{align*}
which minimizes the trade-off between the approximation and the sampling errors. This choice represents commonly used heuristics and it is a good prescription to obtain approximately optimal rate of convergences (\cite{BirgeMassart1998}). However, often times it is unfeasible, since it relies on knowledge of the approximation error, which is typically unknown because it depends on features of the unknown $P$. 

It is thus desirable to construct a choice of tuning parameter that sidestep this issue while still providing similar rates of convergence.  We propose an adaptation of the Lepski method that provides a data-driven choice that satisfies these properties, without requiring knowledge of $B_{k}(P)$. Due to the nature of the Lepski method, in order to establish the desired results we need monotonicity of the sampling and approximation errors as functions of the tuning parameter. Since these functions may not be monotonic, we replace them by monotonic majorants. Formally, let $k \mapsto \bar{B}_{k}(P)$ be a non-increasing function from $\mathbb{R}_{+}$ to itself such that $\bar{B}_{k}(P) \geq ||\psi_{k}(P) - \psi(P)||_{\Theta}$ for all $k \geq 0$, $\lim_{k \rightarrow \infty} \bar{B}_{k}(P) = 0$, and, for each $n \in \mathbb{N}$, let $ k \mapsto \bar{\delta}_{k}(r^{-1}_{n})$ be a non-decreasing function from $\mathbb{R}_{+}$ to itself such that $\bar{\delta}_{k}(r^{-1}_{n}) \geq \delta_{k}(r^{-1}_{n})$.

For each $n \in \mathbb{N}$, given $\mathbb{K}_{n}$ and a non-decreasing function $\upsilon_{n} : \mathbb{K}_{n} \rightarrow \mathbb{R}_{+}$ such that $k \mapsto \upsilon_{n}(k) =  4\bar{\delta}_{k}(r^{-1}_{n})$, the Lepski choice is given by 
\begin{align*}
	\tilde{k}_{n} \equiv \min \{k \colon k  \in \mathcal{L}_{n}  \},
\end{align*}
where
\begin{align}
 \mathcal{L}_{n} \equiv 	\mathcal{L}_{n}(\upsilon_{n}) \equiv \{ k \in \mathbb{K}_{n} \colon  ||\psi_{k}(P_{n}) - \psi_{k'}(P_{n}) ||_{\Theta} \leq \upsilon_{n}(k'),~\forall k' \geq k~in~\mathbb{K}_{n}   \}.
\end{align}

The following theorem is the main result of this section.

\begin{theorem}\label{thm:rate-choice}
		Suppose 
		a regularization, $\boldsymbol{\psi}$, is continuous (at $P$) with respect to $d$ and there exists a real-valued positive diverging sequence $(r_{n})_{n \in \mathbb{N}}$ such that $d(P_{n},P) = o_{P}(r^{-1}_{n})$. 
  	 Then
	\begin{align*}
	||  \psi_{\tilde{k}_{n}}(P_{n}) - \psi(P) ||_{\Theta} = O_{P} \left(  \inf_{k \in \mathbb{K}_{n}} \{ \bar{\delta}_{k}(r^{-1}_{n}) + \bar{B}_{k}(P)   \}     \right).
	\end{align*}
\end{theorem}

\begin{proof}
	See Appendix \ref{app:proof-rate-choice}. 
\end{proof}

\begin{remark}\label{rem:rate-rn}
	The rate $(r_{n})_{n}$ is defined as $d(P_{n},P) = o_{P}(r^{-1}_{n})$, as opposed to $d(P_{n},P) = O_{P}(r^{-1}_{n})$ as in Theorem \ref{thm:consistent}. That is, $r_{n}$ diverges (arbitrary) slower than the usual rates for $P_{n}$ --- which is typically given by $\sqrt{n}$ in our context. This type of lost is common when studying choice of tuning parameters (cf. \cite{GineNickl2008} and references therein). In our setup, it stems from the following fact: Take a rate $(s_{n})_{n}$ such that $d(P_{n},P) = O_{P}(s^{-1}_{n})$. For this rate, there are unknown constants (e.g. $M$ in Lemma \ref{lem:unif-psik-Pn-P}) which will render our data-driven choice infeasible. So to avoid them it suffices to replace $s^{-1}_{n}$ by a (arbitrary) slower rate, e.g. $r^{-1}_{n} = \log (1+n)  s^{-1}_{n}$ or $r^{-1}_{n} = \log (\log (1+n))  s^{-1}_{n}$.   $\triangle$ 
\end{remark}

\begin{remark}
	The rate of convergence does not depend on the ``complexity" of the set $\mathbb{K}_{n}$. This result stems from a certain ``separability" property of the estimator: The probability statements stem from the behavior of $d(P_{n},P)$ which does not depend on $k$ nor on $\mathbb{K}_{n}$, the tuning parameter $k$ only appear through the topological properties of the regularization. $\triangle$
\end{remark}

\begin{remark}[Heuristics of the proof of Theorem \ref{thm:rate-choice}] 
		Heuristically, for any $k \in \mathbb{K}_{n}$ that is larger or equal than $\tilde{k}_{n}$ it follows that $||\psi_{\tilde{k}_{n}}(P_{n}) - \psi(P)  ||_{\Theta}$ is bounded above (up to constants) by  $\bar{\delta}_{k}(r^{-1}_{n}) + \bar{B}_{k}(P)$ with probability approaching one. Lemma \ref{lem:suff-choice-rate} in Appendix \ref{app:proof-rate-choice} formalizes this observation and shows that in order to establish the claim of the theorem it suffices to show existence of a tuning parameter in $\mathbb{K}_{n}$ that is larger or equal than $\tilde{k}_{n}$ (with probability approaching one) and minimizes (up to constants) $k \mapsto \{ \bar{\delta}_{k}(r^{-1}_{n}) + \bar{B}_{k}(P) \}$ over $\mathbb{K}_{n}$. Moreover, since $\tilde{k}_{n}$ is chosen as the minimal value in $\mathcal{L}_{n}$, to obtain the former condition it suffices to show that the tuning parameter belongs to  $\mathcal{L}_{n}$ (with high probability). 
		
		By studying ``projections" onto $\mathbb{K}_{n}$ of the tuning parameter that balances $ k \mapsto \bar{\delta}_{k}(r^{-1}_{n})$ and $k \mapsto \bar{B}_{k}(P)$ we are able to explicitly construct a sequence of tuning parameters that satisfies these conditions; it is in this part that the monotonicity properties of these mappings are used. See Lemmas \ref{lem:hn-in-Fn} and \ref{lem:rate-h(n)-bound} in Appendix \ref{app:proof-rate-choice}.	$\triangle$ 
\end{remark}

The following corollary is a direct consequence of Theorem \ref{thm:rate-choice} and its proof is omitted.

\begin{corollary}
	Suppose $k \mapsto  B_{k}(P)$ and $k \mapsto \delta_{k}(r^{-1}_{n})$ are continuous, and non-increasing and non-decreasing resp.. Then under the conditions of Theorem \ref{thm:rate-choice}, it follows  
	\begin{align*}
|| \psi_{\tilde{k}_{n}(r_{n})}(P_{n}) - \psi(P) ||_{\Theta} = O_{P} \left(  \inf_{k \in \mathbb{K}_{n}} \{ \delta_{k}(r^{-1}_{n}) + B_{k}(P)   \}     \right).
\end{align*}	
\end{corollary}

Theorem \ref{thm:rate-choice} and its corollary show that  our data-driven choice of tuning parameter achieves the same rate as the one corresponding to the ``infeasible" choice, provided the monotonicity conditions hold. Hence, this result,  Theorem \ref{thm:consistent} and Theorem \ref{thm:rate-choice} offer a general road-map for establishing consistency and convergence rates of regularized estimators based on continuity of the regularization.\footnote{Without further restriction on $\mathbb{K}_{n}$ there is no guarantee that $\inf_{k \in \mathbb{K}_{n}} \{ \delta_{k}(r^{-1}_{n}) + B_{k}(P)   \}$ is of the same magnitude as $\inf_{k \in \mathbb{K}} \{ \delta_{k}(r^{-1}_{n}) + B_{k}(P)   \}$. In Appendix \ref{app:rate-choice-Gn-gral}, Proposition \ref{pro:rate-choice-global} gives conditions on $\mathbb{K}_{n}$ that guarantee this result.}

\section{Examples}
\label{sec:exa-cont}

%
%
%
%
%

The following examples illustrate how to apply our results to existing applications and in the process establish a new result. Example \ref{exa:boot-consistent} considers the case of bootstrapping the mean of a distribution when it is known to be non-negative, and it is based on \cite{andrews2000inconsistency}. In this paper, the author showed inconsistency of the bootstrap and proposed several consistent alternatives; we take one --- the ``k-out-of-n" bootstrap (\cite{BickelFreedman1981}) --- and illustrate how our methods can be used to derive the rate of convergence of this procedure and to choose the tuning parameter $k$ that achieves this rate. To our knowledge this last result is novel.\footnote{\cite{BickelLi06} and \cite{bickel2008choice} perform a similar exercise but for a different case: estimation of largest order statistic.} Example \ref{exa:M-consistent} provides primitive conditions for establishing continuity in M-estimation problems.

\begin{example}[Bootstrap when the parameter is on the boundary]\label{exa:boot-consistent}
	Let $\mathcal{M}$ be the class of Borel probability measures over $\mathbb{R}$ with non-negative mean, unit variance and finite third moments; the non-negativity of the mean is a formalization that captures the issue of a parameter at the boundary. The object of interest is the law of an estimator of the mean, $\boldsymbol{z} \mapsto T_{n}(\boldsymbol{z},P) = \sqrt{n}( \max\{ n^{-1} \sum_{i=1}^{n} z_{i} , 0  \} - \max \{  E_{P}[Z] , 0     \}  )$. Thus, let, for each $k \in \mathbb{K} = \mathbb{N}$, $\psi_{k} : \mathcal{P}(\mathbb{R}) \rightarrow \mathcal{P}(\mathbb{R})$ be defined as
	\begin{align*}
	\psi_{k}(P)(A) \equiv \mathbf{P} \left(  \{ \boldsymbol{z} \colon T_{k}(\boldsymbol{z},P)  \in A \}  \right) ,~\forall A~Borel.
	\end{align*}
	In particular, for $P=P_{n}$, it follows that
	\begin{align*}
\psi_{k}(P_{n})(A) = \mathbf{P}_{n} \left(  \sqrt{k} \left(  \max\{ k^{-1}\sum_{i=1}^{k} Z^{\ast}_{i}   , 0   \}   - \max\{  n^{-1}\sum_{i=1}^{n} Z_{i}  , 0  \}  \right)  \in A      \right) ,~\forall A~Borel
\end{align*} 
where $(Z^{\ast}_{i})_{i=1}^{n}$ is an IID sample drawn from $P_{n}$ and $\mathbf{P}_{n}$ is the probability over $\mathbb{Z}^{\infty}$ induced by $P_{n}$.	It is easy to see that $\psi_{n}(P_{n})$ is the standard bootstrap estimator while $\psi_{k}(P_{n})$ for $k < n$ is the $k$-out-of-$n$ bootstrap estimator. \cite{andrews2000inconsistency} showed that the ``plug-in estimator", $\psi_{n}(P_{n})$, while well-defined, fails to approximate the law of $T_{n}$, $\psi_{n}(P)$, even in the limit; but he showed that for certain sequences, $(k_{n})_{n}$, $\psi_{k_{n}}(P_{n})-\psi_{n}(P)$ converge to zero as $n$ diverges. We now recast this result using the tools developed in this paper; by doing so we are able to provide a data-driven choice of the tuning parameter $k_{n}$. 

To do this, we first show that the $(\psi_{k})_{k \in \mathbb{N}}$ is continuous in the sense of Definition \ref{def:reg-cont}. Let $\Theta = \mathcal{P}(\mathbb{R})$ and let $|| \cdot ||_{\Theta} \equiv || \cdot ||_{LB}$, where  recall $LB$ is the class of real-valued Lipschitz with constant one function. This norm is one of the notions of distance typically used to establish validity of the Bootstrap. Also, let $\mathcal{W}(\cdot ,\cdot ) $ denote the Wassertein distance over $\mathcal{P}(\mathbb{Z})$, that is $\mathcal{W}(P,Q) \equiv \inf_{\zeta \in H(P,Q)}  \int |z - z'| \zeta(dz,dz')  $, where $H(P,Q)$ is the set of Borel probabilities over $\mathbb{Z}^{2}$ with marginals $P$ and $Q$. The following proposition suggests the form of the modulus of continuity $\delta_{k}$.
\begin{proposition}\label{pro:boot-cont}
	For any $ k \in \mathbb{N}$, $||\psi_{k}(P) - \psi_{k}(Q)||_{\Theta} \leq 2 \sqrt{k} \mathcal{W}(P,Q)$ for any $P$ and $Q$ in $\mathcal{M} \cup \mathcal{D}$.
\end{proposition}

\begin{proof}
	See Appendix \ref{app:exa-cont}.
\end{proof}

	The previous results suggests $\mathcal{W}$ as the natural distance over $\mathcal{P}(\mathbb{Z})$.  In addition, the result also indicates that $\delta_{k}(t) = 2 \sqrt{k} t$ for all $t \in \mathbb{R}_{+}$, which is increasing and continuous as a function of $(t,k)$. 
	
	We now apply the results in Theorem \ref{thm:rate-choice} to choose the number of draws for the k-out-n bootstrap. Theorem 1 in \cite{FournierHal2015} (their results are applied with $d=1$, $p=1$ and $q=2$) shows that $\mathcal{W}(P_{n},P) = O_{P}(n^{-1/2})$. Therefore, we take $r^{-1}_{n} = l_{n}n^{-1/2}$ where $(l_{n})_{n}$ diverges arbitrary slowly. We also take $\mathbb{K}_{n} = \{1,...,n\}$; 
	it is clear that $k \mapsto \bar{B}_{k}(P) = B_{k}(P)$ and $k \mapsto \bar{\delta}_{k}(r^{-1}_{n}) = \delta_{k}(r^{-1}_{n})$.
	Given these choices, for each $n \in \mathbb{N}$, let   $\tilde{k}_{n}$ be the choice of tuning parameter proposed above. Theorem \ref{thm:rate-choice} imply the following result.
   \begin{proposition}\label{pro:boot-choice}
     $	|| \psi_{\tilde{k}_{n}}(P_{n}) - \psi_{n}(P) ||_{LB} = O_{P} \left( \inf_{k \in \{1,...,n\}} \{ l_{n} \sqrt{k} n^{-1/2} + k^{-1/2} E_{P}[|Z|^{3}]  \}   \right)$.
   \end{proposition}

\begin{proof}
	See Appendix \ref{app:exa-cont}.
\end{proof}

   The RHS of the expression implies that the rate of convergence is given by $\sqrt{l_{n}} n^{-1/4}$. To our knowledge there is no data-driven method to choose the tuning parameter in this example. \cite{bickel2008choice} propose a similar method to ours in a different example: Inference on the extrema of an IID sample. The authors obtain polynomial rates of convergence that are slower than ours but for a stronger norm. $\triangle$ 

\end{example}

\begin{example}[Regularized M-Estimators (cont.)]\label{exa:M-consistent}
	The following proposition shows that the regularization is continuous and more importantly it provides a ``natural" choice of distance and illustrates the role of the regularization structure $\langle (\lambda_{k},\Theta_{k})_{k},Pen \rangle$ and primitives $(\Theta,\phi)$ for determining the rate of convergence of the regularized estimator. Henceforth, let $(\theta,P,k) \mapsto Q_{k}(P,\theta) \equiv E_{P}[\phi(Z,\theta)] + \lambda_{k} Pen(\theta)$. 
\begin{proposition}\label{pro:RME-cont}
	For each $k \in \mathbb{N}$ and $P \in \mathcal{M} \cup \mathcal{D}$, \begin{align*}
		||\psi_{k}(P) - \psi_{k}(P') ||_{L^{q}} \leq \Gamma^{-1}_{k}(d(P,P')),~\forall P' \in  \mathcal{M} \cup \mathcal{D},
	\end{align*}
	where for all $t > 0$
	\begin{align*}
		\Gamma_{k}(t) = \inf_{s \geq t}   \left\{   \min_{\theta \in \Theta_{k} \colon ||\theta - \psi_{k}(P)||_{L^{q}} \geq s}  \frac{ Q_{k}(P,\theta) -  Q_{k}(P,\psi_{k}(P)) }{s} \right\}
	\end{align*}
	and $d(P,P') \equiv \max_{k \in \mathbb{N}} ||P-P'||_{\mathcal{S}_{k}}$, where $\mathcal{S}_{k} \equiv  \left\{  \frac{\phi(.,\theta) - \phi(.,\psi_{k}(P))}{||\theta - \psi_{k}(P)||_{\Theta}} \colon \theta \in \Theta_{k}   \right\} $.\footnote{We define $\Gamma_{k}(0)=0$. The ``$\inf_{s \geq t}$" ensures that $\Gamma_{k}$ is non-decreasing; it can be omitted if such property is not needed. The ``$max_{k \in \mathbb{N}}$" comes from the fact that $d$ cannot depend on $k$ in the definition of continuity.}
\end{proposition}

\begin{proof}
	See Appendix \ref{app:exa-cont}
\end{proof}	

Following \cite{ShenWong1995}, the proof applies the standard arguments due to Wald --- for establishing consistency of estimators --- to ``strips" of the sieve set $\Theta_{k}$; by doing so, one improves the rates obtained from the standard Wald approach.

The proposition suggests the natural notion of distance over the space of probabilities, that is defined by the class of ``test functions" given by $\left( \frac{\phi(z,\theta) - \phi(z,\psi_{k}(P))}{||\theta - \psi_{k}(P)||_{\Theta}} \right)_{\theta \in \Theta_{k}}$.  By imposing additional conditions on $\phi$ and $\Theta_{k}$ one can embed the class $\mathcal{S}_{k}$ into well-known classes of functions for which one has a bound for the supremum of the empirical process $f \mapsto n^{-1}\sum_{i=1}^{n} f(Z_{i}) - E_{P}[f(Z)] $, and thus bounds for $d(P_{n},P)$. For instance, if $\theta \mapsto \frac{d\phi(z,\theta)}{dz}$ is Lipschitz uniformly in $z$, then by using the mean value theorem and some algebra it follows that $\mathcal{S}_{k} \subseteq LB$ for every $k$, and thus $d(P_{n},P) = O_{P}(n^{-1/2})$ (see \cite{VdV-W1996}). 

The modulus of continuity, $\Gamma^{-1}_{k}$ is non-decreasing and is continuous over $t >0$ (see the proof), and by definition $\Gamma_{k}(0) = 0$.  Its behavior is  determined by how well the criterion separates points in $\Theta_{k}$ relative to the norm $||.||_{L^{q}}$; the flatter $Q_{k}(P,\cdot)$ is around its minimizer, the larger $\Gamma^{-1}_{k}$. Importantly, even though $\Gamma_{k}(t) > 0$ for each $k$ (recall that $\psi_{k}(P)$ is assumed to be unique), as $k$ diverges, $\Gamma_{k}(t)$ may approach zero. This phenomena relates to the potential ill-posedness of the original problem, and will affect the rate of convergence of the estimator. 

To shed some more light on the behavior of $\Gamma_{k}$ and on the potential ill-posedness, consider the case where, $q=2$, $Q(P,\cdot)$ is strictly concave and smooth, and $Pen(.) = ||.||^{2}_{L^{2}}$.  Since $\psi_{k}(P)$ is a minimizer, $Q_{k}(P,\cdot)$ behaves locally as a quadratic function, in particular $\Gamma_{k}(t) \geq 0.5 (C_{k} + \lambda_{k}) t$ for some non-negative constant $C_{k}$ related to the Hessian of $Q(P,\cdot)$, and thus $\Gamma_{k}^{-1}(t) \precsim (C_{k} + \lambda_{k})^{-1}t$.  If $C_{k} \geq c>0$ then $ \Gamma_{k}^{-1}(t) \precsim t$; we deem this case to be well-posed as $||\psi_{k}(P') - \psi_{k}(P)||_{L^{q}} \precsim d(P',P)$.\footnote{This case relates to the so-called identifiable uniqueness condition (see \cite{WW1991}).} On the other hand, if $\lim\inf_{k \rightarrow \infty} C_{k} = 0$ then, while the previous bound for the modulus of continuity is not possible, the following bound $\Gamma_{k}^{-1}(t) \precsim \lambda^{-1}_{k} t$ is. This case is deemed to be ill-posed and $||\psi_{k}(P') - \psi_{k}(P)||_{L^{q}} \precsim \lambda_{k}^{-1} d(P',P)$.

Finally,  under the conditions discussed in the previous paragraph, in the ill-posed case, 
$k \mapsto \bar{\delta}_{k}(.) = \delta_{k}(.)$ if $k \mapsto \lambda_{k}$ is chosen to be non-increasing and continuous.\footnote{For the well-posed case the condition holds trivially.} Thus Theorem \ref{thm:rate-choice} delivers a choice of tuning parameter that achieves consistency and a rate of $\min_{k \in \mathbb{N}} \{  \lambda^{-1}_{k} \times r^{-1}_{n} + \inf_{l \geq k}||\psi_{l}(P) - \psi(P)||_{L^{q}}    \}$, where $(r_{n})_{n}$ is such that $\max_{k \in \mathbb{N}} ||P_{n} - P||_{\mathcal{S}_{k}} = o_{P}(r^{-1}_{n})$.
$\triangle$ 
\end{example}


\section{Asymptotic Representations for Regularized Estimators}
\label{sec:ALR}

The goal of this section is to provide ``easy-to-interpret" sufficient conditions to obtain a generalized asymptotic linear (GAL) representation for regularized estimators, assuming that $P_{n}$ converges to $P$ in some sense. 
Throughout this section we assume $\Theta \subseteq \mathbb{R}$ to simplify the exposition; the results can be easily be extended to vector-valued parameters.\footnote{In other cases where the parameter of interest is infinite-dimensional GAL is too weak and a stronger notion is needed; we refer the reader to a previous version of this paper \cite{jansson2017general} for this case.}

We say a regularization, $\boldsymbol{\psi}$, is \textbf{asymptotically linear} if there exists $\boldsymbol{\nu} \equiv (\nu_{k})_{k\in \mathbb{K}}$ such that, for all $k \in \mathbb{K}$, $\nu_{k} \in L^{2}_{0}(P)  \equiv \{ f\in L^{2}(P)\setminus\{0 \} \colon E_{P}[f(Z)] = 0  \}$ and 
\begin{align*}
	\psi_{k}(P_{n}) = \psi_{k}(P) + n^{-1} \sum_{i=1}^{n} \nu_{k}(Z_{i}) + o_{P}(n^{-1/2}).
\end{align*}
In this case, we say $\boldsymbol{\nu}$ is the \textbf{influence of the regularization}. It is straightforward to show that for such regularizations there exists a sequence $(k_{n})_{n}$ for which the following property  is satisfied:
\begin{definition}[Generalized Asymptotic Linearity: GAL($\boldsymbol{k}$)]\label{def:ALR}
	A regularization $\boldsymbol{\psi}$ satisfies generalized asymptotic linearity for $\boldsymbol{k} \colon \mathbb{N} \rightarrow \mathbb{K}$ at $P \in \mathbb{D}_{\psi}$ with influence $\boldsymbol{\nu}$, if 
	\begin{align}
	\left|  \psi_{\boldsymbol{k}(n)}(P_{n})-\psi_{\boldsymbol{k}(n)}(P) - n^{-1} \sum_{i=1}^{n} \nu_{\boldsymbol{k}(n)}(Z_{i}) \right| = o_{P}(n^{-1/2} ||\nu_{\boldsymbol{k}(n)}||_{L^{2}(P)} ).
	\end{align}
	
\end{definition}

If a regularization satisfies GAL($\boldsymbol{k}$) then, in order to study its asymptotic behavior, it suffices to study the behavior of $ n^{-1/2} \sum_{i=1}^{n} \frac{\nu_{\boldsymbol{k}(n)}(Z_{i})}{||\nu_{\boldsymbol{k}(n)}||_{L^{2}(P)}}$.  Moreover, this property suggests a systematic way for how to scale the estimator: Using by $\sqrt{n}/||\nu_{\boldsymbol{k}(n)}||_{L^{2}(P)}$ as opposed to just $\sqrt{n}$. This insight is particularly useful in situations where root-n estimation is not possible. Thus, this property can be viewed as extending the standard asymptotic linearity one for root-n estimable parameters to a larger class of problems.

Representations akin to GAL are already present in many examples; our contribution, as we view it, is to put these insights on a common framework so they can be applied more generally and to provide primitive properties on the structure of the regularization that guarantee asymptotic linearity --- and consequently, guarantee GAL. Regarding this last point, it is well known that in cases where the ``plug-in" method is used, differentiabilty is the natural  property, the definition of asymptotic linear regularization suggests differentiability of $\psi_{k}$ for each $k \in \mathbb{K}$ as reasonable starting point. Before presenting the definition of differentiable regularization we present a classical example that illustrates GAL, the influence and the scaling when the parameter is not root-n estimable 


\begin{example}[Density Evaluation]\label{exa:pdf-eval} The parameter of interest is the density function evaluated at a point, which can be formally viewed as a mapping from the space of probability distributions to $\mathbb{R}$, given by $P \mapsto \psi(P) = p(0)$, where $p$ denotes the pdf of $P$. It is well known that this problem needs to be regularized. The standard estimator is given by $n^{-1} \sum_{i=1}^{n} \kappa_{k}(Z_{i})$ where $\kappa_{k}( \cdot) = k \kappa(k \cdot )$, $\kappa$ is a kernel (i.e., a smooth function over $\mathbb{R} \setminus \{0\}$ that ingrates to one); $1/k$ acts as the bandwidth of the kernel estimator. This estimator can be cast as $\psi_{k}(P_{n})$ where 
	\begin{align*}
	P \mapsto \psi_{k}(P) = (\kappa_{k} \star P)(0) \equiv  \int_{\mathbb{R}} \kappa_{k} \left( z \right) P(dz),~\forall k \in \mathbb{N}.
	\end{align*}

   It is well-known that the parameter, $p(0)=\psi(P)$, is not root-n estimable and thus the proposed estimator is not asymptotically linear. The following representation, however, does hold:
	\begin{align*}
	\psi_{k}(P_{n})  - \psi_{k}(P) = & n^{-1} \sum_{i=1}^{n} \{ \kappa_{k}(Z_{i}) - E_{P}[\kappa_{k}(Z)] \} 
	\end{align*}
	which can be viewed as a generalization of asymptotic linearity in which the estimator $\psi_{k}(P_{n})$ is centered at $\psi_{k}(P)=\int \kappa_{k}(z)p(z) dz $ instead of $\psi(P)=p(0)$. Moreover, by drawing an analogy with the standard approach for root-n estimable parameters (see \cite{HampelEtAl2011}, \cite{bickeletal1998efficient}, \cite{newey1990semiparametric}), for each $k$, the term in the curly brackets can be thought as an influence function. This term plays a crucial role on determining the asymptotic distribution of the estimator and on determining the proper way of standardizing it. For general regularized estimators, exact representations of this form are not always possible; however, in this section we identify a class of regularizations --- satisfying a certain differentiability notion (see Definition \ref{def:reg-G-diff}) --- that admit, asymptotically, an analogous representation, with the influence function being a function of the derivative of the regularization. 
	
	It is well-known that the scaling is given by $\sqrt{n / k_{n}}$ which is slower (for some $k_{n}$ that diverges with $n$) than the ``standard" $\sqrt{n}$. The $\sqrt{k_{n}}$ correction arises because it is the correct order of the influence function, i.e., $\sqrt{Var_{P} \left( n^{-1/2} \sum_{i=1}^{n} \{ \kappa_{k_{n}}(Z_{i}) - E_{P}[\kappa_{k_{n}}(Z)] \} \right)} = \sqrt{Var_{P}\left( \kappa_{k_{n}}(Z)  \right)} \asymp \sqrt{k_{n}}$. Our results extend this simple observation to a large class of regularizations, thereby providing a systematic way for ``standardizing" the estimator: By using $\sqrt{n}$ divided the standard deviation of the influence function, which it depends on $n$ through the tuning parameter. $\triangle$
\end{example}

%
%


We now present the definition of differentiable regularization. For this, let $\mathcal{T}_{P} \equiv \{ a \mu  \colon a\geq 0~and~ \mu \in \mathcal{D} - \{P\}\}$  
and let $\tau$ be any locally convex topology over $ca(\mathbb{Z})$ dominated by $||.||_{TV}$. \footnote{Since we are working with measures, and not probabilities, it is convenient to allow for (non-metrizable) topologies. Locally convex topology means that it is constructed in terms of a family of semi-norms; dominated by $||.||_{TV}$ means that for any semi-norm, $\rho$, $\rho(Q) =O( ||Q||_{TV})$ for all $Q \in ca(\mathbb{Z})$.} 

\begin{definition}[Differentiable Regularization: DIFF($P,\mathcal{C}$)]\label{def:reg-G-diff}
	A regularization $\boldsymbol{\psi}$ is differentiable at $P \in \mathbb{D}_{\psi}$ tangential to $\mathcal{T}_{P}$ under the class $\mathcal{C} \subseteq 2^{\mathcal{T}_{P}}$, if for any $k \in \mathbb{K}$, there exists a $D \psi_{k}(P) : \mathcal{T}_{P} \rightarrow \Theta$ $\tau$-continuous and linear such that for any $U \in \mathcal{C}$,
	\begin{align}\label{eqn:def-diff} 
	\lim_{t \downarrow 0} \sup_{Q\in U} |\eta_{k}(tQ)|/t =0,~where~Q \mapsto \eta_{k}(Q) \equiv  \psi_{k}(P+Q) - \psi_{k}(P) - D \psi_{k}(P)[Q].
	\end{align}
\end{definition}

\begin{remark}
	The functional $D \psi_{k}(P)$ acts as the gradient of $\psi_{k}$ at $P$. The set $\mathcal{T}_{P}$ is the tangent set, i.e., the set that contains all the directions of the curves at $P$ that we are considering. It turns out that to obtain an asymptotic linear representation for the regularization, it is enough to consider curves of the form $t \mapsto P + t \sqrt{n} (P_{n} - P)$. So, the choice of tangent set seems to be the most natural one. Of course, larger tangent sets will also deliver the desired results but establishing differentiability under them can be harder. 

	The definition does not impose any linear structure on $\mathcal{T}_{P}$ and $t$ is restricted to be non-negative. This feature of the definition is analogous to the idea of directional derivative in \cite{Shapiro1990} which has been shown to be sufficient for showing the validity of the Delta Method (see \cite{Shapiro1990}), and turns out to be enough to also carry out our analysis. See also \cite{FangSantos2014} and \cite{cho_white_2017} for further references, examples and discussion. $\triangle$
\end{remark}

\begin{remark}\label{rem:Diff-U}
	The class $\mathcal{C}$ determines the degree of uniformity of the limit and thus defines different notions of differentiability. It is known that common notions of differentiability can be obtained from different choices of $\mathcal{C}$; see \cite{dudley2010frechet} for a discussion. We now enumerate a few:
	\begin{enumerate}
		\item \textbf{$\tau$-Gateaux:} 
		 $\mathcal{C}$ is the class of finite subsets of $\mathcal{T}_{P}$; denoted by $\mathcal{J}_{\tau}$. 
		\item \textbf{$\tau$-Hadamard:} 
		 $\mathcal{C}$ is the class of $\tau$-compact subsets of $\mathcal{T}_{P}$; denoted by $\mathcal{H}_{\tau}$.
		\item \textbf{$\tau$-Frechet:} 
		  $\mathcal{C}$ is the class of $\tau$-bounded subsets of $\mathcal{T}_{P}$; denoted by $\mathcal{E}_{\tau}$.
	\end{enumerate}
	$\triangle$
\end{remark}

The following is the main result of this section.

\begin{theorem}\label{thm:W-ALR}
	Suppose there exists a class $\mathcal{C} \subseteq 2^{\mathcal{T}}$ such that $\boldsymbol{\psi}$ is $DIFF(P,\mathcal{C})$ and \footnote{Implicit in the differentiability condition lies the assumption that for any $Q \in \mathcal{T}_{P}$, $t \mapsto P + t Q \in \mathbb{D}_{\psi}$. For this to hold, it is sufficient that $P$ belongs to the algebraic interior of $\mathcal{M}$ relative to $\mathcal{T}_{P}$. 
		However, by inspection of the proof of the Theorem, it can be seen that this assumption is not really needed since we only consider curves of the form $t \mapsto P + t_{n} a_{n}(P_{n}-P)$ where $(t_{n},a_{n})$ are such that the curve equals $P_{n}$ which is in $\mathbb{D}_{\psi}$.}
	\begin{enumerate}
		\item[~] For any $\epsilon>0$, there exists a $U \in \mathcal{C}$ and a $N$ such that $\mathbf{P} \left( \sqrt{n} (P_{n} - P)   \in U  \right) \geq 1 - \epsilon$ for all $n \geq N$.
	\end{enumerate} 
	Then, there exists a $\boldsymbol{k} \colon \mathbb{N} \rightarrow \mathbb{K}$ for which $\boldsymbol{\psi}$ satisfies $GAL(\boldsymbol{k})$ and $\lim_{n \rightarrow \infty} \mathbf{k}(n) = \infty$.
\end{theorem}


\begin{proof}
	See Appendix \ref{app:ALR}.
\end{proof}

It is easy to check that the influence of the regularization implied by the theorem is given by the sequence of $L^{2}_{0}(P)$ mappings, $(\varphi_{k}(P))_{k \in \mathbb{N}}$ where
\begin{align*}
z \mapsto \varphi_{k}(P)(z) \equiv D \psi_{k}(P)[\delta_{z}-P].
\end{align*}

While the theorem shows existence of a sequence of tuning parameters for which generalized asymptotic linearity holds, it is silent about how to construct such sequence; we discuss this in Section \ref{sec:discussion}.

\begin{remark}[Heuristics of the Proof]
The proof is straightforward and is comprised of two steps. First, it is shown that  $\boldsymbol{\psi}$ satisfies $GAL(\mathbf{k})$ for any fixed $k$, i.e., $\mathbf{k}(n) = k$. Showing this result is analogous to showing that the regularization is asymptotic linear in the sense defined above, thus, it suffices to show that the reminder of the linear approximation is asymptotically negligible for each fixed $k$, i.e., 
\begin{align}\label{eqn:reminder-diff}
\eta_{k}(P_{n}-P) = o_{P}(n^{-1/2}).
\end{align} 
This is a standard condition for ``plug-in" estimators (e.g. \cite{VdV2000}), and the restriction over the class $\mathcal{C}$ and the definition of differentiability imply it. In some cases, however, it might be straightforward to verify condition (\ref{eqn:reminder-diff}) directly or by other means. 

Second, a diagonalization argument is used to show existence of a diverging sequence. $\triangle$
\end{remark}

\begin{remark}
	A common way of using Theorem \ref{thm:W-ALR} is by finding a class $\mathcal{S}$ that is P-Donsker, which implies that $(\sqrt{n}(P_{n} - P))_{n \in \mathbb{N}}$ is $||.||_{\mathcal{S}}$-compact (see Lemma \ref{lem:Skohorod} in Appendix \ref{app:ALR}), and ensuring $||.||_{\mathcal{S}}$-Hadamard differentiability; e.g. \cite{VdV2000} Ch. 20. \cite{dudley2010frechet} proposes an alternative way of using this result by showing that $\sqrt{n}(P_{n} - P)$ belongs, with high probability, to bounded p-variation sets, so the relevant notion of differentiability is Frechet differentiability (under the p-variation norm). $\triangle$ 
\end{remark}

\subsection{Examples}
\label{sec:WALR-examples}

We now present a series of examples. The goal here is not to break new ground but to use classical examples to illustrate the conditions and components of Theorem \ref{thm:W-ALR}. 

	The following example is the celebrated integrated square density, for which \cite{BickelRitov1990} showed that even though the efficiency bound is finite, no estimator converges at root-n rate; thereby illustrating that in some circumstances studying the local shape of $\psi$ can be quite misleading. Our approach does not suffer from this criticism since it directly captures the (local) behavior of the estimator at hand. It is also general enough to encompass many of the proposed estimators in the literature, including ``leave-one-out" types.

\begin{example}[Integrated Square Density]\label{exa:ipdf2-diff} \label{exa:ipdf2-reg}
		The parameter of interest in this case is given by 
		\begin{align*}
		P \mapsto \psi(P) = \int p(x)^{2} dx.
		\end{align*}
		The model is defined as the class of probability measures, $P$, with Lebesgue density, $p$, such that $p \in L^{\infty}(\mathbb{R})$ and 	
		\begin{align}\label{eqn:ipdf2-smooth}
		|p(x+t) - p(x) | \leq C(x) |t|^{\varrho},~\forall t,x \in \mathbb{R},
		\end{align} 
		with $C \in L^{2}(\mathbb{R})$ and $\varrho \in (0,0.5)$. This restriction is rather mild and is similar to those used in the literature, e.g. \cite{BickelRitov88}, \cite{HallMarron87} and \cite{powell1989semiparametric}. The mapping $\psi$ is well-defined over the model, but not when evaluated at the empirical probability distribution, $P_{n}$, since $P_{n}$ does not have a density; it thus needs to be regularized. 
		We consider a class of regularizations given by 
		\begin{align}\label{eqn:ipdf2-defn}
		P \mapsto \psi_{k}(P) = \int (\kappa_{k} \star P)(x) P(dx),~\forall k \in \mathbb{K},
		\end{align}	
		where $\kappa$ is a kernel such that 
		$\int |\kappa(u)| |u|^{\varrho} du < \infty$, and $t \mapsto \kappa_{k}(t) \equiv  k \kappa(k t)$. Thus, $1/k$ acts as the bandwidth for each  $k \in \mathbb{K}$ which is a tuning set in $\mathbb{R}_{++}$.
		
		Depending on the form of $\kappa$ this regularization encompasses many estimators proposed in the literature. For instance, when $\kappa = \rho + \lambda (\rho - \rho \star \rho)$  with some $\lambda \in \mathbb{R}$ and some kernel $\rho$, and, for any $k>0$, $z \mapsto \hat{p}_{k}(z) = \frac{1}{n} \sum_{i=1}^{n} \rho_{k}(Z_{i} - z)$, it follows that\footnote{Details of the claims 1-3 and 1'-3' below are shown in the Appendix \ref{app:setup}.} 
		\begin{enumerate}
			\item For $\lambda = 0$, the implied estimator is $n^{-1}\sum_{i=1}^{n} \hat{p}_{k}(Z_{i}) = n^{-2}\sum_{i,j} \rho_{k}(Z_{i} - Z_{j})$. 
			\item For $\lambda = -1$, the implied estimator is  $\int (\hat{p}_{k}(z))^{2} dz = n^{-2}\sum_{i,j} (\rho \star \rho)_{k}(Z_{i} - Z_{j})$.
			\item For $\lambda = 1$, the implied estimator is  $2n^{-1}\sum_{i=1}^{n} \hat{p}_{k}(Z_{i}) -  \int (\hat{p}_{k}(z))^{2} dz = n^{-2}\sum_{i,j} (2\rho_{k} - (\rho \star \rho)_{k})(Z_{i} - Z_{j})$. 
		\end{enumerate}
		The  first two estimators are standard; the third estimator is inspired by the one considered in \cite{NeweyHsiehRobbins2004}, wherein $\kappa$ is a twicing kernel.
		Moreover, the formalization in display (\ref{eqn:ipdf2-defn}) captures commonly used ``leave-one-out" estimators by simply imposing $\kappa(0)=0$. For instance, the ``leave-one-out" versions of the estimators 1-3 are given by
		\begin{enumerate}
			\item[1'.] For $\lambda = 0$, the implied estimator is $n^{-2}\sum_{i \ne j} \rho_{k}(Z_{i} - Z_{j})$. 
			\item[2'.] For $\lambda = -1$, the implied estimator is  $n^{-2}\sum_{i \ne j} (\rho \star \rho)_{k} (Z_{i} - Z_{j})$.
			\item[3'.] For $\lambda = 1$, the implied estimator is  $n^{-2}\sum_{i \ne j} (2\rho_{k} - (\rho \star \rho)_{k})(Z_{i} - Z_{j})$ 
			
			
		\end{enumerate}
		These estimators are essentially the ones considered by \cite{GineNickl2008} and \cite{HallMarron87} (see also \cite{POWELL1996} and references therein); 
		the estimator 3' is also a somewhat simplified version of the one considered in \cite{BickelRitov88}.

%
%
%

	We now show that Definition \ref{def:reg-G-diff} is satisfied by our class of regularizations and also establish a rate for the remainder term, $\eta_{k}(P_{n} - P)$ which is used to verify for which sequence of tuning parameter condition \ref{eqn:reminder-diff} holds.
	

	\begin{proposition}\label{pro:ipdf2-diff-1}
	For any $P \in \mathcal{M}$, the regularization defined in expression (\ref{eqn:ipdf2-defn}) is $DIFF(P,\mathcal{E}_{||.||_{LB}})$. For each $k \in \mathbb{K}$,
	\begin{align*}
	Q \mapsto D\psi_{k}(P)[Q] = 2\int (\kappa_{k} \star P)(z) Q(dz), 
	\end{align*}
	and  $Q \mapsto \eta_{k}(Q) = \int (\kappa_{k} \star Q)(z) Q(dz)$ is such that there exists a $L_{k} < \infty $ such that $|\eta_{k}(Q)| \leq L_{k} ||Q||^{2}_{LB}$ for all $Q \in ca(\mathbb{Z})$.
\end{proposition}

\begin{proof}
	See Appendix \ref{app:ipdf2-diff}.
\end{proof}

This proposition implies that for each $k \in \mathbb{K}$, $\psi_{k}$ is $||.||_{LB}$-Frechet differentiable, and since $LB$ is P-Donsker, the conditions in Theorem \ref{thm:W-ALR} are met. The influence is given by  $z \mapsto \varphi_{k}(P)(z) \equiv 2\{ ( \kappa_{k} \star P)(z) - E_{P}[( \kappa_{k} \star P)(Z)]\}$, and since $\sup_{k} ||\varphi_{k}(P)||_{L^{2}(P)} \leq 2 ||p||_{L^{\infty}(\mathbb{R})}   ||\kappa ||_{L^{1}(\mathbb{R})}$ (see Lemma \ref{lem:ipdf2-boundIF} in Appendix \ref{app:ipdf2-diff}), the natural scaling for GAL is $\sqrt{n}$. $\triangle$
	\end{example}

Next, we consider the NPIV example. It is not hard to see that the influence of $\boldsymbol{\gamma}$ will be given by $z \mapsto \int D \psi_{k}(P)^{\ast}[\pi](z) - E_{P}[D \psi_{k}(P)^{\ast}[\pi](Z)]$ provided $D \psi_{k}(P) : \mathcal{T}_{P}^{\ast} \rightarrow L^{2}([0,1])$ and its adjoint $D \psi_{k}(P)^{\ast} : L^{2}([0,1]) \rightarrow \mathcal{T}_{P}^{\ast}$ exists ($\mathcal{T}_{P}^{\ast}$ is the dual of $\mathcal{T}_{P}$). For sieve-based and penalization-based regularization schemes, we characterize $D \psi_{k}(P)^{\ast}$  and show how its standard deviation can be used to appropriately scale the estimator to obtain a generalized asymptotic linear representation regardless of whether the parameter is root-n estimable or not. This last result, illustrates how our method can be used to generalize the approach proposed in \cite{ChenPouzo2015} to general regularizations. As a by-product, we extend the results in \cite{AckerbergChenHahnLiao2014} and link the influence function of the sieve-based regularization to simpler, fully parametric, \emph{misspecified} GMM models.

\begin{example}[NPIV (cont.): The sieve-based Case]\label{exa:NPIV-diff-sieve}
	We study the sieve-based regularization approach, which is constructed using two basis for $L^{2}([0,1])$, $(u_{k},v_{k})_{k \in \mathbb{N}}$, and two indices $k \mapsto (J(k),L(k))$ such that
	\begin{align*}
		& (g,x) \mapsto T_{k,P}[g](x) = (u^{J(k)}(x))^{T} Q_{uu}^{-1} E_{P} \left[ u^{J(k)}(X) g(W) \right], \\
		&  x \mapsto r_{k,P}(x) = (u^{J(k)}(x))^{T} Q_{uu}^{-1} E_{P}[u^{J(k)}(X) Y],\\
		&  \mathcal{R}_{k,P} = (\Pi_{k}^{\ast}T^{\ast}_{k,P} T_{k,P} \Pi_{k} )^{-1}
	\end{align*}
	where $u^{k}(x) \equiv (u_{1}(x),...,u_{k}(x))$, $v^{k}(w) \equiv (v_{1}(w),...,v_{k}(w))$, $\Pi_{k} : L^{2}([0,1]) \rightarrow lin\{ v^{L(k)}  \} \subseteq L^{2}([0,1])$ is the projection operator, $g \mapsto \Pi_{k}[g] = (v^{L(k)})^{T}Q^{-1}_{vv} \int 
		v^{L(k)}(w) g(w) dw$, and $Q_{uu} \equiv E_{Leb}[u^{k}(X)(u^{k}(X))^{T}]$, $Q_{uv} \equiv E_{P}[u^{k}(X)(v^{k}(W))^{T}]$  and $Q_{vv} \equiv E_{Leb}[v^{k}(W)(v^{k}(W))^{T}]$.
	
	The next proposition proves differentiable of the regularization $\boldsymbol{\gamma}$ and provides the expression for the derivative.
	
\begin{proposition}\label{pro:NPIV-DIFF}
	For any $P \in \mathcal{M}$, the sieve-based regularization $\boldsymbol{\gamma}$ is DIFF$(P,\mathcal{E}_{||.||_{LB}})$. For each $k \in \mathbb{N}$, \begin{align*}
	  Q \mapsto D \gamma_{k}(P)[Q] = \int  D \psi_{k}(P)^{\ast}[\pi](z) Q(dz)
	\end{align*}where 
	\begin{align}\notag
	D \psi_{k}(P)^{\ast}[\pi](y,w,x) =& (y - \psi_{k}(P)(w)) (u^{J(k)}(x))^{T} Q^{-1}_{uu} Q_{uv} (Q^{T}_{uv}Q^{-1}_{uu}Q_{uv})^{-1} E_{Leb}[v^{L(k)}(W) \pi(W)] \\ \notag
	&+ \left\{ E_{P}[(\psi(P)(W) - \psi_{k}(P)(W))(u^{J(k)}(X))^{T}] Q^{-1}_{uu} u^{J(k)}(x) \right. \\\label{eqn:GMM-1}
	& \left. \times (v^{L(k)}(w))^{T} (Q^{T}_{uv}Q^{-1}_{uu}Q_{uv})^{-1} E_{Leb}[v^{L(k)}(W) \pi(W)] \right\}.
	\end{align}
	And, for each $k \in \mathbb{N}$, the reminder of $\gamma_{k}$, $\eta_{k}$, is such that $  |\eta_{k}(\zeta)| = o(||\zeta||_{LB})$ for any $\zeta \in \mathbb{D}_{\psi}$.\footnote{The ``o" function may depend on $k$.}
\end{proposition}

\begin{proof}
	See Appendix \ref{app:NPIV-diff}.
\end{proof}

Even though expression for $D \psi_{k}(P)^{\ast}[\pi]$ may look cumbersome, it has an intuitive interpretation: It is identical to the influence function of the parameter $\int \theta^{T} v^{L(k)}(w)\pi(w) dw$ where $\theta$ is the estimand of a \emph{misspecified linear GMM model} where the ``endogenous variables" are $v^{L(k)}(W)$ and the ``instrumental variables" are $u^{J(k)}(X)$; cf. \cite{HallInoue2003}. The first term in the RHS of expression \ref{eqn:GMM-1} also has an intuitive interpretation: It is the influence function of the parameter $\int \theta^{T} v^{L(k)}(w)\pi(w) dw$ but in \emph{well-specified linear GMM model}.

The proposition implies that for the ``fix-$k$" case, expression \ref{eqn:GMM-1} is the proper influence function to be considered. However, one can ask whether as $k$ diverges, the second term (the one in curly brackets) in RHS of expression \ref{eqn:GMM-1} can be ignored. To shed light on this matter, it is convenient to use operator notation for expression \ref{eqn:GMM-1}: 
\begin{align}\notag
D\psi^{\ast}_{k}(P)[\pi](y,w,x) 	= & T_{k,P}\mathcal{R}_{k,P}\Pi_{k}[\pi](x) \times (y - \psi_{k}(P)(w))\\ \label{eqn:GMM-1-op}
& +   \mathcal{R}_{k,P }\Pi_{k}[\pi](w) \times T_{k,P}[\psi(P) - \psi_{k}(P)](x)
\end{align}
(we derive this equality in expression \ref{eqn:proof-NPIV-Dpsi*} in Appendix \ref{app:NPIV-diff}). The term $T_{k,P}[\psi(P) - \psi_{k}(P)]$  is multiplied by  $\mathcal{R}_{k,P}\Pi_{k} [\pi]$, which is different to $T_{k,P}\mathcal{R}_{k,P}\Pi_{k}[\pi]$ --- the factor multiplying  $(y- \psi_{k}(P)(w))$. If $\pi \in Range (T_{P})$ both multiplying factors converge to bounded quantities as $k$ diverges.
 Thus, since $T_{k,P}[\psi(P) - \psi_{k}(P)]$ vanishes, the first summand in the RHS of expression \ref{eqn:GMM-1} ``asymptotically dominates" the second one. This is framework considered in \cite{AckerbergChenHahnLiao2014}. However, if $\pi \notin Range (T_{P})$ --- and thus  $\gamma(P)$ is not root-estimable (see \cite{SeveriniTripathi2012}) --- the situation is more subtle and without additional assumptions it is not clear which term in expression \ref{eqn:GMM-1} dominates. The reason is that the aforementioned multiplying factors will no longer converge to a bounded quantity, and moreover, the rate of growth of $T_{k,P}\mathcal{R}_{k,P}\Pi_{k}[\pi]$ can can be dominated by the rate of  $\mathcal{R}_{k,P}\Pi_{k} [\pi]$.
 
 For this last case of $\pi \notin Range (T_{P})$, the results closest to ours are those in \cite{ChenPouzo2015} wherein the influence function for slower than root-n sieve estimators is derived. Their expression for the influence function is simpler than ours, but this arises from a different set of assumptions and, more importantly, a different approach that directly focus on expressions for ``diverging $k$". $\triangle$

\end{example}

\begin{example}[NPIV (cont.): The Penalization-based Case]\label{exa:NPIV-diff-penal}
	We study the penalization-based regularization case given by
	\begin{align*}
		& (x,g) \mapsto T_{k,P}[g](x) \equiv \int \kappa_{k}(x'-x) \int g(w) P(dw,dx') \\
		& x \mapsto r_{k,P}(x) \equiv \int \kappa_{k}(x'-x) \int y P(dy,dx')\\
		& \mathcal{R}_{k,P} = (T_{k,P}^{\ast} T_{k,P} + \lambda_{k} I )^{-1}
	\end{align*}
	where $\kappa_{k}(\cdot) = k \kappa(k \cdot )$ and $\kappa$ is a smooth, symmetric around 0 pdf.
	
	As opposed to the previous case,  there is no obvious link to a ``simpler" problem like GMM and thus it is not obvious a-priori what the influence function would be and what the proper scaling should be when $\gamma(P)$ is not root-n estimable. Theorem \ref{thm:W-ALR} suggests $D \psi_{k}^{\ast}(P)[\pi]$ and $\sqrt{n/Var_{P}( D \psi_{k}^{\ast}(P)[\pi])}$ as the influence function and scaling factor resp.; the next proposition characterizes it.
	
	\begin{proposition}\label{pro:NPIV-diff-penalty}
		For any $P \in \mathcal{M}$, the Penalization-based regularization  $\boldsymbol{\gamma}$ is DIFF$(P,\mathcal{E}_{||.||_{LB}})$. For each $k \in \mathbb{N}$, $D \gamma_{k}(P)[\zeta] = \int  D \psi_{k}(P)^{\ast}[\pi](z) \zeta(dz)$, where 
		\begin{align}\label{eqn:NPIV-diff-pen-Dpsi}
				 D \psi_{k}^{\ast}(P)[\pi](y,w,x) = & \mathcal{K}_{k}^{2}T_{P}(T_{P}^{\ast}\mathcal{K}_{k}^{2}T_{P} + \lambda_{k} I )^{-1}[\pi](x) \times (y - \psi_{k}(P)(w)) \\ \notag
				 & + \lambda_{k}  (T_{P}^{\ast}\mathcal{K}_{k}^{2}T_{P} + \lambda_{k} I )^{-1}[\pi](w) \times \mathcal{K}^{2}_{k} T_{P} (T_{P}^{\ast}\mathcal{K}_{k}^{2}T_{P} + \lambda_{k} I )^{-1} [\psi_{id}(P)] (x).
		\end{align}
	where $\mathcal{K}_{k}$ is the convolution operator $g \mapsto \mathcal{K}_{k}[g] = \kappa_{k} \star g$. 	And, for each $k \in \mathbb{N}$, the reminder of $\gamma_{k}$, $\eta_{k}$, is such that $  |\eta_{k}(\zeta)| = o( ||\zeta||_{LB})$ for any $\zeta \in \mathbb{D}_{\psi}$.\footnote{The ``o" function may depend on $k$.}
	\end{proposition}
	
\begin{proof}
	See Appendix \ref{app:NPIV-diff}.
\end{proof}	

   If $\pi \in Range(T_{P})$, then the variance term converges to $|| T_{P}[v^{\ast}](X) (Y - \psi(P)(W))  ||^{2}_{L^{2}(P)} = E_{P}[(T_{P}(T_{P}^{\ast}T_{P})^{-1}[\pi](X))^{2} E_{P}[(Y - \psi(P)(W))^{2}|X] ]$ as $k$ diverges, where $v^{\ast} \equiv (T_{P}^{\ast}T_{P})^{-1}[\pi]$. The function $(y,w,x) \mapsto T_{P}[v^{\ast}](x) (y - \psi(P)(w)) $ is the influence function one would obtained by employing the methods in \cite{AiChen2007} (with identity weighting) and $v^{\ast}$ is the Riesz representer of the functional $w \mapsto \int \pi(w) g(w) dw$ using their weak norm $||T_{P}[\cdot]||_{L^{2}(P)}$.

   If $\pi \notin Range(T_{P})$, the variance diverges, and, as in the sieve case, without additional assumptions it is not clear which term dominates the variance term $Var_{P}( D \psi_{k}^{\ast}(P)[\pi])$, as $k$ diverges. This case illustrates how our results can be used to extend the results in \cite{ChenPouzo2015} for irregular sieve-based estimators to more general regularization schemes. $\triangle$	
\end{example}

\subsection{Data-driven Choice of Tuning Parameter and Undersmoothing}
\label{sec:discussion} \label{sec:undersmoothing}

Theorem \ref{thm:W-ALR} implies existence of a $n \mapsto k(n)$ such that\footnote{The display hold provided $\liminf_{n \rightarrow \infty} ||\varphi_{k(n)}(P)||_{L^{2}(P)} >0 $  For the applications we have in mind, this restriction is natural and non-binding. Our results are not designed for cases where $\lim_{k \rightarrow \infty} ||\varphi_{k}(P)||_{L^{2}(P)} = 0$; this case can be handled separately --- and rather easily --- since both the approximation error and the rate of $k \mapsto \eta_{k}(P_{n}-P)$ decrease as $k$ increases.}
\begin{align}
 \frac{\sqrt{n} (\psi_{k(n)}(P_{n}) - \psi(P)) }{||\varphi_{k(n)}(P) ||_{L^{2}(P)} } -  n^{-1/2} \sum_{i=1}^{n} \frac{ \varphi_{k(n)}(P)(Z_{i})  }{|| \varphi_{k(n)}(P) ||_{L^{2}(P)}  }  = O \left( \frac{\sqrt{n} B_{k(n)}(P) }{|| \varphi_{k(n)}(P) ||_{L^{2}(P)} } \right) + o_{P}(1). \label{eqn:ALR-estimator}
\end{align}
I.e., the asymptotic behavior of the regularized estimator --- once scaled and centered --- is characterized by a term due to the approximation error and a stochastic term. For obtaining asymptotic distributions, it is common practice to try to find sequences $(k(n))_{n}$ satisfying Theorem \ref{thm:W-ALR} for which the approximation term in expression (\ref{eqn:ALR-estimator}) vanishes. Unfortunately it is known that such sequences do not always exists  at this level of generality; e.g. \cite{BickelRitov88} and \cite{HallMarron87}. In view of this remark it is natural to seek choices of tuning parameter that make the terms in the RHS of expression (\ref{eqn:ALR-estimator}) as ``small as possible". Such choices will guarantee that GAL and the asymptotic negligibility of the approximation error both hold when possible, and otherwise, will at least yield good rates of convergence for $\psi_{k}(P_{n}) - \psi(P)$. 

The result in this section shows that the data-driven way of choosing tuning parameters described in Section \ref{sec:consistent} satisfies this property. For each $n \in \mathbb{N}$, the data-driven choice of tuning parameter is of the form $ \tilde{k}_{n} = \arg\min\{ k \colon k \in \mathcal{L}_{n}(\Lambda)   \}$ for a suitable chosen function $ k \mapsto \Lambda(k)$. In section \ref{sec:consistent}, the relevant function was $k \mapsto 4 \bar{\delta}_{k}(r^{-1}_{n})$; in this case, however, the structure of the problem is different. In particular, in addition to the reminder term $(\eta_{k})_{k}$ implied by differentiability and the scaled approximation error, there is the additional term given by $n^{-1/2} \sum_{i=1}^{n} \frac{ \varphi_{k(n)}(P)(Z_{i})  }{|| \varphi_{k(n)}(P) ||_{L^{2}(P)}  }$. The following assumption introduces the quantities to construct $(\Lambda_{k})_{k}$.  For each $n$, let $\mathbb{K}_{n}$ be the grid defined as in Section \ref{sec:consistent}.


\begin{assumption}\label{ass:Lepski-undersmooth}
	There exists a $(n,k) \mapsto \bar{\delta}_{j,k}(n)$ for $j \in \{1,2\}$ non-decreasing and a $N \in \mathbb{N}$ such that
	\begin{enumerate}
		\item[(i)] 
		$\sup_{k \in \mathbb{K}_{n}} \frac{ |\eta_{k}(P_{n} - P) |}{\bar{\delta}_{1,k}(n)} \leq 1$ wpa1-$P$.
		\item[(ii)] 
		$\frac{|\mathbb{K}_{n}|}{\sqrt{n}}  \sup_{k' \geq k~in~\mathbb{K}_{n}}  \frac{||\varphi_{k'}(P) -  \varphi_{k}(P)||_{L^{2}(P)}}{\bar{\delta}_{2,k'}(n)} \leq 1$ for all $n \geq N$.
	\end{enumerate}
\end{assumption}

As the proof of Lemma \ref{lem:hn-in-Fn-diff} in Appendix \ref{app:discussion} suggests,  the sequence that defines our tuning parameter, for each $n \in \mathbb{N}$, is given by $k \mapsto \Lambda(k) \equiv 4 (\bar{\delta}_{1,k}(n) + \bar{\delta}_{2,k}(n))$.

\begin{remark}[Discussion of Assumption \ref{ass:Lepski-undersmooth}]
	
%
%
%

Part (ii) implies that $(\bar{\delta}_{2,k}(n))_{n,k}$ acts as a growth rate for an object that, on the hand, involves the complexity of $\mathbb{K}_{n}$ --- given by $|\mathbb{K}_{n}|$ --- and on the other hand, involves the ``length" of $\mathbb{K}_{n}$ --- measured by $k \mapsto ||\varphi_{k}(P)||_{L^{2}(P)}$. In cases where $(||\varphi_{k}(P)||_{L^{2}(P)})_{k}$ is uniformly bounded, the ``length" of $\mathbb{K}_{n}$ (measured by $k \mapsto ||\varphi_{k}(P)||_{L^{2}(P)}$) is also uniformly bounded and part (ii) boils down to $\frac{|\mathbb{K}_{n}|}{\sqrt{n}} \leq \inf_{k \in \mathbb{K}_{n}} \delta_{2,k}(n))$. 
Part (i) implies that $(\bar{\delta}_{1,k}(n))_{n,k}$ also acts as the growth rate, but of a very different quantity: The reminder term of GAL, uniformly on $k \in \mathbb{K}_{n}$. 

	To shed more light on Assumption \ref{ass:Lepski-undersmooth}, suppose there exists a norm $||.||_{\mathcal{S}}$ such that
	\begin{enumerate}
		\item[C1:] There exists, for each $k \in \mathbb{K}$ a modulus of continuity $\bar{\eta}_{k} : \mathbb{R}_{+} \rightarrow \mathbb{R}_{+}$ such that $\eta_{k}(Q) = \bar{\eta}_{k}(||Q||_{\mathcal{S}})$.
		\item[C2:] There exists a real-valued positive diverging sequence $(r_{n})_{n}$ such that $||P_{n} -P||_{\mathcal{S}} = o_{P}(r^{-1}_{n})$.
	\end{enumerate} 
	Condition C1 states that $\eta_{k}$ is continuous with respect to some  norm $||.||_{\mathcal{S}}$  and C2 ensure convergence of $P_{n}$ to $P$ under this norm. These conditions are analogous to the assumptions used to show Theorem \ref{thm:rate-choice}.
	
	Under these conditions it is easy to see that Part (i) follows by choosing $\bar{\delta}_{1,k}(n) = \bar{\eta}_{k}(r^{-1}_{n})$, which acts as $\delta_{k}(r^{-1}_{n})$ in Theorem \ref{thm:rate-choice}, and, in particular, it does not depend on the grid. Part (ii), however, is not necessarily implied by this choice. If the growth rate of the reminder, $ \bar{\eta}_{k}(r^{-1}_{n})$, is small compared to $\frac{|\mathbb{K}_{n}|}{\sqrt{n}}  \sup_{k' \geq k~in~\mathbb{K}_{n}} ||\varphi_{k'}(P) -  \varphi_{k}(P)||_{L^{2}(P)}$ then part (ii) requires that $\bar{\delta}_{2,k}(n)$ to be larger than the latter, i.e., $\bar{\delta}_{2,k}(n) \geq \frac{|\mathbb{K}_{n}|}{\sqrt{n}}  \sup_{k' \geq k~in~\mathbb{K}_{n}} ||\varphi_{k'}(P) -  \varphi_{k}(P)||_{L^{2}(P)}  $. 
   
   	Below we illustrate how to verify these assumptions in the context of Example \ref{exa:ipdf2-diff}.$\triangle$
\end{remark}

     \begin{proposition}\label{pro:Lepski-rate}
	Suppose all the conditions of Theorem \ref{thm:W-ALR} hold, and Assumption \ref{ass:Lepski-undersmooth} holds. Then 
	\begin{align*}
	\left|  	\frac{\sqrt{n} ( \psi_{\tilde{k}_{n}}(P_{n}) - \psi(P) ) }{|| \varphi_{\tilde{k}_{n}}(P)||_{L^{2}(P)}} -	\frac{n^{-1/2} \sum_{i=1}^{n}  \varphi_{\tilde{k}_{n}}(P)(Z_{i})}{|| \varphi_{\tilde{k}_{n}}(P)||_{L^{2}(P)}}   \right|  = O_{P} \left( C^{2}_{n} \sqrt{n} \inf_{k \in \mathbb{K}_{n}}  \left\{  \frac{\bar{\delta}_{1,k}(n) + \bar{\delta}_{2,k}(n) +  \bar{B}_{k}(P)    }{|| \varphi_{k}(P)||_{L^{2}(P)}}   \right\}        \right),
	\end{align*}
	where $C_{n} \equiv \sup_{k',k~in~\mathbb{K}_{n}} \frac{|| \varphi_{k}(P)||_{L^{2}(P)} }{|| \varphi_{k'}(P)||_{L^{2}(P)} } $.
\end{proposition}

\begin{proof}
	See Appendix \ref{app:discussion}.
\end{proof}

 
  The rate in the proposition is --- up to $C^{2}_{n}$ factor --- the minimum value of the sum of two terms: $\sqrt{n} \frac{\bar{\delta}_{1,k}(n) + \bar{\delta}_{2,k}(n)}{||\varphi_{k}(P)||_{L^{2}(P)}}$, that controls the reminder term of GAL and another one, $\sqrt{n} \frac{B_{k}(P)}{||\varphi_{k}(P)||_{L^{2}(P)}} $, that controls the approximation error term. Therefore, if there exists a choice of tuning parameter for which both these terms are asymptotically negligible, our result implies that
\begin{align*}
\sqrt{n} \frac{  \psi_{\tilde{k}_{n}}(P_{n}) - \psi(P)}{||\varphi_{\tilde{k}_{n}}(P)||_{L^{2}(P)}  } = n^{-1/2} \sum_{i=1}^{n} \frac{\varphi_{\tilde{k}_{n}}(P)(Z_{i})}{||\varphi_{\tilde{k}_{n}}(P)||_{L^{2}(P)} } + o_{P}(1).
\end{align*}
That is, the asymptotic distribution of $	\sqrt{n} \frac{  \psi_{\tilde{k}_{n}}(P_{n}) - \psi(P)}{|| \varphi_{\tilde{k}_{n}}(P)||_{L^{2}(P)}  }$ is given that of $n^{-1/2} \sum_{i=1}^{n} \frac{ \varphi_{\tilde{k}_{n}}(P)(Z_{i})}{|| \varphi_{\tilde{k}_{n}}(P)||_{L^{2}(P)} }$.
 On the other hand, if no such sequence exists, the proposition readily implies a rate of convergence of the form $\left|  	\frac{ \psi_{\tilde{k}_{n}}(P_{n}) - \psi(P) }{|| \varphi_{\tilde{k}_{n}}(P)||_{L^{2}(P)}} \right|  = O_{P} \left( n^{-1/2} +  C^{2}_{n} \inf_{k \in \mathbb{K}_{n} }  \left\{  \frac{\bar{\delta}_{1,k}(n) + \bar{\delta}_{2,k}(n) +  \bar{B}_{k}(P)    }{|| \varphi_{k}(P)||_{L^{2}(P)}}  \right\}        \right)$. 
 
 The sequence $(C_{n})_{n}$ quantifies the discrepancy of $k \mapsto || \varphi_{k}(P)||_{L^{2}(P)}$ within the grid. In cases where Assumption \ref{ass:Lepski-undersmooth}(ii) holds for all $k'$ and $k$ in $\mathbb{K}_{n}$, it readily follows that $C_{n} = 1 + |\mathbb{K}_{n}|^{-1} \sup_{k \in \mathbb{K}_{n}} \bar{\delta}_{2,k}(n)/||\varphi_{k}(P)||_{L^{2}(P)}$.

In order to shed more light on these expressions and the assumptions, we applied our results to the estimation of the integrated square PDF (example \ref{exa:ipdf2-reg}). In this setting, \cite{GineNickl2008} already provide a data-driven method to choose the bandwidth which is akin to ours. In fact, our method can be viewed as generalization of theirs to general regularizations, and the example illustrates that, at least in their setup, we do not have to pay an extra price for the added generality.

\begin{example}[Integrated Square Density (cont.)]\label{exa:ipdf2-Lepski}

	In this example the relevant tuning parameter is the bandwidth of the kernel, so we let $k \mapsto k^{-1}$, and as the grid, $\mathbb{K}_{n}$, we use the one proposed by Gine and Nickl (\cite{GineNickl2008}), i.e., $\mathbb{K}_{n} = \{ k \colon k^{-1} \in \mathcal{H}_{n}  \}$  where 
	\begin{align*}
	\mathcal{H}_{n} = \left\{ h \in \left[ \frac{(\log n )^{4}}{n^{2}} , \frac{1}{n^{1-\delta}} \right] \colon~h_{0} = \frac{1}{n^{1-\delta}},~h_{1} = \frac{\log n}{n},~h_{2} = \frac{l^{-1}_{n}}{n},~h_{k+1} = h_{k}/a,~\forall k=2,3,...   \right\}
	\end{align*}
	where $a > 1$, $(l_{n})_{n}$ diverges to infinity slower than $\log n$ and $l^{-1}_{n} < \log n$ and $\delta>0$ is arbitrary close to 1; in particular $\delta>0$ is such that $2\varrho < \frac{1+\delta}{2(1-\delta)}$. Of importance to our analysis are the fact that $|\mathbb{K}_{n}| = O(\log n)$ and that for sufficiently large $n$, any two consecutive elements in $\mathcal{H}_{n}$ are such that $h_{k+1}/h_{k} \leq 1/a$.

	The following lemma suggests an expression for the functions $(n,k) \mapsto \bar{\delta}_{i,k}(n)$ for $i \in \{1,2\}$.
	
	\begin{lemma}\label{lem:ipdf2-suff-assumption-bound}
		For any $M>0$, there exists a $N$ such that for all $n \geq N$,
		\begin{align*}
		 \sup_{h' \leq h~in~\mathcal{H}_{n}} ||\varphi_{1/h}(P) - \varphi_{1/h'}(P)||_{L^{2}(P)} \leq  4 ||C||_{L^{2}(P)} h^{\varrho} E_{|\kappa|}[ |U|^{m+\varrho}]   .
		\end{align*}
		where the function $C$ is the one in expression \ref{eqn:ipdf2-smooth} in Example \ref{exa:ipdf2-reg}, and \begin{align*}
		\mathbf{P} \left( \sup_{h \in \mathcal{H}_{n}} \sqrt{n} |\eta_{1/h}(P_{n} - P) |\geq M \left( \frac{\kappa(0)}{\sqrt{n} h} + \frac{1}{ \sqrt{n h}  } \right)  \right) \leq |\mathcal{H}_{n}| M^{-1}.
		\end{align*}
	\end{lemma}

\begin{proof}
	See Appendix \ref{app:ipdf2-Lepski}.
\end{proof}
	
	Therefore, $\{(n,k) \mapsto \bar{\delta}_{i,k}(n)\}_{i=1,2}$ can be chosen as
	\begin{align*}
	(n,k) \mapsto \bar{\delta}_{1,k} \equiv  (\log n)^{3}   \frac{   k\kappa(0) + \sqrt{k}   }{ n},~and~(n,k) \mapsto \bar{\delta}_{2,k} \equiv   \frac{(\log n)^{3} k^{-(m+\varrho)}}{\sqrt{n}}.
	\end{align*}
	The lemma and this display illustrate the different nature of Assumptions \ref{ass:Lepski-undersmooth}(i)(ii). Part (i) bounds the reminder of the linear approximation and it increases with $k$ and decreases with $n$; this is reflected in the term $\frac{  \left( k\kappa(0)+ \sqrt{k}  \right)  }{ n }$ in the display. Part (ii) on the other hand essentially requires that the bandwidths in the grid $\mathcal{H}_{n}$ are not ``too far apart". In particular, it depends on the size of the bandwidths in $\mathcal{H}_{n}$; this is reflected in the term $k^{-\varrho} $ in the display. 
	
	It follows that $\sup_{n \in \mathbb{N}} C_{n} < \infty$ because there exists a constant $C>1$ such that $k \mapsto ||\varphi_{k}(P)||_{L^{2}(P)} \in [C^{-1},C]$ and is continuous for all $k \geq 1$. 	
	
	We verified that all assumptions of Proposition \ref{pro:Lepski-rate} hold. Moreover, Proposition \ref{pro:ipdf2-bias} in Appendix \ref{app:setup} implies that $h \mapsto \bar{B}_{1/h}(P) = O(h^{2 \varrho})$. Thus, the rate of Proposition \ref{pro:Lepski-rate} is given by\\ $\inf_{h \in \mathcal{H}_{n}} \{  (\log n)^{3}  \left( \frac{   \kappa(0)/h + 1/\sqrt{h}   }{\sqrt{n}} + h^{\varrho} \right) + \sqrt{n} h^{2 \varrho}   \} $. In fact, given our choice of $\mathcal{H}_{n}$ and $\delta$, some straightforward algebra shows that, at least for large $n$, the infimum over $\mathbb{K}_{n}$ and be replaced by the infimum over $\mathbb{R}_{+}$. Therefore,
\begin{align*}
\left|  	\frac{\sqrt{n} ( \psi_{\tilde{k}_{n}}(P_{n}) - \psi(P) ) }{|| \varphi_{\tilde{k}_{n}}(P)||_{L^{2}(P)}} -	\frac{n^{-1/2} \sum_{i=1}^{n}  \varphi_{\tilde{k}_{n}}(P)(Z_{i})}{|| \varphi_{\tilde{k}_{n}}(P)||_{L^{2}(P)}}   \right|  = \left\{ 
\begin{array}{cc}
O_{P} \left(   \left(\frac{\log n }{n} \right)^{\frac{4\varrho}{1 + 4\varrho} - 0.5}    \right) & if~\kappa(0) = 0  \\
O_{P} \left( \left(\frac{\log n }{n} \right)^{\frac{2\varrho}{1 + 2\varrho} - 0.5}      \right) &   if~\kappa(0) > 0 
\end{array}
\right. 
\end{align*}
For the case $\kappa(0) = 0$, we replicate the results by \cite{GineNickl2008}: if $\varrho > 0.25$, the reminder is negligible and root-n consistency follows, otherwise the optimal convergence rate is achieved. $\triangle$

\end{example}

\section{Conclusion}

 We propose an unifying framework to study the large sample properties of regularized estimators that extends the scope of the existing large sample theory for ``plug-in" estimators to a large class containing regularized estimators. Our results suggest that the large sample theory for regularized estimators does not constitute a large departure from the existing large sample theory for ``plug-in" estimators, in the sense that both are based on local properties of the mappings used for constructing the estimator. This last observation indicates that other large sample results developed for ``plug-in" estimators can also be extended to the more general setting of regularized estimators; e.g., estimation of the asymptotic variance of the estimator and, more generally, inference procedure like the bootstrap. We view this as a potentially worthwhile avenue for future research.

\appendix 

\stopcontents[section1]

\setstretch{1}
{\small{
\bibliographystyle{plainnat}
\bibliography{effi-biblio}
}}

\newpage 


\clearpage
\appendix
\startcontents[section2]
\printcontents[section2]{1}{1}{\section*{Table of Contents of the Appendix}}

\bigskip

\textbf{Notation:} Recall that $ca(X)$ for some set $X$ is the Banach space of all Borel measures over $X$ endowed with the the total variation norm, $||\mu||_{TV} = |\mu|(X)$ where $|.|$ is the total variation. 
For a real-valued sequence $(x_{n})_{n}$, $x_{n} \uparrow a \in \mathbb{R} \cup \{\infty\}$ means that the sequence is non-decreasing and its limit is $a$; $x_{n} \downarrow  a$ is defined analogously.

\section{Extensions of our Setup}

In this Appendix we briefly discuss how to extend our theory to general stationary models (Section \ref{app:timeseries}), we also discuss how to extend our setup to capture some sample splitting procedure commonly used in the literature (Section \ref{app:split}).

\subsection{Sample-Splitting Procedures}
\label{app:split}

Our regularized estimator --- like the plug-in one --- is defined in terms of $P_{n}$, and as such is permutation invariant. Thus, estimators that do not enjoy this property are not covered by our setup; perhaps the most notable class of estimators that falls in this category are estimators that rely on sample splitting procedures. We now argue that a slight generalization of our framework can encompass some splitting-sample procedures. 

In order to illustrate the challenges and proposed solutions that arise from these procedures, we present the problem in a simple canonical example. Suppose the parameter of interest is comprised of two quantities: a vector, denoted as $h \in \mathcal{H}$ ($\mathcal{H}$ being some subset of a Euclidean space), and a real number, denoted as $\theta \in \mathbb{R}$. The former should be treated as a so-called ``nuisance parameter" and the latter as the parameter of interest. Moreover the following ``triangular structure" holds: 
\begin{align}\label{eqn:app-SP-1}
	P \mapsto \psi(P) = (\theta(P,h(P)),h(P)),
\end{align}
where $P \mapsto h(P) \in \mathcal{H}$ is the mapping identifying the nuisance parameter and $(P,h) \mapsto \theta(P,h) \in \mathbb{R}$ is the mapping identifying the parameter of interest. The ``triangular structure" means that $h$ only depends on $P$ whereas $\theta$ depends on both $P$ and $h(P)$. An example of this structure is one where $\theta(P,h) = E_{P}[\phi(Z,h)]$ where $\phi$ is a known function that depends on the data $Z$ but also the nuisance parameter. 

Suppose the following estimator is considered. The data is divided in halves (for simplicity we assume the sample size to be even). An estimator, denoted as $\hat{\psi}_{1}$, is constructed by using the first half to construct an estimator of $h$ and using this estimator and the second half of the sample to construct the estimator for $\theta$. Another estimator, denoted as $\hat{\psi}_{2}$, is constructed by reversing the role of the first and second halves of the sample. The final estimator is simply $\hat{\psi} = 0.5 \hat{\psi}_{1} + 0.5 \hat{\psi}_{2}$. To keep the setup as simple as possible, we assume, for now, that the plug-in estimator is used (within each sub-sample)  for estimating both $\theta$ and $h$, i.e., there is no need to regularized the problem.

It is easy to see that $\hat{\psi}$ is not permutation invariant and thus does not fall in our framework. We now propose an alternative formulation of the original problem that, while seemingly redundant and even contrive at first glance, will allow us to extend our framework to this problem. This formulation entails thinking of $\psi$ as a function of \emph{two} probability distributions over $\mathbb{Z}$. Formally, $\bar{\psi} : \mathcal{M} \times \mathcal{M} \rightarrow \mathbb{R} \times \mathcal{H}$, where
\begin{align}\label{eqn:app-SP-2}
	\bar{\psi}(P_{1},P_{2}) = (\theta(P_{1},h(P_{2})),h(P_{2})).
\end{align}
At the population level this distinction is superfluous because, if the true probability is given by $P$, then $\psi(P) = \bar{\psi}(P,P)$. However, by taking $\bar{\psi}$ as the parameter mapping, the split-sample estimator can be formulated as follows. Let $P^{(1)}_{n}$ be the empirical distribution generated by the first half of the sample and $P^{(2)}_{n}$ be the empirical distribution generated by the second half of the sample. It follows that 
\begin{align*}
	\hat{\psi} = 0.5 \bar{\psi}(P^{(1)}_{n},P^{(2)}_{n}) +  0.5 \bar{\psi}(P^{(2)}_{n},P^{(1)}_{n}).
\end{align*} 
That is, the split-sample estimator can be seen as weighted average of two \emph{plug in estimators using the parameter mapping $\bar{\psi}$}. Since the estimators $(P^{(1)}_{n},P^{(2)}_{n})$ will converge to $(P,P)$ under the same conditions that ensure convergence of $P_{n}$ to $P$ (except in the former case the relevant sample size is $n/2$ not $n$), then one can establish consistency and asymptotic linearity by using the typical results for plug-in estimators, but using $\bar{\psi}$ as the original parameter mapping, and not $\psi$. 

The formulation using $\bar{\psi}$ allow us to tackle the case in which the estimation problem for $h$ or $\theta$ needs to be regularized; e.g. if $h$ is a function or a high-dimensional vector. We do this by proposing a regularization --- in the sense of Definition \ref{def:regular} --- \emph{for $\bar{\psi}$ as opposed to $\psi$}, and construct the regularized estimator as
\begin{align*}
0.5 \bar{\psi}_{k}(P^{(1)}_{n},P^{(2)}_{n}) +  0.5 \bar{\psi}_{k}(P^{(2)}_{n},P^{(1)}_{n}),~\forall k \in \mathbb{N}.
\end{align*} 
Thus our results can be applied to this case, by taking the regularization to be $(\bar{\psi}_{k})_{k}$. For instance, to establish consistency, following Theorem \ref{thm:consistent}, it suffices to verify continuity of $(\bar{\psi}_{k})_{k}$.

\bigskip

The example given by expression (\ref{eqn:app-SP-1}) has 3 features that we believe are key in order to extend our general theory for regularized estimators to encompass sample-splitting procedures. We now extrapolate these feature  from this simple canonical example to a more general setup

\begin{enumerate}
	\item The number of splits in the sample is fixed, in the example was 2, in general it can be $s \in \mathbb{N}$ but $s$ is assumed not to grow with $n$. Following the insight in expression \ref{eqn:app-SP-2}, the new parameter is given by $\bar{\psi}(P_{1},...,P_{s})$ where $P_{1},...,P_{s}$ belong to the model $\mathcal{M}$. Moreover, assuming, for simplicity, that $n = s m$ for some $m \in \mathbb{N}$, it also follows that one can construct a vector $P^{(1)}_{n},...,P^{(s)}_{n}$ of empirical probability distributions, one for each sub-sample. 
	\item The estimation procedure within each sub-sample admits a regularization as defined in our paper. That is, there exists a sequence $(\bar{\psi}_{k})_{k}$ such that $\bar{\psi}_{k}(P^{(\pi_{1})}_{n},...,P^{(\pi_{s})}_{n})$ is well-defined for each permutation $\pi_{1},...,\pi_{s}$ of $\{1,...,s\}$, and $\bar{\psi}_{k}(P,...,P)$ converges to $\bar{\psi}(P,...,P) = \psi(P)$ for each $P \in \mathcal{M}$.
	\item The final estimator is a convex combination of the estimators $\bar{\psi}_{k}(P^{(\pi_{1})}_{n},...,P^{(\pi_{s})}_{n})$. For instance, if $s=3$, then the final estimator is of the form $\sum_{i,j,k \in \{1,..,3\}} w_{i,j,k} \bar{\psi}_{k}(P^{(i)}_{n},P^{(j)}_{n},P^{(k)}_{n})$ where $(w_{i,j,k})_{i,j,k}$ are given weights. This last assumption is, in our opinion, less critical than the other two since we conjecture the convex combination can be replaces by a ``smooth" operator.
\end{enumerate}

We believe these features are general enough to encompass the sample-splitting procedures commonly used in applications 
They, however, do rule out cases where the sample splitting procedure demands number of splits that grow with the sample size. 

\subsection{Extension to General Stationary Models}
\label{app:timeseries}

We now briefly discuss how to extend our theory to general stationary models. In this case a \textbf{model} is a family of stationary probability distributions over $\mathbb{Z}^{\infty}$, i.e., a subset of $\mathcal{P}(\mathbb{Z}^{\infty})$ (the set of \emph{stationary} Borel probability distributions over $\mathbb{Z}^{\infty}$).  

Let $P$ denote the marginal distribution over $Z_{0}$ corresponding to $P \in \mathcal{P}(\mathbb{Z}^{\infty})$ (by stationarity, the time dimension is irrelevant).  For a given model $\mathcal{M}^{\infty}$, let $\mathcal{M}$ denote the set of marginal probability distribution over $Z_{0}$ corresponding to $\mathcal{M}^{\infty}$. A \textbf{parameter on model $\mathcal{M}^{\infty}$}  is a mapping from $\mathcal{M}$ to $\Theta$. That is, we restrict attention to mappings that depend only on the marginal distribution. Our theory can also be extended to cases where $\psi$ depends on the joint distribution of a \emph{finite} sub-collections of $\mathbb{Z}^{\infty}$. Allowing for the mapping to depend on the entire $\mathbf{P}$ is mathematical possible, but such object is of little relevance since it cannot be estimated from the data. 

A regularization of a parameter $\psi$ is defined analogously and the (relevant) empirical distribution is given, for each $\boldsymbol{z} \in \mathbb{Z}^{\infty}$, by $P_{n}(A) \equiv n^{-1} \sum_{i=1}^{n} 1\{ \boldsymbol{z}  \colon Z_{i}(\boldsymbol{z}) \in A  \} $ for any Borel set $A \subseteq \mathbb{Z}$.

Theorem \ref{thm:consistent} can be applied to this setup essentially without change, the difference with the i.i.d. setup lies on how to establish converges of $P_{n}$ to $P$ under $d$. Similarly, the notion of differentiability (Definition \ref{def:reg-G-diff}) can also be applied without change. The influence function will also be given by $z \mapsto D \psi_{k}(\mathbf{P})[\delta_{z} - P]$. The scaling, however, will be different, since
\begin{align*}
	E_{\mathbf{P}} \left[ \left( \sqrt{n}  D\psi_{k}(\mathbf{P})[P_{n} - P]             \right)^{2}  \right] = & 	E_{\mathbf{P}} \left[ \left( n^{-1/2} \sum_{i=1}^{n}  D\psi_{k}(\mathbf{P})[\delta_{Z_{i}} - P]             \right)^{2}  \right] \\
	= & || \varphi_{k}(\mathbf{P}) ||^{2}_{L^{2}(P)} \\
	& + 2 n^{-1} \sum_{i < j} E_{\mathbf{P}} \left[ \left( D\psi_{k}(\mathbf{P})[\delta_{Z_{i}} - P]             \right) \left(  D\psi_{k}(\mathbf{P})[\delta_{Z_{j}} - P]             \right)  \right] \\
	= & || \varphi_{k}(\mathbf{P}) ||^{2}_{L^{2}(P)}  + 2 n^{-1} \sum_{i=1}^{n-1} \sum_{j=i+1}^{n} \gamma_{j-i,k}(\mathbf{P}) \\
	\equiv  & || \varphi_{k}(\mathbf{P}) ||^{2}_{L^{2}(P)} (1+ 2 \Phi_{n,k}(\mathbf{P})) 
\end{align*}
where $\gamma_{j,k}(\mathbf{P}) \equiv E_{\mathbf{P}} \left[ \left(  D\psi_{k}(\mathbf{P})[\delta_{Z_{0}} - P]             \right) \left( D\psi_{k}(\mathbf{P})[\delta_{Z_{j}} - P]             \right)  \right]$ and 
\begin{align*}
	\Phi_{n,k}(\mathbf{P})  \equiv \sum_{i=1}^{n-1} \left(1 - \frac{i}{n}  \right) \frac{\gamma_{i,k}(\mathbf{P})}{\gamma_{i,0}(\mathbf{P})}.
\end{align*}

Hence, the natural scaling is $|| \varphi_{k}(\mathbf{P})||_{L^{2}(P)} \sqrt{(1+ 2 \Phi_{n,k}(\mathbf{P})) }$ and not $|| \varphi_{k}(\mathbf{P})||_{L^{2}(P)}$ as in the IID case. We note that our theory, a priori, does not require $\limsup_{n \rightarrow \infty} \Phi_{n,k}(\mathbf{P}) = \infty$. 

In view of the previous discussion, the relevant restriction in Theorem \ref{thm:W-ALR} is
\begin{align*}
	\sqrt{n} \frac{ \eta_{k}(P_{n}-P)}{|| \varphi_{k}(\mathbf{P}) ||_{L^{2}(P)} \sqrt{(1+ 2 \Phi_{n,k}(\mathbf{P})) }} = o_{\mathbf{P}}(1).
\end{align*} 
An analogous amendment  applies to Theorem \ref{thm:S-ALR}.

\section{Appendix for Section \ref{sec:setup}}
\label{app:setup}

The next lemma formalize verifies Claims 1-3 and 1'-3' in the text. Throughout, let $\rho_{k}(\cdot) = k \rho(k \cdot)$ for any $k \in \mathbb{K}$.

\begin{lemma}\label{lem:kernel-conv}
 For all $h>0$, $t \mapsto (\rho_{1/h} \star \rho_{1/h})(t) = (\rho \star \rho)_{1/h}(t)$.
\end{lemma}

\begin{proof}
	For all $t \in \mathbb{R}$, 
	\begin{align*}
	(\rho_{1/h} \star \rho_{1/h})(t) = \int \rho_{1/h}(t-x) \rho_{1/h}(x) dx = & h^{-2} \int \rho((t-x)/h) \rho(x/h) dx = h^{-1} \int \rho(u) \rho(t/h - u) du \\
	= &  h^{-1} (\rho  \star \rho)(t/h)
	\end{align*}
	where the last line follows from symmetry of $\rho$.
\end{proof}

\begin{lemma}
	Claims 1-3 and 1'-3' in the text hold.
\end{lemma}

\begin{proof}
	For each case 1-3 and 1'-3', we show that the $\kappa$ yields the associated estimators and that $\kappa$ is a valid choice in each case.\\
	
	(1) Follows directly from the fact that $\rho_{1/h} \star P_{n} = \hat{p}_{1/h}$. \\
	
	(2) By Lemma \ref{lem:kernel-conv}, $t \mapsto (\rho_{1/h} \star \rho_{1/h})(t) = h^{-1} (\rho  \star \rho)(t/h)$. Hence, by taking $\kappa = \rho \star \rho$ it follows that $t \mapsto \kappa_{1/h}(t) = h^{-1} (\rho \star \rho)(t/h) = 	(\rho_{1/h} \star \rho_{1/h})(t) $. Moreover, $\kappa$ is indeed a pdf, symmetric and continuously differentiable. 
	
	We now show the form of the implied estimator. We use the notation $\langle . , . \rangle$ to denote the dual inner product between $L^{\infty}(\mathbb{R})$ and $ca(\mathbb{R})$, so  
	\begin{align*}
	\int (\kappa_{1/h} \star P)(x) P(dx) =& \langle \rho_{1/h} \star \rho_{1/h} \star P , P \rangle =  \int \int \rho_{1/h}(x-y) (\rho_{1/h} \star P)(y)dy P(dx) \\
	= & \int (\rho_{1/h} \star P)(y) \int \rho_{h}(y-x)  P(dx)   dy \\
	= & \langle \rho_{1/h} \star P , \rho_{1/h} \star P \rangle_{L^{2}}
	\end{align*}
	where the second line follows by symmetry of $\rho$. Since $\rho_{1/h} \star P_{n} = \hat{p}_{h}$ the result follows.\\
	
	(3) Take $\kappa(\cdot) \equiv (-\rho \star \rho(\cdot) + 2\rho(\cdot) )$. It follows that $\int \kappa(u) du = - \int \rho \star \rho(u) du+ 2 \int \rho(u)du = 1$. Smoothness follows from smoothness of $\rho$. Finally, we note that one can write $\kappa(t)$ as $\{\rho\star \rho(t) + 2(\rho(t) - \rho \star \rho(t))\}$.
	
	By Lemma \ref{lem:kernel-conv} $t \mapsto \kappa_{1/h}(t) = h^{-1}(\rho\star \rho)(t/h) + 2h^{-1}(\rho(t/h) - \rho \star \rho(t/h)) =  (\rho_{1/h}\star \rho_{1/h})(t) + 2(\rho_{1/h}(t) - \rho_{1/h} \star \rho_{1/h}(t))$. So the expression of the estimator follows from simple algebra.\\
	
	(1') Since $P$ does not have atoms, $Z_{i} = Z_{j}$ iff $i = j$ a.s.-$\mathbf{P}$. It follows that the estimator is given by  $n^{-2} \sum_{i,j} \kappa_{1/h}(Z_{i}-Z_{j}) = n^{-1} \kappa_{1/h}(0)  + n^{-2} \sum_{i \ne j} \kappa_{1/h}(Z_{i}-Z_{j}) $ a.s.-$\mathbf{P}$ and the result follows since $\kappa(0) = 0$.\\

	(2') The expression of the estimator follows from analogous calculations to those in 1'.\\
	
	(3') By the calculations in (3)
	\begin{align*}
	\int (\hat{p}_{h}(z))^{2} dz  = & \int (\rho_{1/h} \star \rho_{1/h} \star P_{n})(x) P_{n}(dx) = n^{-2} \sum_{i \ne j} \rho_{1/h} \star \rho_{1/h}(Z_{i} - Z_{j}) + n^{-1} \rho_{1/h} \star \rho_{1/h}(0) \\
	= &n^{-2} \sum_{i \ne j} \rho_{1/h} \star \rho_{1/h}(Z_{i} - Z_{j}) + n^{-1} \int (\rho_{1/h}(z))^{2}dz
	\end{align*}
	where the last line follows by symmetry. Hence
	\begin{align*}
	\int (\hat{p}_{h}(z))^{2} dz  - 2 \int (\hat{p}_{h}(z))^{2} dz  + n^{-1} \int (\rho_{1/h}(z))^{2}dz = & - n^{-2} \sum_{i \ne j} \rho_{1/h} \star \rho_{1/h}(Z_{i} - Z_{j})\\
	= & - n^{-2} \sum_{i,j} \rho_{1/h} \star \rho_{1/h}(Z_{i} - Z_{j}) \times 1\{ Z_{i} - Z_{j} \ne 0  \}
	\end{align*}
	where the last line follows because $P$ does not have atoms, so $Z_{i} = Z_{j}$ iff $i = j$ a.s.-$\mathbf{P}$.
	
	Similarly,\begin{align*}
	2 n^{-1} \sum_{i=1}^{n} \hat{p}_{h}(Z_{i}) - 2\rho_{1/h}(0)/n = & 2 \left( n^{-2} \sum_{i,j} \rho_{1/h}(Z_{i} - Z_{j}) - \rho_{1/h}(0)/n   \right) \\
	= &  2 n^{-2} \sum_{i \ne j}   \rho_{1/h}(Z_{i} - Z_{j}) \\
	= & 2 n^{-2} \sum_{i , j}   \rho_{1/h}(Z_{i} - Z_{j}) \times 1\{ Z_{i} - Z_{j} \ne 0 \}.
	\end{align*}
\end{proof}

The following proposition provides bounds for the approximation error.

   	\begin{proposition}\label{pro:ipdf2-bias}
   		There exists a finite constant $C>0$ such that for any $k \in \mathbb{K}$ and any $P \in \mathcal{M}$,
   		\begin{align*}
   		B_{k}(P) \leq C k^{-2\varrho}  E_{|\kappa|}[ |U|^{2 \varrho}].
   		\end{align*}
   	\end{proposition}

\begin{proof}[Proof of Proposition \ref{pro:ipdf2-bias}]
		
	Since $P \in \mathcal{M}$ it admits a smooth pdf, $p$, it follows that
	\begin{align*}
	  	\psi_{k}(P)  - \psi(P)= & \int \left( \int \kappa_{k}(x -y) p(y) dy - p(x)  \right) p(x) dx  \\
	  	= &  \int  \int \kappa(u)  p(x-u/k) du - p(x) p(x) dx \\
	  	= & \int \left( \bar{p} \star p(u/k)   - \bar{p} \star p(0) \right) \kappa(u) du
	\end{align*}
	where $t \mapsto \bar{p}(t) \equiv p(-t) $.  Henceforth, let $ t \mapsto g(t) \equiv \bar{p} \star p(t)$. 
	
	Our condition (\ref{eqn:ipdf2-smooth}) implies that $p$ and $\bar{p}$ belong to the Besov space $\mathcal{B}_{2,\infty}^{\varrho}(\mathbb{R})$. Lemma 12 in \cite{gine2008uniform} implies that $g \in \mathcal{B}_{\infty,\infty}^{2\varrho}(\mathbb{R})$, in fact since $2 \varrho \notin \mathbb{N}$, $g$ is H\"{o}lder continuous with parameter $2 \varrho$. This implies and the previous display imply that $B_{k}(P) \leq C k^{-2 \varrho} \int |u|^{2 \varrho} |\kappa(u)| du$ for some universal constant $C< \infty$.

\end{proof}

\begin{remark}[Remarks about the Condition \ref{eqn:ipdf2-smooth}]
	\cite{GineNickl2008} imposes $p \in H^{\varrho}_{2}(\mathbb{R})$, whereas our restriction essentially implies that $p \in \mathcal{B}^{\varrho}_{2,\infty}(\mathbb{R})$. In that paper and in ours the smoothness coefficient $\varrho$ is less than $0.5$, i.e., we have ``low" degree of smoothness. Because of this, whether or not the kernel is a ``twicing kernel" does not matter for the control of the approximation error. For larger levels of smoothness, e.g. $\varrho > 1$, we expect the ``twicing kernel" --- or higher order kernels in general --- to yield different bounds for the approximation error. The goal of this example is to illustrate the scope of our methodology and thus we decided to stay as closed as possible to the existing literature and omit the case $\varrho > 0.5$. $\triangle$
\end{remark}

  \subsection{Some Remarks on the Regularization Structure in the NPIV Example.}
  \label{app:T-reg-suff}
  
  The general regularization structure, $(\mathcal{R}_{k,P},T_{k,P},r_{k,P})_{k \in \mathbb{N}}$, and conditions 1-2 are taken from \cite{EnglEtAl1996} Ch. 3-4. It is clear from the problem that 
    \begin{align}\label{eqn:NPIV-Domain}
    \mathbb{D}_{\psi} = \{ \mu \in ca(\mathbb{R} \times [0,1]^{2}) \colon E_{\mu}[|Y|^{2}] < \infty~and~E_{\mu}[|h(W)|^{2}] <\infty~\forall h \in L^{2}([0,1])  \}.
    \end{align}

 The next lemma presents useful properties of $\mathbb{D}_{\psi}$. The proof is straightforward and thus omitted.
 
 \begin{lemma}\label{lem:NPIV-domain}
 	(1) $\mathbb{D}_{\psi} \supseteq \mathcal{M} \cup \mathcal{D}$; (2) $\mathbb{D}_{\psi}$ is a linear subspace.
 \end{lemma}

%

  We now discuss canonical examples of regularizations methods for the first and second stage that we consider in this paper. 
  
      \bigskip
      
     \textsc{First Stage Regularization.} For any $P \in \mathbb{D}_{\psi}$ and any $k \in \mathbb{N}$, we can generically write $r_{k,P}$ as  
     \begin{align*}
     r_{k,P}(x) \equiv \int  y  \int  U_{k}(x',x)   P(dy,dx') ,~\forall x\in [0,1],
     \end{align*}
     where $U_{k} \in L^{\infty}([0,1]^{2})$ symmetric. For instance, if
   	\begin{align*}
      	(x',x) \mapsto U_{k}(x',x) =  k u(k (x-x'))
      	\end{align*}   	
  where $u$ is a symmetric around 0, smooth pdf, then $x \mapsto r_{k,P}(x) 	=  \int  y  \int  ku(k(x-x')) P(dy,dx')$, which is the the so-called kernel-based approach; e.g., for ill-posed inverse problems see \cite{HallHorowitz2005} among others.
  
  In the case one defined $r_{P}$ using conditional probabilities, i.e., $r_{P}(x) = \int y p(y|x)dy$. The kernel approach becomes
  \begin{align*}
  	x \mapsto r_{k,P}(x) 	=  \int  y  \frac{ \int  ku(k(x-x')) P(dy,dx')}{\int  ku(k(x-x')) P(dx')};
  \end{align*}
  (e.g. \cite{DarollesFanFlorensRenault2011}). Observe that $r_{P}$ is only defined for probability measures for which the pdf exists.
  
  Another approach is to directly set
  \begin{align*}
  (x',x) \mapsto U_{k}(x',x) =  (u^{k}(x))^{T} Q_{uu}^{-1} u^{k}(x'),
  \end{align*}
  where $(u_{k})_{k \in \mathbb{N}}$ is some basis function in $L^{2}([0,1])$ and $Q_{uu} \equiv E_{Leb}[(u^{k}(X)) (u^{k}(X))^{T}]$. In this case, $x \mapsto r_{k,P}(x) 	=  (u^{k}(x))^{T} Q_{uu}^{-1} E_{P}[u^{k}(X) Y]$, which is the so-called series-based approach; e.g., for ill-posed inverse problems see \cite{AiChen2003,NeweyPowell2003} among others.

  Analogously, one can define $T_{k,P}$ as
    \begin{align*}
    g \mapsto T_{k,P}[g](x) \equiv \int  g(w)  \int  U_{k}(x',x)   P(dw,dx') ,~\forall x\in [0,1],
    \end{align*}
  and the same observations above applied to this case.

  The next lemma characterizes the adjoint for any $P \in \mathcal{M}$ (i.e., $P$ as a pdf $p$). In this case, we can view the regularization as an operator acting on $T_{P}[g](x) = \int g(w) p(w,x) dw$, given by $\mathcal{U}_{k} : L^{2}([0,1]) \rightarrow L^{2}([0,1])$, where $\mathcal{U}_{k}T_{P}[g](x)\equiv \int U_{k}(x',x) \int g(w) p(w,x') dw dx'  $.
  
    \begin{lemma}
    	For any $k \in \mathbb{N}$ and any $P \in \mathcal{M}$ (in particular, it admits a pdf $p$), the adjoint of $T_{k,P}$ is $T^{\ast}_{k,P} : L^{2}([0,1])\rightarrow L^{2}([0,1])$ and is given by
    	\begin{align*}
    	f \mapsto T^{\ast}_{k,P}[f] = T^{\ast} \mathcal{U}_{k}[f].
    	\end{align*} 
    \end{lemma}
    
    \begin{proof}
    	For any $k \in \mathbb{N}$ and any $P \in \mathcal{M}$, 
    	\begin{align*}
    	\langle T_{k,P}[g],f \rangle_{L^{2}([0,1])} = &  \int \left( \mathcal{U}_{k}T_{P}[g](x) \right) f(x) dx \\
    	= & \int g(w)  \int \int U_{k}(x',x) f(x) dx  p(w,x') dx'  dw\\
    	= & 	\langle g, T^{\ast}_{P} \mathcal{U}_{k}[f]  \rangle_{L^{2}([0,1])}
    	\end{align*}
    	for any $g,f \in L^{2}([0,1])$.	
    \end{proof}
    
    If $P \notin \mathcal{M}$, in particular if it does not have a pdf (with respect to Lebesgue), the adjoint operator is different; the reason being that $T^{\ast}_{P}$ does not map onto a space of functions because $P$ does not have a pdf.  In this case, consider the operator $A_{P} : L^{2}([0,1]) \rightarrow ca([0,1])$ given by $f \mapsto A_{P}[f](B) = \int_{w \in B} \int f(x) P(dw,dx) $ for any $B \subseteq [0,1]$ Borel. Note that $|A_{P}[f](.)| \leq \int |f(x)| P(dx) < \infty$ provided that $f \in L^{2}(P)$, which is the case for any $P \in \mathbb{D}_{\psi}$. The next lemma characterizes the adjoint in this case.
    \begin{lemma}
    	For any $k \in \mathbb{N}$ and any $P \in \mathbb{D}_{\psi}$, the adjoint of $T_{k,P}$ is given by
    	\begin{align*}
    	f \mapsto T^{\ast}_{k,P}[f] = A_{P} \mathcal{U}_{k}[f].
    	\end{align*} 
    \end{lemma}

Since $U_{k} \in L^{2}([0,1]^{2})$, $T^{\ast}_{k,P}[f]([0,1])  \precsim ||f||_{L^{2}(P)} ||U_{k}||_{L^{2}([0,1]^{2})}$ which is finite for $P \in \mathbb{D}_{\psi}$. So $T^{\ast}_{k,P}[f]$ in fact maps to $ca([0,1])$. 
    
    \begin{proof}
    	For any $k \in \mathbb{N}$ and any $P \in \mathbb{D}_{\psi}$, 
    	\begin{align*}
    	\langle T_{k,P}[g],f \rangle_{L^{2}([0,1])} = &  \int \int  g(w)  \int  U_{k}(x',x)   P(dw,dx')  f(x) dx \\
    	= & \int  g(w) \left( \int U_{k}(x',x) f(x) dx \right) P(dw,dx')\\
    	= & \int  g(w) \int \mathcal{U}_{k}[f](x') P(dw,dx'),
    	\end{align*} 	
    	for any $g,f \in L^{2}([0,1])$.	
    \end{proof}

   One possibility to avoid the aforementioned technical issue with the adjoint operator is to define a regularization given by 
   \begin{align*}
   g \mapsto T_{k,P}[g](x) \equiv \int  g(w) \left\{  \int  U_{k}(x',x)V_{k}(w',w)   P(dw',dx') \right\} dw ,~\forall x\in [0,1],
   \end{align*} 
   where $U_{k} \in L^{\infty}([0,1]^{2})$ symmetric. For example, if $V_{k}(w',w) = h_{k}^{-1} v((w'-w)/h_{k}$ (and $U_{k}$ is also given by the kernel-based approach), then in this case $(x,w) \mapsto  \mathcal{W}_{k}[P](x,w) \equiv \int  U_{k}(x',x)V_{k}(w',w)   P(dw',dx') $ is a pdf over $[0,1]^{2}$ (regardless of whether $P$ has a pdf or not), and thus 
   \begin{align*}
   f \mapsto T_{k,P}^{\ast}[f](w) = \int f(x) \mathcal{W}_{k}[P](x,w)dx.
   \end{align*}  
   For instance, \cite{HallHorowitz2005} considered a method akin to this.

     In the case $T_{P}$ is defined as a conditional operator, one can consider the sieve-based approach for $U_{k}$ and $V_{k}(w',w) = (v^{k}(w))^{T} Q_{vv}^{-1} v^{k}(w')$ for some $(v_{k})_{k \in \mathbb{N}}$ basis function in $L^{2}([0,1])$. Then, in this case,
	\begin{align*}
	T_{k,P}[g](x) 	= & (u^{k}(x))^{T} Q_{uu}^{-1} E_{P}[u^{k}(X) (v^{k}(W))^{T}Q_{vv}^{-1} E_{Leb}[v^{k}(W)g(W)] ] \\
	= & (u^{k}(x))^{T} Q_{uu}^{-1} Q_{uv} Q_{vv}^{-1} E_{Leb}[v^{k}(W)g(W)] \\
	= & \int g(w) \left\{ \int (u^{k}(x))^{T} Q_{uu}^{-1} u^{k}(x') (v^{k}(w'))^{T} Q_{vv}^{-1} v^{k}(w)  P(dw',dx') \right\} dw
	\end{align*}
  where $Q_{vv} \equiv E_{Leb}[(v^{k}(W)) (v^{k}(W))^{T}]$  and $Q_{uv} \equiv E_{P}[u^{k}(X) (v^{k}(W))^{T}] $, and 
  \begin{align*}
  	f \mapsto T_{k,P}^{\ast}[f](w) = & (v^{k}(w))^{T}Q_{vv}^{-1} Q_{uv}^{T} Q_{uu}^{-1} E_{Leb}[u^{k}(X)f(X)]\\
  	  = & \int f(x) \mathcal{W}_{k}[P](x,w) dx
  \end{align*} 
  where $\mathcal{W}_{k}[P](x,w) = \int (u^{k}(x))^{T} Q_{uu}^{-1} u^{k}(x') (v^{k}(w'))^{T} Q_{vv}^{-1} v^{k}(w)  P(dw',dx') $.

   \bigskip 
   
      \textsc{Second Stage Regularization.}  For the second stage regularization, one widely used approach is the so-called Tikhonov- or Penalization-based approach, given by solving 
      \begin{align*}
      	\arg\min_{\theta \in \Theta} \{ E_{P}\left[ \left( r_{k,P}(X) - T_{k,P}[\theta](X)   \right)^{2} \right]   + \lambda_{k} ||\theta||^{2}_{L^{2}([0,1])} \}
      \end{align*}
      which is non-empty and a singleton. This specification implies that 
      \begin{align*}
      	\mathcal{R}_{k,P} = (T^{\ast}_{k,P} T_{k,P} + \lambda_{k}I)^{-1},
      \end{align*}
      which is well-known to be well-defined, i.e., 1-to-1 and bounded for any $\lambda_{k}>0$.

       Another widely used approach is the sieve-based approach that consists on setting up 
           \begin{align*}
      \arg\min_{\theta \in \Theta_{k}}  E_{P}\left[ \left( r_{k,P}(X) - T_{k,P}[\theta](X)   \right)^{2} \right]   
      \end{align*}
      and specifie the $(\Theta_{k})_{k}$ such that (1) $\cup_{k} \Theta_{k} $ is dense in $\Theta$ and $\Theta_{k}$ has dimension $k$, and (2) $\arg\min$ exists and is a singleton. For instance if $\Theta_{k}$ is convex, then a solution exists and is unique provided that $Kernel(T_{k,P}|\Theta_{k}) = \{0\}$.  In this case 
       \begin{align*}
      \mathcal{R}_{k,P} = (\Pi_{k}^{\ast} T^{\ast}_{k,P} T_{k,P} \Pi_{k})^{-1},
      \end{align*}
      where $\Pi_{k}$ is the projection onto $\Theta_{k}$.

      \bigskip
      
      \textsc{Verification of Definition \ref{def:regular}.} The next Lemma shows that given Conditions 1-2 listed in Example \ref{exa:NPIV}, $(\psi_{k}(P))_{k\in \mathbb{N}}$ (and hence $(\gamma_{k}(P))_{k\in \mathbb{N}}$)  is in fact a regularization.
\begin{lemma}
	Suppose  Conditions 1-2 listed in Example \ref{exa:NPIV} hold. Then $(\psi_{k}(P))_{k \in \mathbb{N}}$ (and hence $(\gamma_{k}(P))_{k\in \mathbb{N}}$)  is a regularization with $\mathbb{D}_{\psi}$ given in \ref{eqn:NPIV-Domain}.
\end{lemma}
       
      \begin{proof}
      	Condition 1 in Definition \ref{def:regular} is satisfied by Lemma \ref{lem:NPIV-domain}. Regarding condition 2 in Definition \ref{def:regular}, note that
      	\begin{align*}
      		||\psi_{k}(P) - \psi(P)||_{L^{2}([0,1])} \leq || \mathcal{R}_{k,P} T^{\ast}_{k,P} [r_{k,P} - r_{P}]||_{L^{2}([0,1])}  + || ( \mathcal{R}_{k,P} T^{\ast}_{k,P} - (T^{\ast}_{P}T_{P})^{-1} T^{\ast}_{P}) [r_{P}]||_{L^{2}([0,1])}, 
      	\end{align*}
      	which vanishes as $k$ diverges by our conditions 1-2. 
      \end{proof}

\section{Appendix for Section \ref{sec:consistent}}
\label{app:consistent}

The next lemma provides an useful ``diagonalization argument" that is used throughout the paper.

\begin{lemma}\label{lem:sub-seq-conv}
	Let $S = \{ k_{1} , k_{2} , ...  \}$ with $k_{i} < k_{i+1}$ for all $i \in \mathbb{N}$. Take a real-valued sequence $(x_{k,n})_{k\in S,n \in \mathbb{N}}$ such that, for each $k \in \mathbb{N}$, $\lim_{n \rightarrow \infty} |x_{k,n}| = 0$. Then, there exists a mapping $n \mapsto k(n) \in S$ such that (a) $\lim_{n \rightarrow \infty} |x_{k(n),n}| = 0$ and (b) $k(n) \uparrow \infty$. 
\end{lemma}

\begin{proof}
	By pointwise convergence of the sequence $(x_{k,n})_{n}$, for any $l \in \mathbb{N}$, there exists a $n(l) \in \mathbb{N}$ such that $|x_{k_{l},n}| \leq 1/2^{k_{l}}$ for all $n \geq n(l)$. WLOG we take $n(l+1) > n(l)$. 
	
	We now construct the mapping $n \mapsto k(n)$ as follows: For each $l \in \mathbb{N}$, let $k(n) \equiv k_{l}$ for all $n \in \{ n(l)+1,...,n(l+1)  \}$; and $k(n) = 0$ for $n \in \{0,...,n(0) \}$. Since the cutoffs $n(.)$ are increasing the set $ \{ n(l)+1,...,n(l+1)  \}$  is non-empty for each $l$. For integer $L>0$, $k(n) > k_{L}$ for all $n \geq n(L)+1$; since $(k_{l})_{l}$ diverges, (b) follows. 
	
	To show (a), for any $\epsilon>0$ take $l_{\epsilon}$ such that $1/2^{k_{l_{\epsilon}}} \leq \epsilon$. Observe that for any $n \geq n(l_{\epsilon})+1$, $|x_{k(n),n}| \leq 1/2^{l_{\epsilon}} \leq \epsilon$ by construction of $(n,k(n))$. Thus, (a) follows. 
\end{proof}

As the next lemma shows, when the regularization is continuous, the ``sampling error" term, $||\psi_{k}(P_{n}) - \psi_{k}(P)||_{\Theta}$ is of order $\delta_{k}(r^{-1}_{n})$ in probability. 

\begin{lemma}\label{lem:unif-psik-Pn-P}
	Suppose a regularization, $\boldsymbol{\psi}$, is continuous (at $P$) with respect to $d$ and there exists a real-valued positive sequence $(r_{n})_{n \in \mathbb{N}}$ such that $d(P_{n},P) = O_{P}(r^{-1}_{n})$. Then, for any $\epsilon>0$, there exists a $M>0$ and a $N \in \mathbb{N}$ such that
	\begin{align*}
\sup_{k \in \mathbb{K}} \mathbf{P} \left( ||\psi_{k}(P_{n}) - \psi_{k}(P)||_{\Theta} > \delta_{k}(M r^{-1}_{n}) \right) \leq \epsilon
	\end{align*}
	for all $n \geq N$. 
\end{lemma}

\begin{proof}[Proof of Lemma \ref{lem:unif-psik-Pn-P}]
	For any $k\in \mathbb{K}$ and any $n \in \mathbb{N}$, by continuity it follows $|| \psi_{k}(P_{n}(\boldsymbol{Z})) - \psi_{k}(P) ||_{\Theta} \leq \delta_{k}(d(P_{n}(\boldsymbol{Z}),P))$ a.s.-$\mathbf{P}$. In what follows, we omit the dependence on $\mathbf{Z}$.
	
	So it suffices to show that there exists a diverging $(k_{n})_{n}$ such that for any $\epsilon>0$, there exists a $N \in \mathbb{N}$ and a $M>0$ such that		
	\begin{align*}
	\sup_{k \in \mathbb{K}} \mathbf{P} \left( \delta_{k}(d(P_{n},P)) > \delta_{k}(M r^{-1}_{n}) \right) \leq \epsilon
	\end{align*}
	for all $n \geq N$. 
	
	Since $t \mapsto \delta_{k}(t)$ is non-decreasing for all $k \in \mathbb{K}$, it follows that $\mathbf{P} \left( \delta_{k}(d(P_{n},P)) > \delta_{k}(M r^{-1}_{n}) \right) \leq \mathbf{P} \left( d(P_{n},P) \geq M r^{-1}_{n} \right)$ for any $(n,k) \in \mathbb{N} \times \mathbb{K}$ and any $M>0$. Hence,
	\begin{align*}
	\sup_{k \in \mathbb{K}} \mathbf{P} \left( \delta_{k}(d(P_{n},P)) > \delta_{k}(M r^{-1}_{n}) \right) \leq \mathbf{P} \left(d(P_{n},P) \geq M r^{-1}_{n} \right)
	\end{align*}	
	for any $n \in \mathbb{N}$ and any $M>0$.
	
	By assumption,  $r_{n} d(P_{n},P) = O_{P}(1)$. This fact and the previous inequality imply the desired result.
\end{proof}

\begin{proof}[Proof of Theorem \ref{thm:consistent}]
	Fix any $\epsilon>0$. By the triangle inequality and laws of probability, for any $(k,n) \in \mathbb{N} \times \mathbb{K}$,
	\begin{align*}
	\mathbf{P} \left(  	||\psi_{k}(P_{n}) - \psi(P) ||_{\Theta}  > \epsilon    \right) \leq 	\mathbf{P} \left(  	||\psi_{k}(P_{n}) - \psi_{k}(P) ||_{\Theta}  > 0.5 \epsilon    \right) + 1\{ 	||\psi_{k}(P) - \psi(P) ||_{\Theta} > 0.5 \epsilon   \}.
	\end{align*}
	
	By assumption, there exists a $(r_{n})_{n \in \mathbb{N}}$ such that $r^{-1}_{n} = o(1)$ and $d(P_{n},P) = O_{P}(r^{-1}_{n})$. Thus, by Lemma \ref{lem:unif-psik-Pn-P}, there exists a $N \in \mathbb{N}$ and a $M>0$ such that for all $k \in \mathbb{K}$ and all $n \geq N$,
	\begin{align*}
	\mathbf{P} \left(  	||\psi_{k}(P_{n}) - \psi(P) ||_{\Theta}  > \epsilon    \right) \leq 	\epsilon + 1\{ \delta_{k}(M r^{-1}_{n}) > 0.5 \epsilon  \} + 1\{ 	||\psi_{k}(P) - \psi(P) ||_{\Theta} > 0.5 \epsilon   \}.
	\end{align*}
	Observe that for each $k$, $\delta_{k}(M r^{-1}_{n}) = o(1)$. Since $\mathbb{K}$ is unbounded it contains a diverging increasing sequence, therefore, by Lemma \ref{lem:sub-seq-conv}, there exists a diverging $(k_{n})_{n \in \mathbb{N}}$ such that  $\delta_{k_{n}}(M r^{-1}_{n}) = o(1)$. This result, condition 2 in the definition of regularization and the previous display at $k=k_{n}$, imply that 
	\begin{align*}
	\limsup_{n \rightarrow \infty}  \mathbf{P} \left(  	||\psi_{k_{n}}(P_{n}) - \psi(P) ||_{\Theta}  > \epsilon    \right) \leq 	\epsilon.
	\end{align*}
	
	Finally, we show that $\delta_{k_{n}}(d(P_{n},P)) = o_{P}(1)$. Since $t \mapsto \delta_{k}(t)$ is non-decreasing for all $k \in \mathbb{K}$ and $d(P_{n},P) = O_{P}(r^{-1}_{n})$ it follows that $\mathbf{P} \left( \delta_{k_{n}}(d(P_{n},P)) \geq  \delta_{k_{n}}(M r^{-1}_{n})   \right) \leq \epsilon$ for all $n \geq N'$ (WLOG we take $N'=N$). Since $\delta_{k_{n}}(M r^{-1}_{n}) = o(1)$ the result follows. 
\end{proof}

\section{Appendix for Section \ref{sec:rate-choice}}
\label{app:rate-choice}

Observe that the set $\mathcal{L}_{n}$ is random. To stress this dependence, with some abuse of notation, we will sometimes use $\mathcal{L}_{n}(\mathbf{z})$ to denote the set.

\subsection{Proof of Theorem \ref{thm:rate-choice}}
\label{app:proof-rate-choice}

The next lemma provides two sufficient conditions that ensure the result in Theorem \ref{thm:rate-choice}. To do this, for any $n \in \mathbb{N}$, let
\begin{align*}
D_{n} \equiv \{ \boldsymbol{z} \in \mathbb{Z}^{\infty} \colon  d(P_{n}(\boldsymbol{z}), P)  \leq r^{-1}_{n}      \}.
\end{align*}

\begin{lemma}\label{lem:suff-choice-rate}
	Suppose there exists a sequence $(j_{n})_{n \in \mathbb{N}}$ such that \begin{enumerate}
		\item For any $\epsilon>0$, there exists a $N$ such that $\mathbf{P}\left(   \{ \mathbf{z} \in \mathbb{Z}^{\infty} \colon j_{n} \in \mathcal{L}_{n}(\boldsymbol{z})\} \cap D_{n} \right) \geq 1-\epsilon$ for all $n \geq N$.
		\item There exists a constant $L<\infty$ such that $\bar{\delta}_{j_{n}}(r^{-1}_{n}) + \bar{B}_{j_{n}}(P) \leq L \inf_{k \in \mathbb{K}_{n}}  \{ \bar{\delta}_{k}(r^{-1}_{n}) + \bar{B}_{k}(P)  \}$.
	\end{enumerate}
Then \begin{align*}
   || \psi_{ \tilde{k}_{n}}(P_{n}) - \psi(P) ||_{\Theta} = O_{P} \left( \inf_{k \in \mathbb{K}_{n}}  \{ \bar{\delta}_{k}(r^{-1}_{n}) + \bar{B}_{k}(P)  \}  \right). 
\end{align*}
\end{lemma}

\begin{proof}
	Let $A_{n} \equiv \{  \boldsymbol{z} \in \mathbb{Z}^{\infty} \colon   j_{n} \in \mathcal{L}_{n}(\boldsymbol{z})    \}$. 
	
	For any $\boldsymbol{z} \in A_{n} \cap D_{n}$, it follows that 
	\begin{align*}
		  || \psi_{ \tilde{k}_{n}(\boldsymbol{z})}(P_{n}(\boldsymbol{z})) - \psi(P) ||_{\Theta} \leq  & || \psi_{ \tilde{k}_{n}(\boldsymbol{z})}(P_{n}(\boldsymbol{z})) - \psi_{j_{n}}(P_{n}(\boldsymbol{z})) ||_{\Theta} +  || \psi_{j_{n}}(P_{n}(\boldsymbol{z})) - \psi(P) ||_{\Theta} \\
		  \leq & || \psi_{ \tilde{k}_{n}(\boldsymbol{z})}(P_{n}(\boldsymbol{z})) - \psi_{j_{n}}(P_{n}(\boldsymbol{z})) ||_{\Theta} + \bar{\delta}_{j_{n}}(r^{-1}_{n}) + \bar{B}_{j_{n}}(P)\\
		 \leq &  4 \bar{\delta}_{j_{n}}(r^{-1}_{n}) + \bar{\delta}_{j_{n}}(r^{-1}_{n}) + \bar{B}_{j_{n}}(P),
	\end{align*}
	where the first linear follows from triangle inequality; the second line follows from the fact that $\boldsymbol{z} \in D_{n}$ and $t \mapsto \bar{\delta}_{k}(t)$ is non-decreasing; the third line follows from the fact that $\boldsymbol{z} \in A_{n}$ and thus $j_{n} \geq \tilde{k}_{n}(\boldsymbol{z})$. Thus, 
	\begin{align*}
		A_{n} \cap D_{n} \subseteq & \left\{   \boldsymbol{z} \in \mathbb{Z}^{\infty} \colon    || \psi_{ \tilde{k}_{n}(\boldsymbol{z})}(P_{n}(\boldsymbol{z})) - \psi(P) ||_{\Theta} \leq 5 \left(  \bar{\delta}_{j_{n}}(r^{-1}_{n}) + \bar{B}_{j_{n}}(P) \right)   \right\} \\
		\subseteq & \left\{   \boldsymbol{z} \in \mathbb{Z}^{\infty} \colon    || \psi_{ \tilde{k}_{n}(\boldsymbol{z})}(P_{n}(\boldsymbol{z})) - \psi(P) ||_{\Theta} \leq 5L \inf_{k \in \mathbb{K}_{n}}  \{ \bar{\delta}_{k}(r^{-1}_{n}) + \bar{B}_{k}(P)  \}  \right\}
	\end{align*}
	where the last linear follows from the second condition. Since by condition 1, $	A_{n} \cap D_{n}$ occurs with high probability, the result follows.
\end{proof}

We now construct a sequence $(h_{n})_{n}$ that satisfies both conditions of the lemma. To do this, let for each $n \in \mathbb{N}$,
\begin{align*}
	\mathbb{K}_{n}^{+} \equiv & \{ k \in \mathbb{K}_{n} \colon  \bar{\delta}_{k}(r^{-1}_{n}) \geq \bar{B}_{k}(P)  \}~\and~\\
	\mathbb{K}_{n}^{-} \equiv & \{ k \in \mathbb{K}_{n} \colon  \bar{\delta}_{k}(r^{-1}_{n}) \leq \bar{B}_{k}(P)  \}	.
\end{align*} 

\begin{remark}
	For any $n \in \mathbb{N}$, $\mathbb{K}_{n}^{+}$ or $\mathbb{K}_{n}^{-}$ are non-empty. 
%
$\triangle$
\end{remark}

For each $n \in \mathbb{N}$, let
\begin{align*}
T^{+}_{n} =     \bar{\delta}_{h^{+}_{n}}(r^{-1}_{n}) + \bar{B}_{h^{+}_{n}}(P)  
\end{align*}
if $\mathbb{K}_{n}^{+}$ is non-empty where  
\begin{align*}
h^{+}_{n} = \min \{  k \colon k \in \mathbb{K}_{n}^{+} \};
\end{align*}
and $T^{+}_{n} =   + \infty $, if $\mathbb{K}_{n}^{+}$ is empty. Similarly, 
\begin{align*}
T^{-}_{n} =     \bar{\delta}_{h^{-}_{n}}(r^{-1}_{n}) + \bar{B}_{h^{-}_{n}}(P)  
\end{align*}
if $\mathbb{K}_{n}^{-}$ is non-empty where  
\begin{align*}
h^{-}_{n} = \max \{  k \colon k \in \mathbb{K}_{n}^{-} \};
\end{align*}
and $T^{-}_{n} =   + \infty$, if $\mathbb{K}_{n}^{-}$ is empty.

\begin{remark}
	(1) Observe that when $\mathbb{K}_{n}^{+}$  (resp. $\mathbb{K}_{n}^{-}$ ) is non-empty, since it is discrete, $h^{+}_{n}$ (resp. $h^{-}_{n}$ )  is well-defined. 
	
	Intuitively,  $h^{+}_{n}$  is the ``round up" version within $\mathbb{K}_{n}$ of $k(n)$; and  $h^{-}_{n}$  is the ``round down" version within $\mathbb{K}_{n}$ of $k(n)$.
	
	\medskip 
	
	(2) By our previous observation and the fact that either $\mathbb{K}_{n}^{+}$ or $\mathbb{K}_{n}^{-}$ is non-empty, it follows that either $T^{+}_{n}$ or $T^{-}_{n}$ is finite. $\triangle$
\end{remark}

Finally, for each $n \in \mathbb{N}$, let $h_{n} \in \mathbb{K}_{n}$ be such that
\begin{align*}
    h_{n} = h^{+}_{n} 1\{  T^{+}_{n} \leq T^{-}_{n}  \} + h^{-}_{n} 1\{  T^{+}_{n} > T^{-}_{n}  \}.
\end{align*}

\begin{lemma}\label{lem:h(n)-prop}
	For each $n \in \mathbb{N}$, $h_{n}$ exists and \begin{align*}
	\bar{\delta}_{h_{n}}(r^{-1}_{n}) + \bar{B}_{h_{n}}(P)  = \min \left\{  T^{-}_{n} , T^{+}_{n}      \right\}.
	\end{align*}
\end{lemma}

\begin{proof}
   For each $n$, by our previous remark, either $T^{+}_{n}$ or $T^{-}_{n}$ is finite.
   
   If $T^{+}_{n} = \infty$, then $T^{-}_{n} < \infty = T^{+}_{n} $ so $h^{-}_{n}$ exists and $h_{n} = h^{-}_{n}$. 
   
   If $T^{-}_{n} = \infty$, then $T^{+}_{n} < \infty = T^{-}_{n} $ so $h^{+}_{n}$ exists and $h_{n} = h^{+}_{n}$.  
   
   Finally, if both are finite, then both $h^{+}_{n}$ and $h^{-}_{n}$ exist.
   
   The fact that 
   \begin{align*}
   \bar{\delta}_{h_{n}}(r^{-1}_{n}) + \bar{B}_{h_{n}}(P)  = \min \left\{  T^{-}_{n} , T^{+}_{n}      \right\}
   \end{align*}
   follows by construction.
\end{proof}

\begin{lemma}\label{lem:rate-h(n)-bound}
For each $n \in \mathbb{N}$, 
	\begin{align*}
	   \bar{\delta}_{h_{n}}(r^{-1}_{n}) + \bar{B}_{h_{n}}(P)   \leq 2 \inf_{k \in \mathbb{K}_{n}} \{  \bar{\delta}_{k}(r^{-1}_{n}) + \bar{B}_{k}(P)   \}  .
	\end{align*}
\end{lemma}

\begin{proof}
	Observe that 
	\begin{align*}
		\inf_{k \in \mathbb{K}_{n}} \{  \bar{\delta}_{k}(r^{-1}_{n}) + \bar{B}_{k}(P)   \}  \geq \min \{ \inf_{k \in \mathcal{G}^{+}_{n}} \{ \bar{\delta}_{k}(r^{-1}_{n}) + \bar{B}_{k}(P)    \} ,  \inf_{k \in \mathcal{G}^{-}_{n}} \{ \bar{\delta}_{k}(r^{-1}_{n}) + \bar{B}_{k}(P)    \}     \}
	\end{align*}
	where the  infimum is defined as $+\infty$ if the corresponding set is empty. 
	
     Fix any $n \in \mathbb{N}$, if $\mathcal{G}^{+}_{n} \ne \{\emptyset\}$, 
	\begin{align*}
		\inf_{k \in \mathcal{G}^{+}_{n}} \{ \bar{\delta}_{k}(r^{-1}_{n}) + \bar{B}_{k}(P)    \}  \geq \inf_{k \in \mathcal{G}^{+}_{n}} \{ \bar{\delta}_{k}(r^{-1}_{n})   \} = \bar{\delta}_{h^{+}_{n}}(r^{-1}_{n}) \geq 0.5 \left(  \bar{\delta}_{h^{+}_{n}}(r^{-1}_{n})  + \bar{B}_{h^{+}_{n}}(P)  \right)
	\end{align*}
	where the first inequality follows from the fact that $\bar{B}_{k}(P) \geq 0$; the second one (the equality) follows from the fact that $ k \mapsto \bar{\delta}_{k}(r^{-1}_{n})$ is non-decreasing and  that $h^{+}_{n}$ is minimal over $\mathcal{G}^{+}_{n} $; the third inequality follows from the fact that $\bar{\delta}_{h^{+}_{n}}(r^{-1}_{n})  \geq \bar{B}_{h^{+}_{n}}(P)$.

	Similarly, if $\mathcal{G}^{-}_{n} \ne \{\emptyset\}$, 
	\begin{align*}
	\inf_{k \in \mathcal{G}^{-}_{n}} \{ \bar{\delta}_{k}(r^{-1}_{n}) + \bar{B}_{k}(P)    \}  \geq \inf_{k \in \mathcal{G}^{-}_{n}} \{ \bar{B}_{k}(P)   \} = \bar{B}_{h^{-}_{n}}(P) \geq 0.5 \left(  \bar{\delta}_{h^{-}_{n}}(r^{-1}_{n})  + \bar{B}_{h^{-}_{n}}(P)  \right).
	\end{align*}
	Observe that here we use monotonicity of $k \mapsto \bar{B}_{k}(P)$. 
	
	Thus, \begin{align*}
			\inf_{k \in \mathbb{K}_{n}} \{  \bar{\delta}_{k}(r^{-1}_{n}) + \bar{B}_{k}(P)   \}  \geq 0.5 \min \{ T^{-}_{n} , T^{+}_{n}  \},
	\end{align*}
	and by Lemma \ref{lem:h(n)-prop} the desired result follows.

\end{proof}

\begin{lemma}\label{lem:hn-in-Fn}
  For any $n \in \mathbb{N}$, $\mathbf{P}(\{ \mathbf{z} \in \mathbb{Z}^{\infty} \colon h_{n} \notin \mathcal{L}_{n}(\boldsymbol{z}) \} ) \leq \mathbf{P}(D_{n}^{C})$.
\end{lemma}

\begin{proof}
	For any $n \in \mathbb{N}$,
	\begin{align*}
		\mathbf{P}( \{ \mathbf{z} \in \mathbb{Z}^{\infty} \colon h_{n} \notin \mathcal{L}_{n}(\boldsymbol{z}) \}  ) \leq \mathbf{P}( \{ \mathbf{z} \in \mathbb{Z}^{\infty} \colon h_{n} \notin \mathcal{L}_{n}(\boldsymbol{z}) \}   \cap D_{n} ) + \mathbf{P}(D_{n}^{C}).
	\end{align*}
	
	By definition of $\mathcal{L}_{n}$ (omitting the dependence on $\mathbf{Z}$),
	\begin{align*}
		\{ h_{n} \notin \mathcal{L}_{n} \}   \subseteq \left\{   \exists k \in \mathbb{K}_{n} \colon k > h_{n}~and~ ||\psi_{k}(P_{n}) - \psi_{h_{n}}(P_{n})||_{\Theta} > 4 \bar{\delta}_{k}(r^{-1}_{n})  \right\}.
	\end{align*}
	By triangle inequality and the fact that $t \mapsto \bar{\delta}_{k}(t)$ is non-decreasing, 
	\begin{align}\label{eqn:hn-in-Fn-1}
\{h_{n} \notin \mathcal{L}_{n} \} \cap D_{n} \subseteq \left\{   \exists k \in \mathbb{K}_{n} \colon k > h_{n}~and~ \bar{\delta}_{k}(r^{-1}_{n}) + \bar{B}_{k}(P) + \bar{\delta}_{h_{n}}(r^{-1}_{n}) + \bar{B}_{h_{n}}(P) > 4 \bar{\delta}_{k}(r^{-1}_{n})  \right\}.
\end{align}

   We now derive a series of useful claims.
   
   \medskip 

	\textbf{Claim 1:} If there exists $k \in \mathbb{K}_{n}$ such that $k > h_{n}$ and $h_{n} = h^{-}_{n}$, then $k \in \mathcal{G}^{+}_{n}$. \textbf{Proof:} If $h_{n}=h^{-}_{n}$, then $h_{n}$ is the largest element of $\mathcal{G}^{-}_{n}$ and thus $k \notin 	\mathcal{G}^{-}_{n}$, which means that $k \in \mathcal{G}^{+}_{n}$. $\square$
	
	\medskip 
	
	 A corollary of this claim is that if there exists $k \in \mathbb{K}_{n}$ such that $k > h_{n}$ and $h_{n} = h^{-}_{n}$, then $\mathcal{G}^{+}_{n}$ is non-empty. From this claim, we derive the following two claims. 
	 
	 \medskip 
	
	\textbf{Claim 2:} If there exists a $k > h_{n}$, then  $\bar{\delta}_{h_{n}}(r^{-1}_{n}) + \bar{B}_{h_{n}}(P) \leq 2 \bar{\delta}_{h^{+}_{n}}(r^{-1}_{n})$. \textbf{Proof:} If $h_{n} = h^{+}_{n}$, then $\bar{\delta}_{h_{n}}(r^{-1}_{n}) + \bar{B}_{h_{n}}(P) \leq \bar{\delta}_{h^{+}_{n}}(r^{-1}_{n}) + \bar{B}_{h^{+}_{n}}(P) \leq 2 \bar{\delta}_{h^{+}_{n}}(r^{-1}_{n})$. If $h_{n} = h^{-}_{n}$, by the previous claim it follows that $\mathcal{G}^{+}_{n}$ is non-empty and thus $h^{+}_{n}$ is well-defined, thus $\bar{\delta}_{h_{n}}(r^{-1}_{n}) + \bar{B}_{h_{n}}(P) \leq \bar{\delta}_{h^{+}_{n}}(r^{-1}_{n}) + \bar{B}_{h^{+}_{n}}(P) \leq 2 \bar{\delta}_{h^{+}_{n}}(r^{-1}_{n}) $. $\square$

    \medskip 
    
	\textbf{Claim 3:} For any $k > h_{n}$, $\bar{\delta}_{k}(r^{-1}_{n}) \geq \bar{B}_{k}(P)$. \textbf{Proof:} If $h_{n} = h^{+}_{n}$ then the claim follows because $k \mapsto \bar{\delta}_{k}(r^{-1}_{n}) - \bar{B}_{k}(P)$ is non-decreasing. 
	If $h_{n}=h^{-}_{n}$, then $k \in \mathcal{G}^{+}_{n}$ by Claim 1 and thus $\bar{\delta}_{k}(r^{-1}_{n}) \geq \bar{B}_{k}(P)$. $\square$

\medskip 

By Claims 2 and 3, it follows that  if there exists $k \in \mathbb{K}_{n}$ such that  $k \geq h_{n}$, then $\bar{\delta}_{k}(r^{-1}_{n}) + \bar{B}_{k}(P) + \bar{\delta}_{h_{n}}(r^{-1}_{n}) + \bar{B}_{h_{n}}(P) \leq 2 \bar{\delta}_{k}(r^{-1}_{n}) + 2\bar{\delta}_{h^{+}_{n}}(r^{-1}_{n}) \leq 4\bar{\delta}_{k}(r^{-1}_{n})$ where the last inequality follows monotonicity of $k \mapsto \bar{\delta}_{k}(r^{-1}_{n})$ 
and the fact that $k  > h^{+}_{n}$ because $k > h_{n}$ and so by Claim 1 $k \in \mathbb{K}_{n}^{+}$ and $h^{+}_{n}$ is minimal in this set. Applying this to expression \ref{eqn:hn-in-Fn-1}, it follows that	\begin{align}
\{h_{n} \notin \mathcal{L}_{n} \} \cap D_{n} \subseteq \left\{   \exists k \in \mathbb{K}_{n} \colon k \geq h_{n}~and~ 4\bar{\delta}_{k}(r^{-1}_{n}) > 4 \bar{\delta}_{k}(r^{-1}_{n})  \right\},
\end{align}
which is empty. Hence, $\mathbf{P}( \{ \mathbf{z} \in \mathbb{Z}^{\infty} \colon h_{n} \notin \mathcal{L}_{n}(\boldsymbol{z}) \} ) \leq \mathbf{P}(D_{n}^{C})$ as desired.
\end{proof}

\begin{proof}[Proof of Theorem \ref{thm:rate-choice}]
	We verify that  $(h_{n})_{n \in \mathbb{N}}$ satisfies both conditions in Lemma \ref{lem:suff-choice-rate}. By Lemma \ref{lem:rate-h(n)-bound} condition 2 in the Lemma \ref{lem:suff-choice-rate} holds with $L=2$. To check condition 1 in the Lemma \ref{lem:suff-choice-rate}, observe that 
   \begin{align*}
   	\mathbf{P} \left( \mathbb{Z}^{\infty} \setminus \left\{ \{ \mathbf{z} \in \mathbb{Z}^{\infty} \colon h_{n} \in \mathcal{L}_{n}(\boldsymbol{z}) \}  \cap D_{n} \right\} \right) \leq \mathbf{P} \left(  \{ \mathbf{z} \in \mathbb{Z}^{\infty} \colon h_{n} \notin \mathcal{L}_{n}(\boldsymbol{z}) \}    \right) + \mathbf{P} \left( D^{C}_{n} \right).
   \end{align*}
   Thus, by Lemma \ref{lem:hn-in-Fn} and the fact $\lim_{n \rightarrow \infty}\mathbf{P}(D_{n}^{C}) = 0$, $(h_{n})_{n \in \mathbb{N}}$ condition 2 is satisfied.
\end{proof}

\subsection{Proof of Proposition \ref{pro:rate-choice-global}}
\label{app:rate-choice-global}

For any $n \in \mathbb{N}$, let \begin{align*}
k(n) = \min \{ k \in \mathbb{R}_{+} \colon   \bar{\delta}_{k}(r^{-1}_{n} ) \geq \bar{B}_{k}(P) \}.
\end{align*}

\begin{lemma}\label{lem:feat-k(n)}
	For each $n \in \mathbb{N}$, $k(n)$ exists and solves
	\begin{align*}
	\bar{\delta}_{k(n)}(r^{-1}_{n} ) = \bar{B}_{k(n)}(P) = \min_{k \in \mathbb{R}_{+}} \max \{ 	\bar{\delta}_{k}(r^{-1}_{n} ), \bar{B}_{k}(P) \}.
	\end{align*}
\end{lemma}

\begin{proof}
	For each $n$ consider the set $\{ k \in \mathbb{R}_{+} \colon \bar{\delta}_{k}(r^{-1}_{n} ) \geq \bar{B}_{k}(P)    \}$. The set is closed since $k \mapsto \bar{B}_{k}(P)$ and $k \mapsto \bar{\delta}_{k}(r^{-1}_{n})$ are continuous. Since $\bar{\delta}_{k}(r^{-1}_{n} ) >0$ and $\bar{B}_{k}(P) = o(1)$, if follows that there exists a $K(n)<\infty$ such that  $\bar{\delta}_{k}(r^{-1}_{n} ) \geq \bar{B}_{k}(P)$ for any $k \geq K(n)$. Thus the set is non-empty and since we are minimizing the identity function, the minimizer exists and uniquely determined by 
	\begin{align*}
	\bar{\delta}_{k(n)}(r^{-1}_{n} ) = \bar{B}_{k(n)}(P).
	\end{align*}
	The second  equality is obvious.
\end{proof}

%

The next lemma shows that balancing the sampling and approximation error yields the same rate as the ``optimal" choice. 

\begin{lemma}\label{lem:Rate-k}
	For any $n \in \mathbb{N}$,
	\begin{align*}
	\bar{\delta}_{k(n)}(r^{-1}_{n})  \leq  \inf_{k \in \mathbb{R}_{+}} \{  \bar{\delta}_{k}(r^{-1}_{n}) + \bar{B}_{k}(P) \} \leq 2 \bar{\delta}_{k(n)}(r^{-1}_{n}).
	\end{align*}	
\end{lemma}

\begin{proof}
	Observe that for any $n \in \mathbb{N}$ and any $\epsilon>0$, there exists $k^{\ast}(n)$ such that 
	\begin{align*}
	\bar{\delta}_{k^{\ast}(n)}(r^{-1}_{n})  +  \bar{B}_{k^{\ast}(n)}(P) - \epsilon  \leq  \inf_{k \in \mathbb{R}_{+}} \{  \bar{\delta}_{k}(r^{-1}_{n}) + \bar{B}_{k}(P) \} \leq 2 \bar{\delta}_{k(n)}(r^{-1}_{n}).
	\end{align*}
	The upper bound follows from the fact that $\inf_{k \in \mathbb{R}_{+}} \{  \bar{\delta}_{k}(r^{-1}_{n}) + \bar{B}_{k}(P) \}  \leq   \bar{\delta}_{k(n)}(r^{-1}_{n}) + \bar{B}_{k(n)}(P) $ and definition of $k(n)$. 
	
	If $k^{\ast}(n) \geq k(n)$, then $\bar{\delta}_{k^{\ast}(n)}(r^{-1}_{n})  \geq \bar{\delta}_{k(n)}(r^{-1}_{n})  $ since $k \mapsto \bar{\delta}_{k}(t)$ is non-decreasing for any $t \geq 0$ 
	On the other hand, if $k^{\ast}(n) < k(n)$, then $\bar{B}_{k^{\ast}(n)}(P) \geq \bar{B}_{k(n)}(P) =  \bar{\delta}_{k(n)}(r^{-1}_{n})$ where the last equality follows from Lemma \ref{lem:feat-k(n)}. Therefore, for any $n \in \mathbb{N}$ and any $\epsilon>0$, 
	\begin{align*}
	\bar{\delta}_{k(n)}(r^{-1}_{n})  - \epsilon  \leq  \inf_{k \in \mathbb{R}} \{  \bar{\delta}_{k}(r^{-1}_{n}) + \bar{B}_{k}(P) \} \leq 2 \bar{\delta}_{k(n)}(r^{-1}_{n}).
	\end{align*}
	Since $\epsilon>0$ is arbitrary the result follows.
\end{proof}

\subsection{Extensions  of Theorem \ref{thm:rate-choice}}
\label{app:rate-choice-Gn-gral}

In this section we show that Theorem \ref{thm:rate-choice} can be extended to whole when $\mathbb{K}_{n}$ is any \emph{closed} set of $\mathbb{K}$. The extension is merely technical as one needs to ensure that some minimizers are attained over $\mathbb{K}_{n}$ when this set is not finite. 

First, one needs to ensure that $\tilde{k}_{n}$ exists with probability approaching 1. Lemma \ref{lem:hn-in-Fn} shows that with probability approaching one, the set $\mathcal{L}_{n}$ is non-empty. Thus, it suffices to argue that $\mathcal{L}_{n}$ is closed. It is easy to see that the following assumption is sufficient for this.
\begin{assumption}\label{ass:psik-cont-k}
	For each $P \in \mathcal{D}$ and each $t \geq 0$, the mapping $k \mapsto \psi_{k}(P)$ 
	is continuous over $\mathbb{K}_{n}$.
\end{assumption}
Observe that when $\mathbb{K}_{n}$ is finite this condition is trivially satisfied and that is why it is not imposed in the text.

The following theorem is an extension of Theorem \ref{thm:rate-choice} to the case where $\mathbb{K}_{n}$ is a closed set (not necessarily finite) of $\mathbb{K}$.

\begin{theorem}\label{thm:rate-choice-ext}
	Suppose all assumptions in Theorem \ref{thm:rate-choice} hold. And suppose further that $\mathbb{K}_{n}$ is a closed set (not necessarily finite) of $\mathbb{K}$ and that Assumptions \ref{ass:psik-cont-k} holds. Then
	\begin{align*}
	||  \psi_{\tilde{k}_{n}}(P_{n}) - \psi(P) ||_{\Theta} = O_{P} \left(  \inf_{k \in \mathbb{K}_{n}} \{ \bar{\delta}_{k}(r^{-1}_{n}) + \bar{B}_{k}(P)   \}     \right).
	\end{align*}
\end{theorem}

\begin{proof}
	The proof is identical to the one of Theorem \ref{thm:rate-choice}. Assumption \ref{ass:psik-cont-k} and the fact that $\mathbb{K}$ is closed ensure that the quantities defined in the proof exist.
\end{proof}

The following proposition extends the result in Theorem \ref{thm:rate-choice} to an un-restricted one --- where the infimum is not restricted to the set $\mathbb{K}_{n}$ but is taken over the whole $\mathbb{K}$. Unsurprisingly, in order to obtain this result, additional conditions are needed.

\begin{proposition}\label{pro:rate-choice-global}
	Suppose all conditions in Theorem \ref{thm:rate-choice} hold, and that $k \mapsto \bar{\delta}_{k}(t)$ and $k \mapsto \bar{B}_{k}(P)$ are continuous and
	\begin{align}\label{eqn:rate-choice-global-grid}
	\frac{\min_{k \in \mathbb{K}^{+}_{n}} \bar{\delta}_{k}(r^{-1}_{n}) }{\max_{k \in \mathbb{K}^{-}_{n}} \bar{\delta}_{k}(r^{-1}_{n}) } = O(1)
	\end{align}
	where $\mathbb{K}^{+}_{n} \equiv \{ k \in \mathbb{K}_{n} \colon \bar{\delta}_{k}(r^{-1}_{n}) \geq \bar{B}_{k}(P)  \}$ and $\mathbb{K}^{-}_{n} \equiv \{ k \in \mathbb{K}_{n} \colon \bar{\delta}_{k}(r^{-1}_{n}) \leq \bar{B}_{k}(P)  \}$ are non-empty.	
	Then 
	\begin{align*}
	|| \psi_{\tilde{k}_{n}}(P_{n}) - \psi(P)  ||_{\Theta} = O_{P} \left(  \inf_{k \in \mathbb{K}} \{ \bar{\delta}_{k}(r^{-1}_{n}) + \bar{B}_{k}(P)    \}     \right).
	\end{align*}	
\end{proposition}

\begin{proof}[Proof of Proposition \ref{pro:rate-choice-global}]
	By inspection of the proof of Lemma \ref{lem:suff-choice-rate} it suffices to show existence of a sequence $(j_{n})_{n}$ for which Condition 1 and the following strengthening of condition 2 holds:
	
	\medskip
	
	Condition 2': There exists a constant $L< \infty$ such that $\bar{\delta}_{j_{n}}(r^{-1}_{n}) + \bar{B}_{j_{n}}(P) \leq L \inf_{k \in \mathbb{R}_{+}} \{ \bar{\delta}_{k}(r^{-1}_{n}) + \bar{B}_{k}(P) \}$ for any $n \in \mathbb{N}$. 
	
	\medskip
	
	As for the proof of Theorem \ref{thm:rate-choice}, we propose $j_{n} = h_{n}$ for all $n \in \mathbb{N}$. By Lemma \ref{lem:hn-in-Fn} condition 1 holds, so it only remains to show that Condition 2' holds. 
	
	Under the conditions in the proposition, $h^{+}_{n}$ and $h^{-}_{n}$ are well-defined for all $n \in \mathbb{N}$. Moreover, they are either the same or consecutive elements in $\mathbb{K}_{n}$. Thus, under the conditions in the proposition, there exists a $C< \infty$ and a $N \in \mathbb{N}$ such that $\bar{\delta}_{h^{+}_{n}}(r^{-1}_{n}) \leq C \bar{\delta}_{h^{-}_{n}}(r^{-1}_{n})$ for all $n \geq N$. Therefore, for all $n \geq N$,
	\begin{align*}
	\bar{\delta}_{h_{n}}(r^{-1}_{n}) + \bar{B}_{h_{n}}(P) = & \min\{  \bar{\delta}_{h^{-}_{n}}(r^{-1}_{n}) + \bar{B}_{h^{-}_{n}}(P) , \bar{\delta}_{h^{+}_{n}}(r^{-1}_{n}) + \bar{B}_{h^{+}_{n}}(P)  \}\\
	\leq & \min\{  \bar{\delta}_{h^{-}_{n}}(r^{-1}_{n}) + \bar{B}_{h^{-}_{n}}(P) , 2 \bar{\delta}_{h^{+}_{n}}(r^{-1}_{n})   \}\\
	\leq & \min\{  \bar{\delta}_{h^{-}_{n}}(r^{-1}_{n}) + \bar{B}_{h^{-}_{n}}(P) , 2 C \bar{\delta}_{h^{-}_{n}}(r^{-1}_{n})  \} \\
	\leq & 2 C \bar{\delta}_{h^{-}_{n}}(r^{-1}_{n}).
	\end{align*}
	
	Observe that  $h^{-}_{n} \leq k(n)$ because, by definition $\bar{\delta}_{h^{-}_{n}}(r^{-1}_{n}) \leq \bar{B}_{h^{-}_{n}}(P)$ but $k(n)$ satisfies $\bar{\delta}_{k(n)}(r^{-1}_{n}) \geq \bar{B}_{k(n)}(P)$, so under 
	the fact that $k \mapsto \bar{B}_{k}(P)$ is non-increasing it must follow that $h^{-}_{n} \leq k(n)$. Thus, $\bar{\delta}_{h^{-}_{n}}(r^{-1}_{n}) \leq \bar{\delta}_{k(n)}(r^{-1}_{n})$ and the result follows from Lemma \ref{lem:Rate-k}.
\end{proof}

\begin{remark}[Sufficient conditions for $\mathbb{K}_{n}^{+}$ and $\mathbb{K}_{n}^{-}$ to be non-empty in the proposition]
	The condition that $\bar{\delta}_{j}(r^{-1}_{n} ) < \bar{B}_{j}(P)$ for some $j \in \mathbb{K}_{n}$ is easy to satisfy as any fix $j$ (e.g. $j=1$) will satisfy this condition eventually. The other inequality is more delicate but the next lemma provides the basis for its verification.
	
	\begin{lemma}\label{lem:delta-k(n)-o(1)}
		$\limsup_{n \rightarrow \infty} \bar{\delta}_{k(n)}(r^{-1}_{n}) = 0$.
	\end{lemma}
	
	\begin{proof}
		Suppose not. Then there exists a sub-sequence $(n(j))_{j}$ and a $c>0$ such that $\bar{\delta}_{k(n(j))}(r^{-1}_{n(j)}) \geq c$ for all $j$. Clearly $(k(n(j)))_{j}$ must diverge, so $\bar{B}_{k(n(j))}(P) = o(1)$, but then $k(n(j))$ cannot be balancing both terms.
	\end{proof}
	
	Let $(j(n))_{n}$ be such that $\liminf_{n \rightarrow \infty} \bar{\delta}_{j(n)}(r^{-1}_{n} ) > 0$. Since $\limsup_{n \rightarrow \infty} \bar{\delta}_{k(n)}(r^{-1}_{n} ) = 0$, it follows that $j(n) > k(n)$ eventually and thus $\bar{\delta}_{j(n)}(r^{-1}_{n} ) > \bar{B}_{j(n)}(P)$. 
	
	Thus, any set $\mathbb{K}_{n}$ such that $\mathbb{K}_{n} \ni \{1,j(n)\}$ will satisfy that $\mathbb{K}_{n}^{+}$ and $\mathbb{K}_{n}^{-}$ are non-empty, at least for sufficiently large $n$. $\triangle$

\end{remark}

\begin{remark}[On the conditions in the proposition] 
	
	The continuity condition is technical and it ensures that certain minimizers/maximizers are well-defined. Condition \ref{eqn:rate-choice-global-grid} imply the following two restrictions that are used in the proof:

	(1) The fact that both $\mathbb{K}^{+}_{n}$ and $\mathbb{K}^{-}_{n}$ are non-empty ensures that the set $\mathbb{K}_{n}$ surrounds the choice of tuning parameter that balances the sampling error and the monotone envelope of the approximation error. If this condition fails, the minimal value of $k \mapsto \{ \bar{\delta}_{k}(r^{-1}_{n}) + \bar{B}_{k}(P)    \}   $ over $\mathbb{K}_{n}$ cannot be expected to be close to the value achieved when balancing both terms and thus close to the minimal value over $\mathbb{R}_{+}$. In Appendix \ref{app:rate-choice-global} we argue that $\mathbb{K}_{n} = \{1,....,j(n)\}$ where $(j(n))_{n}$ is such that $\liminf_{n \rightarrow \infty} \bar{\delta}_{j(n)}(r^{-1}_{n}) > 0$ satisfies this assumption, at least for large $n$.

	(2) The second role is more subtle. It essentially restricts --- uniformly --- the coarseness of the set $\mathbb{K}_{n}$ in terms of $k \mapsto \bar{\delta}_{k}(t)$. If $\bar{\delta}_{k}(t) = a(t) \times C_{k}$ and $\mathbb{K}_{n} = \mathbb{N}$, then the condition essentially imposes that $\limsup_{k \rightarrow \infty} C_{k+1}/C_{k} < \infty$; thus it allows for $C_{k} \asymp Poly(k)$ and $ \log C_{k} \asymp k $ but not for $\log C_{k} \asymp k^{2}$.  $\triangle$ 
\end{remark}

\subsection{Appendix for Section \ref{sec:exa-cont}}
\label{app:exa-cont}

 For any probability $P$ over a set $A$ and any $k \in \mathbb{N}$, let $\bigotimes_{i=1}^{k} P$ be the product probability measure over $\prod_{i=1}^{k} A$ induced by $P$. Also, recal that the Wasserstein distance for $p\geq 1$ over $\mathcal{P}(\prod_{i=1}^{k} \mathbb{Z})$ for some $k \geq 1$ is defined as 
\begin{align*}
 \mathcal{W}_{p}(P,Q) \equiv \left( \inf_{\mu \in H(P,Q)} \int ||x - y||^{p} \mu (dx,dy)   \right)^{1/p}
\end{align*}
for any $P,Q$ in $\mathcal{P}(\prod_{i=1}^{k} \mathbb{Z})$, where $H(P,Q)$ is the class of Borel probability measures over $\mathbb{Z}^{2k}$ with marginals $P$ and $Q$. It is well-known that 
\begin{align*}
	(P,Q) \mapsto ||P-Q||_{Lip\left( \prod_{i=1}^{k} \mathbb{Z}  \right)} = \mathcal{W}_{1}(P,Q)
\end{align*}
where for any set $A$, let $LB(A)$ to denote the class of bounded Lipschitz (with constant 1) real-valued functions on $A$; see \cite{villani2008optimal} p. 60.

The following lemma is used in the proof

\begin{lemma}\label{lem:gamma-k}
	For any $P,Q$ in $\mathcal{P}(\mathbb{Z})$, any $k \in \mathbb{N}$ and any $\mu \in H(P,Q)$, $\bigotimes_{i=1}^{k} \mu \in H \left(\bigotimes_{i=1}^{k} P, \bigotimes_{i=1}^{k} Q \right) $.
\end{lemma}

\begin{proof}[Proof of Lemma \ref{lem:gamma-k}]
	It is clear that the marginal of $\bigotimes_{i=1}^{k} \mu $ of a pair $(x_{i},y_{i})$ is $\mu$.  Therefore, for any $A_{1},...,A_{k}$ Borel subsets on $\mathbb{Z}$,
	\begin{align*}
	\bigotimes_{i=1}^{k} \mu \left( (A_{1} \times \mathbb{Z} ) , ...,  (A_{k} \times \mathbb{Z} )  \right) = \prod_{i=1}^{k} \mu ( A_{i} \times \mathbb{Z} )  =^{\ast}  \prod_{i=1}^{k} P(A_{i})
	\end{align*}
	where $\ast$ follows because $\mu \in H(P,Q)$. Equivalently, 
	\begin{align*}
	\int g(\vec{x})   \bigotimes_{i=1}^{k} \mu (d\vec{x},d\vec{y}) = \int g(\vec{x}) \bigotimes_{i=1}^{k} P (d\vec{x})  
	\end{align*} 
	for any $g$ belonging to the class of ``simple" functions on $\prod_{i=1}^{k} \mathbb{Z}$: The class of functions of the form $g(\vec{x}) = \sum_{i=1}^{k} 1_{A_{i}}(x_{i})$ for any $A_{1},...,A_{k}$ Borel subsets on $\mathbb{Z}$. Since the class of ``simple" functions is dense in $\mathbb{C}(\mathbb{Z}^{k},\mathbb{R})$ (the class of continuous and bounded functions over $\mathbb{Z}$), by taking limits and using the previous display one can show that
	\begin{align*}
	\int f(\vec{x})   \bigotimes_{i=1}^{k} \mu (d\vec{x},d\vec{y}) = \int f(\vec{x}) \bigotimes_{i=1}^{k} P (d\vec{x}) 
	\end{align*} 	
	for any $f \in \mathbb{C}(\prod_{i=1}^{k} \mathbb{Z},\mathbb{R})$. That is, the marginal probability of $\bigotimes_{i=1}^{k} \mu $ for the first $k$ coordinates is $\bigotimes_{i=1}^{k} P$. A completely analogous argument shows that the marginal probability of $\bigotimes_{i=1}^{k} \mu $ for the last $k$ coordinates is $\bigotimes_{i=1}^{k} Q$. 
\end{proof}

\begin{proof}[Proof of Proposition \ref{pro:boot-cont}]
	First note that, for any $f \in LB$, $E_{\psi_{n}(P) }[f(Z)] = \int \psi_{n}(P) (f(Z) \geq t)  dt = \int  \mathbf{P} (f(T_{n}(\boldsymbol{z},P)) \geq t) dt = E_{\mathbf{P}} \left[  f \left(  T_{n}(\boldsymbol{z},P) \right)   \right]   $. Hence, it follows that
	\begin{align*}
	||\psi_{k}(P) - \psi_{k}(Q)||_{\Theta} \leq & \sup_{f \in Lip} \left| E_{\mathbf{P}} \left[ f \left( T_{k}(\boldsymbol{z},P)  \right) \right]   - E_{Q^{\infty}} \left[ f \left( T_{k}(\boldsymbol{z},P)  \right) \right]    \right| \\
	& + \sup_{f \in Lip}  \left| E_{Q^{\infty}} \left[ f \left( T_{k}(\boldsymbol{z},P)  \right) - f \left( T_{k}(\boldsymbol{z},Q)  \right) \right]    \right|\\
	\equiv & T_{1,k}(P,Q) + T_{2,k}(P,Q). 
	\end{align*}
	
	We now show that both terms, $T_{1,k}(P,Q)$ and $T_{2,k}(P,Q)$, are bounded by $\sqrt{k} \mathcal{W}_{1}(P,Q)$.

	For any $k \in \mathbb{N}$, $f \left( T_{k}(\boldsymbol{z},P)  \right)  \equiv f_{k} \left( \sqrt{k} \max\{ k^{-1} \sum_{i=1}^{k} Z_{i}(\boldsymbol{z})  , 0  \}  \right) $ where $f_{k} \equiv f( \cdot - \sqrt{k} \max\{ E_{P}[Z],0 \} )$. It is easy to see that for any $k$, $f_{k} \in LB$ (given that $f \in LB$). Moreover, the mapping $t \mapsto g_{k}(t) \equiv f_{k}(\max\{ t, 0 \})$ is also in $Lip$ because 
	\begin{align*}
	|g_{k}(t) - g_{k}(t')| \leq |\max\{t,0\} - \max\{t',0\}| \leq |t'-t|,~\forall t,t'.
	\end{align*}
	Finally, the mapping $t \mapsto g_{k}(\sqrt{k} t)/\sqrt{k}$ is also in $LB$ since $g_{k} \in Lip$.  Hence,
	\begin{align}\label{eqn:T1k-1}
	T_{1,k}(P,Q) \leq &  \sqrt{k} \sup_{f \in Lip} \left| E_{\bigotimes_{i=1}^{k}   P} \left[ f \left( k^{-1} \sum_{i=1}^{k} Z_{i}  \right) \right]   - E_{\bigotimes_{i=1}^{k} Q} \left[ f \left( k^{-1} \sum_{i=1}^{k} Z_{i}  \right) \right]    \right|.
	\end{align}
	
	For any $g : \mathbb{R} \rightarrow \mathbb{R}$, let $\bar{g} : \mathbb{R}^{k} \rightarrow \mathbb{R}^{k}$ be defined as
	\begin{align*}
	\vec{z} \equiv (z_{1},...,z_{k}) \mapsto \bar{g}(\vec{z}) \equiv	g\left(  k^{-1} \sum _{i=1}^{k} z_{i}   \right).
	\end{align*}
	
	We now show that for any $g \in LB(\mathbb{R})$, $k \bar{g} \in LB(\mathbb{R}^{k})$. This follows because
	\begin{align*}
	|\bar{g}(\vec{z})  - \bar{g}(\vec{z}^{\prime}) | \leq |k^{-1} \sum_{i=1}^{k} \{ z_{i} - z^{\prime}_{i} \} | \leq k^{-1} ||\vec{z} - \vec{z}^{\prime}||_{1} .
	\end{align*}
	
	This result allow us to bound from above the LHS of the expression  \ref{eqn:T1k-1} so that
	\begin{align*}
		& \sqrt{k} \sup_{f \in LB} \left| E_{\bigotimes_{i=1}^{k}   P} \left[ f \left( k^{-1} \sum_{i=1}^{k} Z_{i}  \right) \right]   - E_{\bigotimes_{i=1}^{k} Q} \left[ f \left( k^{-1} \sum_{i=1}^{k} Z_{i}  \right) \right]    \right| \\
		 \leq & k^{-1/2} 	\sup_{f \in LB(\mathbb{R}^{k})} \left| E_{\bigotimes_{i=1}^{k}   P} \left[  f(Z)  \right]   - E_{\bigotimes_{i=1}^{k}   Q} \left[  f(Z')  \right]   \right| \\
		 = & k^{-1/2} \mathcal{W}_{1} \left( \bigotimes_{i=1}^{k} P, \bigotimes_{i=1}^{k} Q \right). 
	\end{align*}
	
	For any $\mu \in H(P,Q)$,  $\bigotimes_{i=1}^{k} \mu \in \mathcal{P}(\mathbb{Z}^{k} \times \mathbb{Z}^{k})$ where $\mathbb{Z}^{k} \equiv \prod_{i=1}^{k} \mathbb{Z}$. Moreover, by Lemma \ref{lem:gamma-k}, $\bigotimes_{i=1}^{k} \mu \in H \left( \bigotimes_{i=1}^{k} P,\bigotimes_{i=1}^{k} Q \right)$.

	For any $\eta>0$, let $\mu^{\ast} \in H(P,Q)$ be the approximate minimizer of $\mathcal{W}_{1}(P,Q)$, i.e., 
	\begin{align*}
	  \int |x - y| \mu^{\ast}(dx,dy)   \leq \mathcal{W}_{1}(P,Q) + \eta,
	\end{align*}
	as $\bigotimes_{i=1}^{k} \mu^{\ast} \in  H \left( \bigotimes_{i=1}^{k} P,\bigotimes_{i=1}^{k} Q \right)$, it follows that 
	\begin{align*}
	\mathcal{W}_{1}\left( \bigotimes_{i=1}^{k} P, \bigotimes_{i=1}^{k} Q \right) \leq &  	\int_{\mathbb{Z}^{2k}} ||\vec{x} - \vec{y}||_{1} \bigotimes^{k}_{i=1} \mu^{\ast}(dx_{i},dy_{i})   \\
	= & \sum_{i=1}^{k} \int_{\mathbb{Z}^{2}}  |x_{i} - y_{i}|  \bigotimes^{k}_{i=1} \mu^{\ast}(dx_{i},dy_{i})   \\
	= & \sum_{i=1}^{k} \int_{\mathbb{Z}^{2}}  |x_{i} - y_{i}|  \mu^{\ast}(dx_{i},dy_{i})  \\
	= & k \mathcal{W}_{1}(P,Q) + k \eta.
	\end{align*}
	Since $\eta>0$ is arbitrary, it follows that $	\mathcal{W}_{1}(P^{k},Q^{k})  \leq  k \mathcal{W}_{1}(P,Q) $. Thus implying
	\begin{align*}
	T_{1,k}(P,Q)  \leq \sqrt{k}\mathcal{W}_{1}(P,Q).
	\end{align*}
	
	Regarding the term $T_{2,k}(P,Q) $, observe that
	\begin{align*}
	  T_{2,k}(P,Q) \leq & |T_{k}(\boldsymbol{z},P) - T_{k}(\boldsymbol{z},Q)| \\
	  = & \sqrt{k} | \max\{ E_{P}[Z],0  \}  -  \max\{ E_{Q}[Z'],0  \}  | \\
	  \leq & \sqrt{k} |  E_{P}[Z] - E_{Q}[Z'] |.
	\end{align*}
	Since $E_{P}[Z] = \int_{\mathbb{Z}^{2}} z \mu(dz,dz')$ for any $\mu \in H(P,Q)$. 
	\begin{align*}
	 T_{2,k}(P,Q) \leq 	\sqrt{k} \left| E_{\mu}[Z] - E_{\mu}[Z']  \right| \leq \sqrt{k} \int |z-z'| \mu(dz,dz').
	\end{align*}
	Choosing $\mu$ as the (approximate) minimizer of $\mathcal{W}_{1}(P,Q)$ it follows that 
	\begin{align*}
	T_{2,k}(P,Q)  \leq \sqrt{k}\mathcal{W}_{1}(P,Q).
	\end{align*}
\end{proof}

To show the proposition \ref{pro:boot-choice}, let $\psi(P) \in \mathcal{P}(\mathbb{R})$ be defined as the probability of $\max\{\zeta, 0\}$ if $P$ is such that $E_{P}[Z]=0$ and the probability of $\zeta$ otherwise, where $\zeta \sim N(0,1)$. The following lemma shows that $\psi(P)$ is the limit of $(\psi_{k}(P))_{k \in \mathbb{N}}$.

\begin{lemma}\label{lem:boot-bias}
	For any $k \in \mathbb{N}$ and any $P \in \mathcal{M}$,
	\begin{align*}
	||\psi_{k}(P) - \psi(P)||_{LB} \leq 6 k^{-1/2} E_{P}[|Z|^{3}] + 1\{  E_{P}[Z] >0  \} 2 \Phi(-\sqrt{k}E_{P}[Z])
	\end{align*} 
\end{lemma}

\begin{proof}
	Since $P \in \mathcal{M}$, $T_{k} (\boldsymbol{z},P)   = \max\{ k^{-1/2} \sum_{i=1}^{k} (Z_{i}(\boldsymbol{z}) -  E_{P}[Z]),  - \sqrt{k} E_{P}[Z]   \}  $ 
	for any $k \in \mathbb{N}$. 
	
	By triangle inequality and definition of $||.||_{LB}$,
	\begin{align*}
	|| \psi_{k}(P) - \psi(P) ||_{LB} \leq &  \sup_{f \in LB} E \left[ \left| f\left( T_{k} (\boldsymbol{z},P) \right)   - f \left( \max\{ \zeta , - \sqrt{k} E_{P}[Z]    \}   \right)         \right|   \right]\\
	& + \sup_{f \in LB} E \left[ \left| f \left( \max\{ \zeta , - \sqrt{k} E_{P}[Z]    \}   \right) -  f \left( s_{P}(\zeta)   \right)        \right|   \right]\\
	\equiv & Term_{1}(k) + Term_{2}(k),
	\end{align*}
	where $\zeta \sim N(0,1)$ and $t \mapsto s_{P}(t) = \max\{ t, 0 \} \times 1\{ E_{P}[Z]  = 0  \} + t \times 1\{ E_{P}[Z] > 0  \}$.

	We now provide a bound for terms $Term_{1}(k)$. For any $f \in LB$ and any $k \in \mathbb{N}$, the mapping $t \mapsto f_{k}(t) \equiv  f(\max\{ t, - \sqrt{k} E_{P}[Z] \})$ satisfies, for any $t \leq t'$, 
	\begin{align*}
	|f_{k}(t') - f_{k}(t)| \leq | \max\{ t', - \sqrt{k} E_{P}[Z] \} - \max\{ t, - \sqrt{k} E_{P}[Z] \} |
	\end{align*}
	where the RHS is equal to $0$ if $t \leq t' \leq - k E_{P}[Z]$, $t' - (- k E_{P}[Z]) \leq t'-t$; if  $t \leq  - k E_{P}[Z] \leq t'$; and $t'-t$ if $  - k E_{P}[Z] \leq t \leq t'$. Hence $|f_{k}(t') - f_{k}(t)| \leq |t'-t|$. The same inequality holds when $t' \leq t$, so $f_{k}$ is in $LB$. Therefore, 
	\begin{align*}
	Term_{1}(k) 	\leq  \sup_{f \in LB} \left| E_{\mathbf{P}} \left[ f\left( k^{-1/2} \sum_{i=1}^{k} (Z_{i} - E_{P}[Z])  \right)  -  f\left( \zeta  \right)       \right]    \right| \leq  6 k^{-1/2} E_{P}[|Z|^{3}]
	\end{align*}
	where the last line follows from Berry-Esseen Inequality for Lipschitz functions (see \cite{BarbourChen2005} Thm. 3.2 in Ch. 1).
	
	Regarding $Term_{2}(k)$, we note that if $E_{P}[Z] = 0$, then $Term_{2}(k) = 0$, because $s_{P}(\zeta) = \max\{\zeta,0\}$. So we only need a bound for $E_{P}[Z] > 0$. Under this condition, \begin{align*}
	Term_{2}(k) \leq   \sup_{f \in LB} E \left[ 1\{ \zeta \leq - \sqrt{k} E_{P}[Z]  \}\left| f \left( \max\{ \zeta , - \sqrt{k} E_{P}[Z]    \}   \right) -  f \left( \zeta   \right)        \right|   \right] .
	\end{align*} 
	Since $||f||_{L^{\infty}} \leq 1$, the inequality further implies that $Term_{2}(k) \leq 2 E \left[ 1\{ \zeta \leq - \sqrt{k} E_{P}[Z]  \} \right] = 2 \Phi(-\sqrt{k} E_{P}[Z])$.
\end{proof}

\begin{proof}[Proof of Proposition \ref{pro:boot-choice}]
	By the triangle inequality,	
	\begin{align*}
	|| \psi_{\tilde{k}_{n}}(P_{n}) - \psi_{n}(P) ||_{LB} \leq &  || \psi_{\tilde{k}_{n}}(P_{n}) - \psi(P) ||_{LB} + || \psi(P) - \psi_{n}(P) ||_{LB}  
	\end{align*}
	
	By Lemma \ref{lem:boot-bias}, $||\psi_{k}(P) - \psi(P)||_{LB} \leq 6 k^{-1/2} E_{P}[|Z|^{3}] + 1\{  E_{P}[Z] >0  \} 2 \Phi(-\sqrt{k}E_{P}[Z])$. Thus, we can invoke Theorem \ref{thm:rate-choice} and its corollary to show that $|| \psi_{\tilde{k}_{n}}(P_{n}) - \psi(P) ||_{LB} = O_{P} \left(  \inf_{k \in \{1,...,n\}} \{ l_{n} \sqrt{k} n^{-1/2} + 1\{  E_{P}[Z] >0  \} 2 \Phi(-\sqrt{k}E_{P}[Z])+ k^{-1/2} E_{P}[|Z|^{3}] \} \right)$. It is clear that the choice $k$ that achieves the infimum will diverge with $n$, so for this choice of $k$, $1\{  E_{P}[Z] >0  \} 2 \Phi(-\sqrt{k}E_{P}[Z])$ will eventually be dominated by $k^{-1/2} E_{P}[|Z|^{3}]$. Hence $|| \psi_{\tilde{k}_{n}}(P_{n}) - \psi(P) ||_{LB} = O_{P} \left(  \inf_{k \in \{1,...,n\}} \{ l_{n} \sqrt{k} n^{-1/2} + k^{-1/2} E_{P}[|Z|^{3}] \} \right)$
	
	The desired result follows from the fact that for sufficiently large $n$, $|| \psi(P) - \psi_{n}(P) ||_{LB} \leq 8 n^{-1/2} E_{P}[|Z|^{3}]$ and the fact that there exists a $C > 0$ such that $\inf_{k \in \{1,...,n\}} \{ l_{n} \sqrt{k} n^{-1/2} +  k^{-1/2} E_{P}[|Z|^{3}] \} \geq C n^{-1/2}$ for all $n$. 	
\end{proof}

\begin{proof}[Proof of Proposition \ref{pro:RME-cont}]
	Throughout, fix $k$ and $P,P'$ and let $||.||_{\Theta} \equiv ||.||_{L^{q}}$. Let $\Theta_{k}(M) \equiv \{ \theta \in \Theta_{k} \colon ||\theta - \psi_{k}(P)||_{\Theta} \geq M  \}$.  And, let \begin{align*}
		t \mapsto U_{k}(t) \equiv \inf_{\theta \in \Theta_{k}(t) }  \frac{ Q_{k}(P,\theta) - Q_{k}(P,\psi_{k}(P))}{t}.
	\end{align*}
	Towards the end of the proof we show that $U_{k}$ is continuous. Let $t \mapsto \Gamma_{k}(t) \equiv  \inf_{s \geq t} U_{k}(s)$; it follows that $\Gamma_{k} \leq U_{k}$, $\Gamma_{k}$ is non-decreasing and by the Theorem of the Maximum $\Gamma_{k}$ is continuous.

	We show that $||\psi_{k}(P) - \psi_{k}(P') ||_{\Theta} \geq M \equiv \Gamma^{-1}_{k}(d(P,P'))$ cannot occur.\footnote{Note that $\bar{U}_{k}^{-1}(t) \equiv \inf \{ s \colon \bar{U}_{k}(s) \geq t  \}$.} To do this, we show that $1\{ ||\psi_{k}(P) - \psi_{k}(P') ||_{\Theta} \geq M  \} = 0$. Observe that \begin{align*}
		1\{ ||\psi_{k}(P) - \psi_{k}(P') ||_{\Theta} \geq M  \}  = & 1\{ \cup_{j \in \mathbb{N}}  \{  2^{j} M \geq  ||\psi_{k}(P) - \psi_{k}(P') ||_{\Theta} \geq 2^{j-1}M  \} \}  \\
		\leq & \max_{j \in \mathbb{N}} 1\{   2^{j} M \geq  ||\psi_{k}(P) - \psi_{k}(P') ||_{\Theta} \geq 2^{j-1}M  \}.
	\end{align*}
	
	For each $(j,k) \in \mathbb{N}^{2}$, let $S_{j,k} \equiv \{ \theta \in \Theta_{k} \colon 2^{j} M \geq  ||\psi_{k}(P) - \theta ||_{\Theta} \geq 2^{j-1}M  \}$. It follows that, for any $j \in \mathbb{N}$,
	\begin{align*}
		1 \{  \psi_{k}(P') \in S_{j,k}   \} \leq 1 \left\{   \inf_{\theta \in S_{j,k}} Q(P',\theta) \leq Q(P',\psi_{k}(P))       \right\}
	\end{align*}
	because $\psi_{k}(P) \in \Theta_{k} \setminus S_{j,k}$. Observe that, for any $\theta \in S_{j,k} \cup \{ \psi_{k}(P) \} \subseteq \{ \theta \in \Theta_{k} \colon  ||\theta - \psi_{k}(P)||_{\Theta} \leq 2^{j}M \}$
	 \begin{align*}
		Q(P',\theta) - Q(P',\psi_{k}(P)) \geq & Q(P,\theta) - Q(P,\psi_{k}(P)) \\
		&  -  \left| Q(P',\theta) - Q(P',\psi_{k}(P)) - \{ Q(P,\theta) - Q(P,\psi_{k}(P)) \}    \right|\\
		\geq &  Q(P,\theta) - Q(P,\psi_{k}(P)) - 2^{j} M \varDelta_{j,k}(P,P') 
	\end{align*}
	where \begin{align*}
		\varDelta_{j,k}(P,P')  \equiv \sup_{\theta \in \Theta_{k} \colon  ||\theta - \psi_{k}(P)||_{\Theta} \leq 2^{j}M  }  \frac{\left| Q(P',\theta) - Q(P',\psi_{k}(P)) - \{ Q(P,\theta) - Q(P,\psi_{k}(P)) \}    \right|}{||\theta - \psi_{k}(P)||_{\Theta} }.
	\end{align*}
	
	Hence,\begin{align*}
		1 \{  \psi_{k}(P') \in S_{j,k}   \} \leq & 1 \left\{   \inf_{ \theta \in \Theta_{k} \colon  ||\theta - \psi_{k}(P)||_{\Theta} \geq 2^{j-1}M  }   \frac{Q(P,\theta) - Q(P,\psi_{k}(P)) }{2^{j-1}M} \leq  0.5  \varDelta_{j,k}(P,P')       \right\} \\
		\leq & 1 \left\{   \inf_{ \theta \in \Theta_{k} \colon  ||\theta - \psi_{k}(P)||_{\Theta} \geq 2^{j-1}M  }   \frac{Q(P,\theta) - Q(P,\psi_{k}(P)) }{2^{j-1}M} \leq  0.5  \max_{k \in \mathbb{N}} \varDelta_{\infty,k}(P,P')       \right\} \\
		\leq & 1 \left\{  \Gamma_{k}(2^{j-1}M) \leq  0.5  \max_{k \in \mathbb{N}} \varDelta_{\infty,k}(P,P')       \right\}.
	\end{align*}
	
	Since $\bar{U}_{k}$ is non-decreasing, the previous display implies that \begin{align*}
		1 \{  \psi_{k}(P') \in S_{j,k}   \} \leq  1 \left\{  2^{j-1}M \leq  \Gamma_{k}^{-1}(0.5  \max_{k \in \mathbb{N}} \varDelta_{\infty,k}(P,P') )      \right\}
	\end{align*}
	which equals zero by the definition of $M$, the fact that $\Gamma_{k}^{-1}$ is non-decreasing and $2^{j-1} \geq 1$.
      
   \medskip 
	
	We now show that  $t \mapsto U_{k}(t)$ (and thus $\Gamma_{k}$) is continuous. Consider the problem $\inf_{\theta \in \Theta_{k}(M)} Q_{k}(P,\theta)$, and consider the set $L_{k}(M) \equiv  \{ \theta \in \Theta_{k}(M) \colon Pen(\theta) \leq \lambda_{k}^{-1} Q(P,\theta_{k})    \}$ for some (any) $\theta_{k} \in \Theta_{k}(M)$ which is non-empty and close. To solve the former minimization problem it suffices to solve   $\inf_{\theta \in L_{k}(M)} Q_{k}(P,\theta)$, because the minimum value cannot be outside $L_{k}(M)$. Because $Pen$ is lower-semi-compact, $L_{k}(M)$ is compact (a closed subset of a compact set) so this and lower-semi-continuity of $Q_{k}(P,\cdot)$ ensures that $\inf_{\theta \in L_{k}(M)} Q_{k}(P,\theta)$ is achieved by an element in $L_{k}(M)$ and the same is true for the original $V_{k}(M) \equiv \inf_{\theta \in \Theta_{k}(M)} Q_{k}(P,\theta)$. We just showed that the correspondence $M \mapsto L_{k}(M)$ is compact-valued, it is also continuous. By virtue of the Theorem of the Maximum, $V_{k}$ is continuous; it is also non-decreasing. The function $t \mapsto U_{k}(t) = V_{k}(t)/t$ is also continuous in $t>0$.
\end{proof}

\section{Appendix for Section \ref{sec:ALR}}
\label{app:ALR}

\begin{proof}[Proof of Theorem \ref{thm:W-ALR}]
    We first show the desired result for a fixed $k$, i.e., $\mathbf{k}(n) = k$ for any $n \in \mathbb{N}$.

    Let $(z,k) \mapsto \varphi_{k}(P) \equiv D\psi_{k}(P)[\delta_{z} - P]$ which is well-defined because $\delta_{z} - P \in \mathcal{T}_{P}$. 	
    We now show that $\varphi_{k}(P) \in L^{2}_{0}(P)$. The fact that has mean zero (provided it exists) is trivial, so we only show that $\int |\varphi_{k}(P) (z)|^{2} P(dz) < \infty$. The topology is locally convex and thus generated by a family of semi-norms. Suppose there exists a $L<\infty$ such that $|D \psi_{k}(P)[\delta_{z} - P]| \leq \rho(\delta_{z}-P)$ where $\rho$ is a member of the family. Because the topology $\tau$ is assumed to be dominated by $||.||_{TV}$ it follows that  $\rho(\delta_{z}-P) \leq C ||\delta_{z} - P ||_{TV} \leq 2 C$ for some finite $C$ for any $z \in \mathbb{Z}$. And thus $\int |\varphi_{k}(P) (z)|^{2} P(dz) \leq 2CL < \infty$ as desired.
    
    We now show that there exists a member of the family of semi-norms, $\rho$, and a $L < \infty $ such that  $|D \psi_{k}(P)[Q]| \leq L \rho(Q)$ for all $Q \in ca(\mathbb{Z})$. Suppose not, that is, for any $R>0$ and any $\rho$, there exists a $Q$ such that $\rho(Q) = 1$ and  $|D \psi_{k}(P)[Q]| > R$. Since $D\psi_{k}(P)$ is continuous with respect to $\tau$, there exists a member, $\rho$, of the family of semi-norms such that for any $\epsilon>0$ there exists $\delta>0$ such that if $Q$ is such that $\rho(Q) \leq \delta$,  then $|D \psi_{k}(P)[Q]| < \epsilon$.  Let $R = \epsilon/\delta$. There exists a $Q$ such that $\rho(Q)=1$ and $\delta |D \psi_{k}(P)[Q]| > \epsilon$. Let $\nu = \delta Q$, then $\rho(\nu) = \delta$ but $\delta |D \psi_{k}(P)[Q]| = |D \psi_{k}(P)[\delta Q]| =  |D \psi_{k}(P)[\nu]| >\epsilon $ but this is a contradiction.

    We now show that $\eta_{k}(P_{n} - P) = o_{P}(n^{-1/2})$ for each $k \in \mathbb{K}$. 
    Let $n \mapsto \mathbb{G}_{n} \equiv \sqrt{n}(P_{n} - P)$.  It follows that,  a.s.-$\mathbf{P}$, $t \mapsto P + t \mathbb{G}_{n}$ is a valid curve because $\mathbb{G}_{n} \in \mathcal{T}_{P}$ a.s.-$\mathbf{P}$.

    Fix any $\epsilon>0$ and let $U \in \mathcal{C}$ be as in the condition of the statement of the theorem. Then, letting $D_{n} \equiv \{  \boldsymbol{z} \in \mathbb{Z}^{\infty} \colon \mathbb{G}_{n}(\boldsymbol{z})  \in U  \}$, it follows that  
    \begin{align*}
    \mathbf{P}\left(   \sqrt{n} |\eta_{k}(P_{n} - P)|  \geq \epsilon     \right) \leq 	\mathbf{P}\left(  \sqrt{n}  |\eta_{k}(P_{n} - P)|  \geq \epsilon     \mid D_{n} \right) + \mathbf{P}(D^{C}_{n}). 
    \end{align*}
    The second term in the RHS is less than $\epsilon$ for all $n \geq N$. Regarding the first term in the RHS it follows that, over $D_{n}$, 
    \begin{align*}
    |\eta_{k}\left(  P_{n} - P  \right)| /t_{n} \leq  \sup_{Q \in U} |\eta_{k}\left(  t_{n} Q  \right)| /t_{n}
    \end{align*}
    where $t_{n} \equiv 1/\sqrt{n}$. Thus, by definition of differentiability, the first term in the RHS also vanishes as $n \rightarrow \infty$. So the desired result follows. 
    
    Therefore, for any $k \in \mathbb{K}$
    \begin{align*}
    	\frac{1}{||\varphi_{k}(P)||_{L^{2}(P)}} \left| \sqrt{n}(\psi_{k}(P_{n}) - \psi_{k}(P))  - n^{-1} \sum_{i=1}^{n} D \psi_{k}(P)[\delta_{Z_{i}} - P] \right| = 	\frac{\sqrt{n}|\eta_{k}(P_{n} - P)|}{||\varphi_{k}(P)||_{L^{2}(P)}}
    \end{align*}
     and $\frac{\sqrt{n}|\eta_{k}(P_{n} - P)|}{||\varphi_{k}(P)||_{L^{2}(P)}} = o_{P}(1)$, as desired.

    \bigskip
    
    We now shows existence of a diverging sequence by using the first part and the diagonalization lemma \ref{lem:sub-seq-conv}.
    
    For any $\epsilon>0$, $k \in \mathbb{N}$ and $n \in \mathbb{N}$, let  $T(\epsilon,k,n)\equiv \mathbf{P}\left( \sqrt{n} \frac{|\eta_{k}(P_{n}(\boldsymbol{z})-P)|}{||\varphi_{k}(P)||_{L^{2}(P)}} \geq \epsilon  \right)$. To show the desired result it suffices to show that there exists a non-decreasing diverging sequence $(j(n))_{n}$ such that for all $\epsilon>0$, there exists a $\bar{N}$ such that
    \begin{align*}
    T(\epsilon,j(n),n) \leq \epsilon,
    \end{align*}
    for all $n \geq \bar{N}$.
    
    We shows that, for any $k \in \mathbb{K}$,  $\lim_{n \rightarrow \infty} T(2^{-k},k,n) =0$. 
    By Lemma \ref{lem:sub-seq-conv}, there exists a non-decreasing diverging sequence $(j(n))_{n \in \mathbb{N}}$ such that $\lim_{n \rightarrow \infty} T(2^{-j(n)},j(n),n) =0$; i.e., for any $\epsilon>0$, there exists a $N(\epsilon)$ such that $T(2^{-j(n)},j(n),n) \leq \epsilon$ for all $n\geq N(\epsilon)$. 
    
    Since $j(.)$ diverges, there exists a $N_{\epsilon}$ such that $1/2^{j(n)} \leq \epsilon$ for all $n \geq N_{\epsilon}$.  By these observations and the fact that $\epsilon \mapsto T(\epsilon,k,n)$ is non-increasing,
    \begin{align*}
    T(\epsilon,j(n),n) \leq T(2^{-j(n)},j(n),n)  \leq \epsilon
    \end{align*}
    for all  $n \geq \bar{N}_{\epsilon} \equiv \max\{ N_{\epsilon} , N(\epsilon)  \}$, and we thus showed the desired result.	
    
\end{proof}


The following result is a well-known representation result (see \cite{VdV-W1996}) and is stated merely for convenience. 

\begin{lemma}\label{lem:Skohorod}
	Let $\boldsymbol{z} \mapsto \mathbb{G}_{n}(\boldsymbol{z}) \equiv \sqrt{n}(P_{n}(\boldsymbol{z}) - P)$. Suppose $\mathcal{S}$ is P-Donsker. Then there exists a tight Borel measurable $\mathbb{G} \in L^{\infty}(\mathcal{S})$ such that for any $\epsilon>0$, there exists a Borel set $A \subseteq \mathbb{Z}^{\infty}$ such that $\mathbf{P}(A) \geq 1-\epsilon$ and $||\mathbb{G}_{n}(\boldsymbol{z}) - \mathbb{G}||_{\mathcal{S}} = o(1)$ for all $\boldsymbol{z} \in A$.  
\end{lemma}

In the following proof, almost uniformly means that for any $\epsilon>0$, there exists a Borel set $A \in \tilde{\mathbb{Z}^{\infty}}$ such that $\tilde{\boldsymbol{P}}(A) \geq 1-\epsilon$ and $\sup_{\tau \in A }||\tilde{\mathbb{G}}_{n}(\tau) - \tilde{\mathbb{G}}||_{L^{\infty}(\mathcal{S})} = o(1)$.

\begin{proof}[Proof of Lemma \ref{lem:Skohorod}]
	It is well-known that the following representation is also valid: $ \mathbb{G}_{n} : \mathbb{Z}^{\infty} \rightarrow L^{\infty}(\mathcal{S})$. Since $\mathcal{S}$ is a Donsker Class, $\mathbb{G}_{n}$ converges weakly to some $\mathbb{G}$ tight Borel measurable element in $L^{\infty}(\mathcal{S})$ (e.g. see \cite{VdV-W1996} Ch. 2.1). By Theorem 1.10.3 in \cite{VdV-W1996} there exists a probability space $(\tilde{\mathbb{Z}^{\infty}},\tilde{\boldsymbol{P}})$ and a sequence of maps $\tilde{\mathbb{G}}_{n} : \tilde{\mathbb{Z}^{\infty}} \rightarrow L^{\infty}(\mathcal{S})$ for all $n \in \mathbb{N}$ and $\tilde{\mathbb{G}} : \tilde{\mathbb{Z}^{\infty}} \rightarrow L^{\infty}(\mathcal{S})$ such that (i) $||\tilde{\mathbb{G}}_{n} - \tilde{\mathbb{G}}||_{L^{\infty}(\mathcal{S})} = o(1)$  almost uniformly; and (ii) $\tilde{\mathbb{G}}_{n}$ and $\tilde{\mathbb{G}}$ have the same law as $\mathbb{G}_{n}$ and $\mathbb{G}$ resp. 
\end{proof}

\subsection{Appendix for Example \ref{exa:ipdf2-diff}}
\label{app:ipdf2-diff}

\begin{lemma}\label{lem:ipdf2-boundIF}
	For all $k \in \mathbb{N}$ and $P \in \mathcal{M}$, 
	\begin{align*}
	||\varphi_{k}(P)||_{L^{2}(P)} \leq ||p||^{2}_{L^{\infty}(\mathbb{R})}   ||\kappa ||^{2}_{L^{1}(\mathbb{R})}.
	\end{align*}
\end{lemma}

\begin{proof}
	It suffices to show that $E_{P}[|(\kappa_{k} \star  P)(Z)|^{2} ] \leq ||p||^{2}_{L^{\infty}}  \left(  \int  |\kappa(u)| du  \right)^{2}$. To do this, note that 
	\begin{align*}
	E_{P}[|(\kappa_{k} \star  P)(Z)|^{2} ]  = & \int \left(  \int  k \kappa((x-z)k) p(x) dx  \right)^{2} p(z) dx \\
	= & \int \left(  \int  \kappa(u) p(z + u/k) du  \right)^{2} p(z) dz \\
	\leq & ||p||^{2}_{L^{\infty}}  \left(  \int  |\kappa(u)| du  \right)^{2}.
	\end{align*} 
\end{proof}

\begin{proof}[Proof of Proposition \ref{pro:ipdf2-diff-1}]
	Consider the curve $t \mapsto P + tQ$ for any $Q \in ca(\mathbb{R})$. It is a valid curve because $\mathbb{D}_{\psi} = ca(\mathbb{R})$. Therefore
	\begin{align*}
	\psi_{k}(P+tQ) - \psi_{k}(P) =& t \left\{  \int (\kappa_{k} \star Q)(x) P(dx) + \int (\kappa_{k} \star P)(x) Q(dx) \right\} \\
	& + t^{2} \int (\kappa_{k} \star Q)(x) Q(dx).
	\end{align*}
	Since $\kappa$ is symmetric, $\int (\kappa_{k} \star P)(x) Q(dx) = \int (\kappa_{k} \star Q)(x) P(dx)$. From this display, $Q \mapsto \eta_{k}(Q) =   \int (\kappa_{k} \star Q)(x) Q(dx)$ and $Q \mapsto D \psi_{k}(P)[Q]=2\int (\kappa_{k} \star P)(x) Q(dx)$.
	
	The mapping $Q \mapsto 2\int (\kappa_{k} \star P)(x) Q(dx)$ is clearly linear. Also, note that for any reals $x$ and $x'$
	\begin{align*}
	|(\kappa_{k} \star P)(x) - (\kappa_{k} \star P)(x')| = & k \left| \int p(y) (\kappa(k(y-x))dy - \kappa(k(y-x')) ) dy  \right| \\
	\leq &  C k^{2} |x-x'|
	\end{align*}
	for some $C < \infty$. The last inequality follows from the fact that $x \mapsto \kappa(x)$ is smooth. Therefore $x \mapsto (k^{2}/C)(\kappa_{k} \star P)(x)$ is in LB. This implies that for any $Q'$ and $Q$ in $ca(\mathbb{Z})$,
	\begin{align*}
	\left| \int (\kappa_{k} \star P)(x) Q'(dx) - \int (\kappa_{k} \star P)(x) Q(dx)  \right| \leq k^{2} C ||Q'-Q||_{LB}
	\end{align*}
	and thus $Q \mapsto  2\int (\kappa_{k} \star P)(x) Q(dx)$ is $||.||_{LB}$-continuous.

	We now bound $Q \mapsto \eta_{k}(Q)$. For any $x' > x$ and any $Q \in \mathcal{T}_{P}$, it follows that
	\begin{align*}
	\kappa_{k} \star Q (x) - \kappa_{k} \star Q (x')  = & \int k (\kappa( k (y-x) ) - \kappa( k (y-x') ) ) Q(dy) \\
	= & \int k^{2} \int_{x}^{x'} \kappa'(y- t) dt  Q(dy) \\
	= & k^{2} \int_{x}^{x'}  \int \kappa'(y-t) Q(dy)
	\end{align*}
	where the last line follows from the fact that $t \mapsto \kappa(t)$ is bounded. Since $\kappa$ is smooth, $y \mapsto \kappa'(y-t)$ is Lipschitz (with some constant $L$) for any $t \in \mathbb{R}$. Hence 
	\begin{align*}
	|	\kappa_{k} \star Q (x) - \kappa_{k} \star Q (x') | \leq L |x'-x|  ||Q||_{LB}.	\end{align*} 
	Thus, the mapping $x \mapsto (\kappa_{k} \star Q) (x) $ is bounded and Lipschitz with constant $L ||Q||_{LB}$. Therefore,
	\begin{align*}
	|\eta_{k}(Q)|  \leq L ||Q||^{2}_{LB}. 
	\end{align*}	
\end{proof}

\subsection{Appendix for Examples \ref{exa:NPIV-diff-sieve} and \ref{exa:NPIV-diff-penal}}
\label{app:NPIV-diff}

First, note that $\mathbb{D}_{\psi} \subseteq ca(\mathbb{R} \times [0,1]^{2})$ (defined in Appendix \ref{app:T-reg-suff}). For any $k \in \mathbb{N}$, let $F_{k} : L^{2}([0,1]) \times \mathbb{D}_{\psi} \rightarrow   L^{2}([0,1])$ be such that
\begin{align*}
F_{k}(\theta,Q) \equiv \left\{   
\begin{array}{cc}
(T^{\ast}_{k,Q}T_{k,Q} + \lambda_{k} I)[\theta] - T_{k,Q}^{\ast}[r_{k,Q}] & for~Penalization-Based \\
(\Pi^{\ast}_{k} T^{\ast}_{k,Q}T_{k,Q} \Pi_{k})[\theta] - \Pi_{k}^{\ast} T_{k,Q}^{\ast}[r_{k,Q}] & for~Sieve-Based 	
\end{array}
\right.
\end{align*}
for any $(\theta,Q) \in L^{2}([0,1]) \times  \mathbb{D}_{\psi}$. Note that for any $Q \in \mathbb{D}_{\psi}$ the integrals defining the operators are well-defined by assumptions and so is $T^{\ast}_{k,Q}$; see Appendix \ref{app:T-reg-suff} for a discussion.

Let $\varepsilon_{P}(Y,W) \equiv  Y - \psi(P)(W)$ and $\varepsilon_{k,P}(Y,W) \equiv  Y - \psi_{k}(P)(W)$. Also, throughout this section we use the notation introduced in Appendix \ref{app:T-reg-suff} to denote $T_{k,P}$ and other quantities.

\begin{lemma}\label{lem:psik-Fdiff}
	For any $P \in \mathbb{D}_{\psi}$ and any $k \in \mathbb{N}$, $\psi_{k}$ is $||.||_{LB}$-Frechet differentiable tangential to $\mathbb{D}_{\psi}$ at $P$ with derivative given by:
	
	(1) For the Penalization-Based: 
	\begin{align*}
	D\psi_{k}(P)[Q] 	= & (T_{k,P}^{\ast} T_{k,P} + \lambda_{k} I )^{-1} T^{\ast}_{k,P} \mathbb{T}_{k,Q}[\varepsilon_{k,P}] \\
	& - (T_{k,P}^{\ast} T_{k,P} + \lambda_{k} I )^{-1} T^{\ast}_{k,Q}T_{k,P}[\psi_{k}(P) - \psi(P)],~\forall Q \in  \mathbb{D}_{\psi}
	\end{align*}
	where $x \mapsto \mathbb{T}_{k,P}[g](x) \equiv \int \kappa_{k}(x'-\cdot) \int (\psi_{k}(P)(w) - y)Q(dy,dw,dx')$. \\
	
	(2) For the Sieve-Based: 
	\begin{align*}
	D\psi_{k}(P)[Q] 	= & (\Pi^{\ast}_{k} T_{k,P}^{\ast} T_{k,P} \Pi_{k} )^{-1} \Pi^{\ast}_{k}  T^{\ast}_{k,P} \mathbb{T}_{k,Q}[\varepsilon_{k,P}] \\
	& - (\Pi^{\ast}_{k}  T_{k,P}^{\ast} T_{k,P} \Pi_{k}  )^{-1}  \Pi^{\ast}_{k} T^{\ast}_{k,Q}T_{k,P}[\psi_{k}(P) - \psi(P)],~\forall Q \in  \mathbb{D}_{\psi}
	\end{align*}
	where $x \mapsto \mathbb{T}_{k,P}[g](x) \equiv (u^{J(k)}(x))^{T} Q_{uu}^{-1} E_{P}[u^{J(k)}(X) g(Y,W)]$ for any $ g \in L^{2}(P)$.  
\end{lemma}

\begin{proof}
	See the end of this Section.
\end{proof}

%
%
%

The following corollary trivially follows.
\begin{corollary} \label{cor:gammak-diff}
	For the sieve-based and the penalization-based: For any $ P \in \mathbb{D}_{\psi}$\\
	
	(1) The regularization $\boldsymbol{\gamma}$ is $DIFF(P,\mathcal{E}_{||.||_{LB}})$.\\
	
	(2)  For each $k \in \mathbb{N}$, the reminder of $\gamma_{k}$, $\eta_{k}$, is such that $ |\eta_{k}(\zeta)| = o(||\zeta||_{LB})$, for any $\zeta \in \mathbb{D}_{\psi}$.\footnote{The ``o" function may depend on $k$.} 
\end{corollary}

\begin{proof}
	See the end of this Section.
\end{proof}

\begin{proof}[Proof of Proposition \ref{pro:NPIV-DIFF}]
	The result follows from Corollary \ref{cor:gammak-diff}. Lemma \ref{lem:psik-Fdiff}(2) derives the expression for $D \psi_{k}(P)$; we now expand this expression in terms of the basis functions. 
	
	For any $g,f \in L^{2}([0,1])$, 
	\begin{align*}
	T_{k,P} \Pi_{k}[g](x) = & T_{k,P}\left[ (v^{L(k)}(.))^{T} Q^{-1}_{vv} E_{Leb}[v^{L(k)}(W)g(W)]     \right](x) \\
	= &  (u^{J(k)}(x))^{T} Q^{-1}_{uu} Q_{uv} Q^{-1}_{vv} E_{Leb}[v^{L(k)}(W)g(W)],
	\end{align*}
	and 
	\begin{align*}
	\langle T_{k,P} \Pi_{k}[g],f \rangle _{L^{2}([0,1])} = & \int (u^{J(k)}(x))^{T} Q^{-1}_{uu} Q_{uv} Q^{-1}_{vv} E_{Leb}[v^{L(k)}(W)g(W)] f(x) dx \\
	= & \int E_{Leb}[(u^{J(k)}(X))^{T}f(x)] Q^{-1}_{uu} Q_{uv} Q^{-1}_{vv} v^{L(k)}(w) g(w) dw  
	\end{align*}
	so $\Pi_{k}^{\ast} T^{\ast}_{k,P} : L^{2}([0,1],Leb) \rightarrow L^{2}([0,1],Leb)$ and is given by
	\begin{align*}
	f \mapsto \Pi_{k}^{\ast} T^{\ast}_{k,P}[f](.)  =E_{Leb}[(u^{J(k)}(X))^{T}f(X)] Q^{-1}_{uu} Q_{uv} Q^{-1}_{vv} v^{L(k)}(.).
	\end{align*} 
	Hence
	\begin{align*}
	\Pi_{k}^{\ast} T^{\ast}_{k,P} T_{k,P} \Pi_{k} [g](.) = (v^{L(k)}(.))^{T} Q^{-1}_{vv} Q_{uv}^{T} Q_{uu}^{-1}  Q_{uv} Q^{-1}_{vv} E_{Leb}[v^{L(k)}(W)g(W)]. 
	\end{align*} 
	We now compute the inverse of this operator. Consider solving for $g(.) = (v^{L(k)}(.))^{T} Q^{-1}_{vv} b$ for some $b \in \mathbb{R}^{k}$  such that \begin{align*}
	& \Pi_{k}^{\ast} T^{\ast}_{k,P} T_{k,P} \Pi_{k} [g](.) = (v^{L(k)}(.))^{T} Q^{-1}_{vv} b \\
	\iff &  Q_{uv}^{T} Q_{uu}^{-1}  Q_{uv} Q^{-1}_{vv} E_{Leb}[v^{L(k)}(W)g(W)] = b\\
	\iff & E_{Leb}[v^{L(k)}(W)g(W)]  = Q_{vv} (Q_{uv}^{T} Q_{uu}^{-1}  Q_{uv})^{-1} b.
	\end{align*}
	Hence, $(\Pi_{k}^{\ast} T^{\ast}_{k,P} T_{k,P} \Pi_{k} )^{-1}[g](.) = (v^{L(k)}(.))^{T}  (Q_{uv}^{T} Q_{uu}^{-1}  Q_{uv})^{-1} b $. Therefore 
	\begin{align*}
	(\Pi^{\ast}_{k} T_{k,P}^{\ast} T_{k,P} \Pi_{k} )^{-1} \Pi^{\ast}_{k}  T^{\ast}_{k,P} \mathbb{T}_{k,Q}[\varepsilon_{k,P}] = &   (v^{L(k)}(.))^{T}  (Q_{uv}^{T} Q_{uu}^{-1}  Q_{uv})^{-1} Q_{uv} Q_{uu}^{-1} E_{Leb}[u^{J(k)}(X) \mathbb{T}_{k,Q}[\varepsilon_{k,P}] ]\\
	= & (v^{L(k)}(.))^{T}  (Q_{uv}^{T} Q_{uu}^{-1}  Q_{uv})^{-1} Q_{uv} Q_{uu}^{-1} E_{Q}[u^{J(k)}(X) \varepsilon_{k,P}(Y,W)].
	\end{align*}
	And 
	\begin{align*}
	& (\Pi^{\ast}_{k} T_{k,P}^{\ast} T_{k,P} \Pi_{k} )^{-1} \Pi^{\ast}_{k}  T^{\ast}_{k,Q} T_{k,P}[\psi_{k}(P) - \psi_{id}(P)]\\
	= &   (v^{L(k)}(.))^{T}  (Q_{uv}^{T} Q_{uu}^{-1}  Q_{uv})^{-1} E_{Leb}[v^{L(k)}(W) \Pi^{\ast}_{k} T^{\ast}_{k,Q} T_{k,P}[\psi_{k}(P) - \psi_{id}(P)](W) ].
	\end{align*}
	
	It is easy to see that for any $Q$, $D \gamma_{k}(P)[Q] =   \int \pi(w) D \psi_{k}(P)(w)[Q] dw$, the goal is to cast this as $\int  D\psi^{\ast}_{k}(P)[\pi](z) Q(dz) $. To this end, note that
	\begin{align*}
	& \int \pi(w) D \psi_{k}(P)(w)[Q] dw \\
	= &  \int \pi(w) (v^{L(k)}(w))^{T}  (Q_{uv}^{T} Q_{uu}^{-1}  Q_{uv})^{-1} Q_{uv} Q_{uu}^{-1} E_{Q}[u^{J(k)}(X) \varepsilon_{k,P}(Y,W)] dw \\
	& - \int \pi(w)  (v^{L(k)}(w))^{T}  (Q_{uv}^{T} Q_{uu}^{-1}  Q_{uv})^{-1} E_{Leb}[v^{L(k)}(W) \Pi_{k}^{\ast}  T^{\ast}_{k,Q} T_{k,P}[\psi_{k}(P) - \psi_{id}(P)](W) ] dw\\
	\equiv & Term_{1,k} + Term_{2,k}.
	\end{align*}  
	Regarding the first term, note that \begin{align*}
	Term_{1,k} =  \int E_{Leb}[\pi(W) (v^{L(k)}(W))^{T}]  (Q_{uv}^{T} Q_{uu}^{-1}  Q_{uv})^{-1} Q_{uv} Q_{uu}^{-1} u^{J(k)}(x) \varepsilon_{k,P}(y,w) Q(dy,dw,dx).
	\end{align*}
	
	We can cast $Term_{2,k} = - \langle \Pi_{k}[\pi] , (\Pi_{k}^{\ast} T^{\ast}_{k,P} T_{k,P} \Pi_{k} )^{-1}\left[ \Pi_{k}^{\ast} T^{\ast}_{k,Q} T_{k,P}[\psi_{k}(P) - \psi_{id}(P)]  \right] \rangle_{L^{2}}$, and thus 
	\begin{align*}
	Term_{2,k} = & - \langle T_{k,Q}(\Pi_{k}^{\ast} T^{\ast}_{k,P} T_{k,P} \Pi_{k} )^{-1}\Pi_{k}[\pi] ,  T_{k,P}[\psi_{k}(P) - \psi_{id}(P)]   \rangle_{L^{2}} \\
	= & -  \int (u^{J(k)}(x))^{T} Q_{uu}^{-1} E_{Q}[u^{J(k)}(X)(\Pi_{k}^{\ast} T^{\ast}_{k,P} T_{k,P} \Pi_{k} )^{-1}\Pi_{k}[\pi](W) ] T_{k,P}[\psi_{k}(P) - \psi_{id}(P)](x) dx \\
	= &    \int E_{P}[(\psi_{id}(P)(W) - \psi_{k}(P)(W)) (u^{J(k)}(X))^{T}] Q_{uu}^{-1} u^{J(k)}(x) (\Pi_{k}^{\ast} T^{\ast}_{k,P} T_{k,P} \Pi_{k} )^{-1}\Pi_{k}[\pi](w) \\
	& \times Q(dw,dx)
	\end{align*}
	where the second line follows from definition of $T_{k,P}$.
	
	Therefore, \begin{align*}
	D\psi^{\ast}_{k}(P)[\pi](y,w,x) 	= & E_{Leb}[\pi(W) (v^{L(k)}(W))^{T}]  (Q_{uv}^{T} Q_{uu}^{-1}  Q_{uv})^{-1} Q_{uv} Q_{uu}^{-1} u^{J(k)}(x) \varepsilon_{k,P}(y,w)\\
	& +  E_{P}[(\psi_{id}(P)(W) - \psi_{k}(P)(W)) (u^{J(k)}(X))^{T}] Q_{uu}^{-1} u^{J(k)}(x)\\
	& \times (v^{L(k)}(w))^{T} (Q_{uv}^{T} Q_{uu}^{-1}  Q_{uv})^{-1} E_{Leb}[ v^{L(k)}(W) \pi(W)]. 
	\end{align*}
	In the operator notation this expression equals 
	\begin{align}\notag
	D\psi^{\ast}_{k}(P)[\pi](y,w,x) 	= & T_{k,P}(\Pi_{k}^{\ast} T^{\ast}_{k,P} T_{k,P} \Pi_{k} )^{-1}\Pi_{k}[\pi](x) \varepsilon_{k,P}(y,w)\\ \label{eqn:proof-NPIV-Dpsi*}
	& +  T_{k,P}[\psi_{id}(P) - \psi_{k}(P)](x) \times (\Pi_{k}^{\ast} T^{\ast}_{k,P} T_{k,P} \Pi_{k} )^{-1}\Pi_{k}[\pi](w).
	\end{align}
\end{proof}

\begin{proof}[Proof of Proposition \ref{pro:NPIV-diff-penalty}]
		The result follows from Corollary \ref{cor:gammak-diff}. Lemma \ref{lem:psik-Fdiff}(1) derives the expression for $D \psi_{k}(P)$; we now expand this expression in terms of the basis functions.
	
	Note that $\psi_{k}(P) = (T_{k,P}^{\ast} T_{k,P} + \lambda_{k} I )^{-1}  T_{k,P}^{\ast} r_{k,P}$ and $r_{k,P}(.) = \int \kappa_{k}(x'-\cdot) \int \psi_{id}(P)(w) p(w,x')dwdx'$ so that $\psi_{k}(P) = (T_{k,P}^{\ast} T_{k,P} + \lambda_{k} I )^{-1}  T_{k,P}^{\ast} T_{k,P}[\psi_{id}(P)]$. Hence
	\begin{align*}
	&(T_{k,P}^{\ast} T_{k,P} + \lambda_{k} I )^{-1} T^{\ast}_{k,Q}T_{k,P}[\psi_{k}(P) - \psi_{id}(P)] \\
	= & (T_{k,P}^{\ast} T_{k,P} + \lambda_{k} I )^{-1} T^{\ast}_{k,Q}T_{k,P}[((T_{k,P}^{\ast} T_{k,P} + \lambda_{k} I )^{-1}  T_{k,P}^{\ast} T_{k,P} - I) \psi_{id}(P)]\\
	= & - \lambda_{k} (T_{k,P}^{\ast} T_{k,P} + \lambda_{k} I )^{-1} T^{\ast}_{k,Q}T_{k,P} (T_{k,P}^{\ast} T_{k,P} + \lambda_{k} I )^{-1} [\psi_{id}(P)].
	\end{align*}
	
	Thus\begin{align*}
	D\gamma_{k}(P)[Q] = & \langle \pi ,  \mathcal{R}_{k,P} T^{\ast}_{k,P} \mathbb{T}_{k,Q}[\varepsilon_{k,P}]\rangle_{L^{2}} \\
	& + \lambda_{k} \langle \pi, \mathcal{R}_{k,P} T^{\ast}_{k,Q}T_{k,P} \mathcal{R}_{k,P} [\psi_{id}(P)] \rangle_{L^{2}}\\
	= & \langle T_{k,P}\mathcal{R}_{k,P}[\pi] ,   \mathbb{T}_{k,Q}[\varepsilon_{k,P}]\rangle_{L^{2}} \\
	& + \lambda_{k} \langle  T_{k,Q}\mathcal{R}_{k,P}[\pi], T_{k,P} \mathcal{R}_{k,P} [\psi_{id}(P)] \rangle_{L^{2}}.
	\end{align*}
	
	Note that 
	\begin{align*}
	T_{k,P}[g](x) = \int \kappa_{k}(x'-x) \int g(w) p(w,x')dw,dx' = \int \kappa_{k}(x'-x) T_{P}[g](x') dx' = \mathcal{K}_{k}[T_{P}[g]](x)
	\end{align*}
	and by symmetry of $\kappa$, $T^{\ast}_{k,P} = T_{P}[\mathcal{K}_{k}]$, where $\mathcal{K}_{k}$ is simply the convolution operator.

	Therefore, 	 
	\begin{align*}
	\langle T_{k,P}\mathcal{R}_{k,P}[\pi] ,   \mathbb{T}_{k,Q}[\varepsilon_{k,P}]\rangle_{L^{2}} = & \int T_{k,P}\mathcal{R}_{k,P}[\pi](x) \int \kappa_{k}(x'-x) \int \varepsilon_{k,P}(y,w) Q(dy,dw,dx') dx \\
	= & \int \left( \int \kappa_{k}(x-x')  T_{k,P}\mathcal{R}_{k,P}[\pi](x) dx \right) \varepsilon_{k,P}(y,w) Q(dy,dw,dx') \\
	= & \int  \mathcal{K}_{k}^{2}T_{P}\mathcal{R}_{k,P}[\pi](x) \varepsilon_{k,P}(y,w) Q(dy,dw,dx).
	\end{align*}
	and 
	\begin{align*}
	\langle  T_{k,Q}\mathcal{R}_{k,P}[\pi], T_{k,P} \mathcal{R}_{k,P} [\psi_{id}(P)] \rangle_{L^{2}} = & \int \int \kappa_{k}(x'-x) \int \mathcal{R}_{k,P}[\pi](w) Q(dw,dx') T_{k,P} \mathcal{R}_{k,P} [\psi_{id}(P)] (x) dx\\
	= & \int   \mathcal{R}_{k,P}[\pi](w) \left( \int \kappa_{k}(x'-x) T_{k,P} \mathcal{R}_{k,P} [\psi_{id}(P)] (x) dx \right) Q(dw,dx')  \\
	= & \int   \mathcal{R}_{k,P}[\pi](w) \mathcal{K}^{2}_{k} T_{P} \mathcal{R}_{k,P} [\psi_{id}(P)] (x)  Q(dw,dx').
	\end{align*}	
	
	Therefore
	\begin{align*}
	D \psi_{k}^{\ast}(P)[\pi](y,w,x) = & \mathcal{K}_{k}^{2}T_{P}\mathcal{R}_{k,P}[\pi](x) \varepsilon_{k,P}(y,w) + \lambda_{k}  \mathcal{R}_{k,P}[\pi](w) \mathcal{K}^{2}_{k} T_{P} \mathcal{R}_{k,P} [\psi_{id}(P)] (x).
	\end{align*} 		
\end{proof}

\subsubsection{Proofs of Supplementary Lemmas.}

\begin{proof}[Proof of Lemma \ref{lem:psik-Fdiff}]
The proof follows by the Implicit Function Theorem in \cite{AmbrosettiProdi95} p. 38 with one minor modification. 

First observe that $F_{k}$ takes values in $L^{2}([0,1]) \times \mathbb{D}_{\psi}$ which is a subspace of $L^{2}([0,1]) \times \overline{\mathbb{D}}_{\psi}$ --- the closure being taken with respect to $||.||_{LB}$. The space $L^{2}([0,1]) \times \overline{\mathbb{D}}_{\psi}$ is a Banach space under the norm $||.||_{L^{2}(P)} + ||.||_{LB}$.


We now check the rest of the assumptions of the theorem for each case separately.

	\medskip
	
	(1) Observe that $\theta \mapsto F_{k}(\theta,Q)$ is linear, so $\frac{d F_{k}(\psi_{k}(P),P)}{d\theta} = (T^{\ast}_{k,P}T_{k,P} + \lambda_{k} I) : L^{2}([0,1]) \rightarrow  L^{2}([0,1])$. By our conditions (1)-(2) stated in Example \ref{exa:NPIV} $Kernel((T^{\ast}_{k,P}T_{k,P} + \lambda_{k} I))=\{ 0 \}$ and $(T^{\ast}_{k,P}T_{k,P} + \lambda_{k} I)$ has closed range --- the range of an operator $A$ is closed iff 0 is not an accumulation point of the spectrum of $A^{\ast} A$. Thus $\frac{d F_{k}(\psi_{k}(P),P)}{d\theta}$ is 1-to-1 and onto.
	
	Thus, by the Implicit Function Theorem in \cite{AmbrosettiProdi95} p. 38, there exists a $||.||_{LB}$-open set $U$ of $P$ in $\overline{\mathbb{D}}_{\psi}$ such that $D\psi_{k}(P) = \left(\frac{d F_{k}(\psi_{k}(P),P)}{d\theta} \right)^{-1}\left[ \frac{d F_{k}(\psi_{k}(P),P)}{dP} \right]$ for any $P \in U$.   
	
	We now characterize this expression. For any $k \in \mathbb{N}$ and any $Q \in \mathbb{D}_{\psi}$,
	\begin{align*}
	\frac{d F_{k}(\psi_{k}(P),P)}{dP}[Q] = & (T^{\ast}_{k,Q}T_{k,P} + T^{\ast}_{k,P}T_{k,Q})[\psi_{k}(P)]  - \left( T^{\ast}_{k,Q}[r_{k,P}] + T^{\ast}_{k,P}[r_{k,Q}]  \right) \\
	= & -  T^{\ast}_{k,Q}\left[r_{k,P} - T_{k,P}[\psi_{k}(P)]\right] + T^{\ast}_{k,P} \left[T_{k,Q}[\psi_{k}(P)] - r_{k,Q} \right]\\
	\equiv & Term_{1} + Term_{2}.
	\end{align*}

	Note that
	\begin{align*}
	r_{k,P} - T_{k,P}[\psi_{k}(P)](.) = &  \int \kappa_{k}(x'-\cdot) (y - \psi_{k}(P)(w)) p(y,w,x')dydwdx' \\
	= &  T_{k,P}[\psi(P)-\psi_{k}(P)](.)
	\end{align*} 
	where the last equality follows because $\int (y - \psi(P)(w)) p(y,w,X) dydw  = 0$.  Thus 
	\begin{align*}
	Term_{1}   = &  - T_{k,Q}^{\ast} T_{k,P}[\psi(P)-\psi_{id}(P)].
	\end{align*}
	
	Also
	\begin{align*}
	T_{k,Q}[\psi_{k}(P)](.) - r_{k,Q}(.) = \int \kappa_{k}(x'-\cdot) \int (\psi_{k}(P)(w) - y)Q(dy,dw,dx')
	\end{align*}
	so 
	\begin{align*}
	Term_{3} = T_{k,P} \mathbb{T}_{k,Q}[\varepsilon_{k,P}]
	\end{align*}
	Thus, the result follows.

	\medskip
	
	(2) The proof is analogous to the one for part (1), so  we only present an sketch.  By assumption, $\Pi_{k}^{\ast} T^{\ast}_{k,P}T_{k,P} \Pi_{k}$ is 1-to-1 for each $P$.	
	
	Also, for any $k \in \mathbb{N}$ and any $Q \in \mathbb{D}_{\psi}$,
	\begin{align*}
	\frac{d F_{k}(\psi_{k}(P),P)}{dP}[Q] = & (\Pi_{k}^{\ast} T^{\ast}_{k,Q}T_{k,P} \Pi_{k} + \Pi_{k}^{\ast} T^{\ast}_{k,P}T_{k,Q} \Pi_{k} )[\psi_{k}(P)]  \\
	& -  \Pi_{k}^{\ast} \left( T^{\ast}_{k,Q}[r_{k,P}] + T^{\ast}_{k,P} [r_{k,Q}]  \right) \\
	= & \Pi_{k}^{\ast} \left[ ( T^{\ast}_{k,Q}T_{k,P} +  T^{\ast}_{k,P}T_{k,Q} )[\psi_{k}(P)]  -  \left( T^{\ast}_{k,Q}[r_{k,P}] +  T^{\ast}_{k,P} [r_{k,Q}]  \right) \right]
	\end{align*} 	
	where the second line follows because $\Pi_{k}[\psi_{k}(P)] = \psi_{k}(P)$.
	
	We note that 
	\begin{align*}
	T_{k,Q}[\psi_{k}(P)] - r_{k,Q} = & (u^{J(k)}(X))^{T}Q_{uu}^{-1} \int u^{J(k)}(x)(\psi_{k}(P)(w) - y) Q(dy,dw) \\
	= & -\mathbb{T}_{k,Q}[\varepsilon_{k,P}] 
	\end{align*}
	and 
	\begin{align*}
	T_{k,P}[\psi_{k}(P)] - r_{k,P} = & (u^{J(k)}(X))^{T}Q_{uu}^{-1} \int u^{J(k)}(x)(\psi_{k}(P)(w) - \psi_{id}(P)(w) - \varepsilon_{P}(y,w)) P(dy,dw) \\
	= &  T_{k,P}[\psi_{k}(P)-\psi_{id}(P)]
	\end{align*}
	since $\int  \varepsilon_{P}(y,w)) P(dy,dw) = 0 $. 
\end{proof}

\begin{proof}[Proof of Corollary \ref{cor:gammak-diff}]
	(1) 	By lemma \ref{lem:psik-Fdiff}, for each $k \in \mathbb{N}$, $\psi_{k}$ is $||.||_{LB}$-Frechet differentiable, i.e., for any $Q \in  \mathbb{D}_{\psi}$,
	\begin{align*}
	\left \Vert \psi_{k}(Q) - \psi_{k}(P) - D\psi_{k}(P)[Q-P] \right \Vert_{L^{2}([0,1])} = o(||P-Q||_{LB}).
	\end{align*}
	
	Since $\mathbb{D}_{\psi}$ is linear and $\mathbb{D}_{\psi} \supseteq lin ( \mathcal{D} - \{P \} )$  (see Lemma \ref{lem:NPIV-domain}), the curve $t \mapsto P + t \zeta$ with $\zeta \in \mathbb{D}_{\psi}$ maps into $\mathbb{D}_{\psi}$. Therefore, $\boldsymbol{\psi}$ is $DIFF(P,\mathcal{E}_{||.||_{LB}})$.
	
	By duality
	\begin{align*}
	\sup_{\ell \in L^{2}([0,1]) \colon ||\ell||_{L^{2}([0,1])}=1} \left | \ell[\psi_{k}(P + t \zeta) - \psi_{k}(P) -t D\psi_{k}(P)[\zeta]] \right | = t o(||\zeta||_{LB})
	\end{align*}  
	(here we are abusing notation by using $\ell$ as both an element of $L^{2}([0,1])$ and as the functional). Since $\gamma_{k}$ is linear functional of $\psi_{k}$ this display readily implies that $\gamma_{k}$ is $||.||_{LB}$-Frechet differentiable. This in turn implies part (1).

	\medskip

	(2) Part (1) implies that,
	\begin{align*}
	|\eta_{k}(\zeta)| = o(||\zeta||_{LB}) 
	\end{align*}
	for any $\zeta \in \mathbb{D}_{\psi}$.
	
%
%
%

\end{proof}

\subsection{Appendix for Section \ref{sec:discussion}}
\label{app:discussion}

In this section, with a slight abuse of notation we will use $\mathcal{L}_{n}$ or $\mathcal{L}_{n}(\mathbf{z})$ to denote $\mathcal{L}_{n}(\Lambda)$ for a given realization of the data $\mathbf{z}$.

\subsubsection{Proof of Proposition \ref{pro:Lepski-rate}}

In analogy to the proof of Theorem \ref{thm:rate-choice}, let $n \mapsto k(n)$ be defined as
\begin{align*}
k(n) \equiv \min \left\{  k \in \mathbb{R}_{+} \colon  \bar{\delta}_{1,k}(n) \geq  \bar{B}_{k}(P)  \right\}.
\end{align*}

Also, for each $n \in \mathbb{N}$, let \begin{align*}
A_{n} \equiv  \left\{ \boldsymbol{z} \in \mathbb{Z}^{\infty} \colon \sup_{k \in \mathbb{K}_{n}} \frac{  |\eta_{k}(P_{n}(\boldsymbol{z}) - P)|}{\bar{\delta}_{1,k}(n)} \leq 1    \right\},
\end{align*}
and
\begin{align*}
	B_{n} \equiv \left\{ \boldsymbol{z} \in \mathbb{Z}^{\infty} \colon \sup_{k' \geq k~in~\mathbb{K}_{n}} \frac{ \left| D \psi_{k'}(P)[P_{n}(\boldsymbol{z}) - P]  -   D \psi_{k}(P)[P_{n}(\boldsymbol{z}) - P]   \right|}{\bar{\delta}_{2,k'}(n)} \leq 1 \right\}.
\end{align*}
and \begin{align*}
   k \mapsto \bar{\delta}_{k}(n) \equiv \bar{\delta}_{1,k}(n) + \bar{\delta}_{2,k}(n).
\end{align*}

\begin{lemma}\label{lem:suff-choice-rate-diff}
	Suppose there exists a sequence $(j_{n})_{n}$ such that 
	\begin{enumerate}
		\item For any $\epsilon>0$, there exists a $N$ such that $\mathbf{P}(\{ \mathbf{z}\in \mathbb{Z}^{\infty} \colon j_{n} \in \mathcal{L}_{n}(\mathbf{z})\} \cap A_{n} \cap B_{n}) \geq 1-\epsilon$ for all $n \geq N$.
		\item There exists a $L < \infty$ such that $\frac{\bar{\delta}_{j_{n}}(n) + \bar{B}_{j_{n}}(P)}{||\varphi_{j_{n}}(P)||_{L^{2}(P)}} \leq L \inf_{k \in \mathbb{K}_{n}} \frac{\bar{\delta}_{k}(n) +  \bar{B}_{k}(P)}{||\varphi_{k}(P)||_{L^{2}(P)}}$.
	\end{enumerate}
	Then for any $\epsilon>0$, there exists a $N$ such that 
	\begin{align*}
	\mathbf{P} \left(  \frac{\sqrt{n}}{||\varphi_{\tilde{k}(n)}(P)||_{L^{2}(P)} } \left|    \psi_{\tilde{k}(n)}(P_{n}) - \psi(P) - D\psi_{\tilde{k}(n)}(P)[P_{n} - P]  \right| \geq  2C_{n}L  \inf_{k \in \mathbb{K}_{n}} \frac{\bar{\delta}_{k}(n) + \sqrt{n} \bar{B}_{k}(P)}{||\varphi_{k}(P)||_{L^{2}(P)}}  \right) \leq \epsilon
	\end{align*}
	for all $n \geq N$, where $C$ is as in Assumption \ref{ass:Lepski-undersmooth}(iv).
\end{lemma}

\begin{proof}

	For any $n \in \mathbb{N}$ and any $\boldsymbol{z} \in \left\{ \boldsymbol{z} \in \mathbb{Z}^{\infty} \colon  j_{n} \in \mathcal{L}_{n}(\boldsymbol{z}) \right\} \cap A_{n} \cap B_{n} $
\begin{align*}
& \frac{\sqrt{n}}{||\varphi_{\tilde{k}_{n}(\boldsymbol{z})}(P)||_{L^{2}(P)} } \left|    \psi_{\tilde{k}_{n}(\boldsymbol{z})}(P_{n}(\boldsymbol{z})) - \psi(P) - D\psi_{\tilde{k}_{n}(\boldsymbol{z})}(P)[P_{n}(\boldsymbol{z}) - P]  \right| \\
\leq & \sqrt{n} \frac{|\eta_{j_{n}}(P_{n}(\boldsymbol{z}) - P)| + \bar{B}_{j_{n}}(P)}{||\varphi_{\tilde{k}_{n}(\boldsymbol{z})}(P)||_{L^{2}(P)} } \\
& + \frac{\sqrt{n}}{||\varphi_{\tilde{k}_{n}(\boldsymbol{z})}(P)||_{L^{2}(P)} } \left( \left|  \psi_{\tilde{k}_{n}(\boldsymbol{z})}(P_{n}(\boldsymbol{z})) - \psi_{j_{n}}(P_{n}(\boldsymbol{z})) \right| + \left|D\psi_{\tilde{k}_{n}(\boldsymbol{z})}(P)[P_{n}(\boldsymbol{z}) - P] - D\psi_{j_{n}}(P)[P_{n}(\boldsymbol{z}) - P] \right| \right) \\
\leq & 2\sqrt{n} \frac{ \bar{\delta}_{1,j_{n}}(n) +   \bar{B}_{j_{n}}(P)}{||\varphi_{\tilde{k}_{n}(\boldsymbol{z})}(P)||_{L^{2}(P)} } + \sqrt{n}  \frac{ \bar{\delta}_{2,j_{n}}(n) }{||\varphi_{\tilde{k}_{n}(\boldsymbol{z})}(P)||_{L^{2}(P)} }
\end{align*}
where the first inequality follows from the definition of differentiability and simple algebra; the second inequality follows from the fact that $\boldsymbol{z} \in A_{n} \cap B_{n}$ and from the fact that $j_{n} \geq \tilde{k}_{n}(\boldsymbol{z})$ and the definition of $\mathcal{L}_{n}(\mathbf{z})$.

This result and the definition of $(C_{n})_{n}$ in the Proposition \ref{pro:Lepski-rate} imply 
\begin{align*}
 \frac{\sqrt{n}}{||\varphi_{\tilde{k}_{n}(\boldsymbol{z})}(P)||_{L^{2}(P)} } \left|    \psi_{\tilde{k}_{n}(\boldsymbol{z})}(P_{n}(\boldsymbol{z})) - \psi(P) - D\psi_{\tilde{k}_{n}(\boldsymbol{z})}(P)[P_{n}(\boldsymbol{z}) - P]  \right| 
\leq  2 C_{n} \sqrt{n} \frac{ \bar{\delta}_{j_{n}}(n) +   \bar{B}_{j_{n}}(P)}{||\varphi_{j_{n}}(P)||_{L^{2}(P)} } .
\end{align*}
	Thus, by condition 2, 
	\begin{align*}
	\frac{\sqrt{n}}{||\varphi_{\tilde{k}_{n}(\boldsymbol{z})}(P)||_{L^{2}(P)} } \left|    \psi_{\tilde{k}_{n}(\boldsymbol{z})}(P_{n}(\boldsymbol{z})) - \psi(P) - D\psi_{\tilde{k}_{n}(\boldsymbol{z})}(P)[P_{n}(\boldsymbol{z}) - P]  \right|  \leq 2 L C_{n} \sqrt{n}  \inf_{k \in \mathbb{K}_{n}} \frac{\bar{\delta}_{k}(n) + \bar{B}_{k}(P)}{||\varphi_{k}(P)||_{L^{2}(P)}}.
	\end{align*}

	The result thus follows from condition 1. 
\end{proof}

We now construct a sequence $(h_{n})_{n}$ that satisfies both conditions of the lemma. The construction is completely analogous to the one in the proof of Theorem \ref{thm:rate-choice} but using $\bar{\delta}_{k}(n)/\sqrt{n}$ instead of $\delta_{k}(r^{-1}_{n})$. 

Let, for each $n \in \mathbb{N}$,
\begin{align*}
\mathbb{K}_{n}^{+} \equiv & \{ k \in \mathbb{K}_{n} \colon  \bar{\delta}_{k}(n)\geq \bar{B}_{k}(P)  \}~and~\\
\mathbb{K}_{n}^{-} \equiv  & \{ k \in \mathbb{K}_{n} \colon  \bar{\delta}_{k}(n)  \leq \bar{B}_{k}(P)  \}.
\end{align*} 

%
%

For each $n \in \mathbb{N}$, let
\begin{align*}
T^{+}_{n} =    \sqrt{n}  \frac{ \bar{\delta}_{h^{+}_{n}}(n) + \bar{B}_{h^{+}_{n}}(P)  }{||\varphi_{h^{+}_{n}}(P)||_{L^{2}(P)}  }
\end{align*}
if $\mathbb{K}_{n}^{+}$ is non-empty where  
\begin{align*}
h^{+}_{n} = \min \{  k \colon k \in \mathbb{K}_{n}^{+} \};
\end{align*}
and $T^{+}_{n} =   + \infty $, if $\mathbb{K}_{n}^{+}$ is empty. Similarly, 
\begin{align*}
T^{-}_{n} =    \sqrt{n}  \frac{ \bar{\delta}_{h^{-}_{n}}(n) + \bar{B}_{h^{-}_{n}}(P) } {||\varphi_{h^{-}_{n}}(P)||_{L^{2}(P)}  }
\end{align*}
if $\mathbb{K}_{n}^{-}$ is non-empty where  
\begin{align*}
h^{-}_{n} = \max \{  k \colon k \in \mathbb{K}_{n}^{-} \};
\end{align*}
and $T^{-}_{n} =   + \infty$, if $\mathbb{K}_{n}^{-}$ is empty.

%
%
%

Finally, for each $n \in \mathbb{N}$, let $h_{n} \in \mathbb{K}_{n}$ be such that
\begin{align*}
h_{n} = h^{+}_{n} 1\{  T^{+}_{n} \leq T^{-}_{n}  \} + h^{-}_{n} 1\{  T^{+}_{n} > T^{-}_{n}  \}.
\end{align*}

\begin{lemma}\label{lem:h(n)-prop-diff}
	For each $n \in \mathbb{N}$, $h_{n}$ exists and \begin{align*}
	 \sqrt{n} \frac{\bar{\delta}_{h_{n}}(n) +  \bar{B}_{h_{n}}(P) }{||\varphi_{h_{n}}(P)||_{L^{2}(P)}  }  = \min \left\{  T^{-}_{n} , T^{+}_{n}      \right\}.
	\end{align*}
\end{lemma}

\begin{proof}
	The proof is identical to the one of Lemma \ref{lem:h(n)-prop}.
\end{proof}

\begin{lemma}\label{lem:rate-h(n)-bound-diff}
	Suppose Assumption  \ref{ass:Lepski-undersmooth} holds. For each $n \in \mathbb{N}$, 
	\begin{align*}
	 \sqrt{n} \frac{ \bar{\delta}_{h_{n}}(n) +  \bar{B}_{h_{n}}(P) } {||\varphi_{h_{n}}(P)||_{L^{2}(P)}  }    \leq 2 C_{n}  \sqrt{n} \inf_{k \in \mathbb{K}_{n}} 	\frac{ \bar{\delta}_{k}(n) +  \bar{B}_{k}(P) } {||\varphi_{k}(P)||_{L^{2}(P)}  } 
	\end{align*}
	where $(C_{n})_{n}$ is as in Proposition \ref{pro:Lepski-rate}.
\end{lemma}

\begin{proof}
	Observe that 
	\begin{align*}
	\inf_{k \in \mathbb{K}_{n}}  \sqrt{n}	\frac{ \bar{\delta}_{k}(n) +  \bar{B}_{k}(P) } {||\varphi_{k}(P)||_{L^{2}(P)}  }    \geq  \sqrt{n} \min \left\{ \inf_{k \in \mathcal{G}^{+}_{n}}  	\frac{ \bar{\delta}_{k}(n) +  \bar{B}_{k}(P) } {||\varphi_{k}(P)||_{L^{2}(P)}  }       ,  \inf_{k \in \mathcal{G}^{-}_{n}}  	\frac{ \bar{\delta}_{k}(n) +  \bar{B}_{k}(P) } {||\varphi_{k}(P)||_{L^{2}(P)}  }     \right\}
	\end{align*}
	where the  infimum is defined as $+\infty$ if the corresponding set is empty. 
	
	Fix any $n \in \mathbb{N}$, if $\mathcal{G}^{+}_{n} \ne \{\emptyset\}$, 
	\begin{align*}
	\inf_{k \in \mathcal{G}^{+}_{n}}  	\frac{ \bar{\delta}_{k}(n) + \bar{B}_{k}(P) } {||\varphi_{k}(P)||_{L^{2}(P)}  }      \geq \inf_{k \in \mathcal{G}^{+}_{n}}   	\frac{ \bar{\delta}_{k}(n)  } {||\varphi_{k}(P)||_{L^{2}(P)}  }      \geq C^{-1}_{n}  	\frac{ \bar{\delta}_{h^{+}_{n}}(n)  } {||\varphi_{h^{+}_{n}}(P)||_{L^{2}(P)}  }   \geq 0.5 C^{-1}_{n} \left(  	\frac{ \bar{\delta}_{h^{+}_{n}}(n)  +  \bar{B}_{h^{+}_{n}}(P) } {||\varphi_{h^{+}_{n}}(P)||_{L^{2}(P)}  }  \right)
	\end{align*}
	where the first inequality follows from the fact that $\bar{B}_{k}(P) \geq 0$; the second one follows from the fact that $h^{+}_{n}$ is minimal over $\mathcal{G}^{+}_{n}$ 
    and the fact that $\inf_{k \in \mathcal{G}^{+}_{n}} \frac{||\varphi_{h^{+}_{n}}(P)||_{L^{2}(P)}}{||\varphi_{k}(P)||_{L^{2}(P)}} = \left( \sup_{k \in \mathcal{G}^{+}_{n}} \frac{||\varphi_{k}(P)||_{L^{2}(P)}}{||\varphi_{h^{+}_{n}}(P)||_{L^{2}(P)}}   \right)^{-1} \geq \left( \sup_{k \in \mathbb{K}_{n} \colon  k \geq h^{+}_{n}} \frac{||\varphi_{k}(P)||_{L^{2}(P)}}{||\varphi_{h^{+}_{n}}(P)||_{L^{2}(P)}}   \right)^{-1} \geq C^{-1}_{n}$.

	Similarly, if $\mathcal{G}^{-}_{n} \ne \{\emptyset\}$, 
	\begin{align*}
	\inf_{k \in \mathcal{G}^{-}_{n}} \frac{ \bar{\delta}_{k}(n) +  \bar{B}_{k}(P) } {||\varphi_{k}(P)||_{L^{2}(P)}  }    \geq \inf_{k \in \mathcal{G}^{-}_{n}}  \frac{  \bar{B}_{k}(P) } {||\varphi_{k}(P)||_{L^{2}(P)}  }    \geq C^{-1}_{n}  \frac{  \bar{B}_{h^{-}_{n}}(P) } {||\varphi_{h^{-}_{n}}(P)||_{L^{2}(P)}  }     \geq 0.5 C^{-1}_{n} \left( \frac{ \bar{\delta}_{h^{-}_{n}}(n) +  \bar{B}_{h^{-}_{n}}(P) } {||\varphi_{h^{-}_{n}}(P)||_{L^{2}(P)}  }     \right).
	\end{align*}
	Where here we use monotonicity of $k \mapsto \bar{B}_{k}(P)$ and the fact that $\inf_{k \in \mathcal{G}^{-}_{n}} \frac{||\varphi_{h^{-}_{n}}(P)||_{L^{2}(P)}}{||\varphi_{k}(P)||_{L^{2}(P)}} = \left( \sup_{k \in \mathcal{G}^{-}_{n}} \frac{||\varphi_{k}(P)||_{L^{2}(P)}}{||\varphi_{h^{-}_{n}}(P)||_{L^{2}(P)}}   \right)^{-1} \geq \left( \sup_{k \in \mathbb{K}_{n} \colon  k \leq h^{-}_{n}} \frac{||\varphi_{k}(P)||_{L^{2}(P)}}{||\varphi_{h^{-}_{n}}(P)||_{L^{2}(P)}}   \right)^{-1} \geq C^{-1}_{n}$. 
	
	Thus, \begin{align*}
	\inf_{k \in \mathbb{K}_{n}} \{  \delta_{k}(r^{-1}_{n}) + \bar{B}_{k}(P)   \}  \geq 0.5 \min \{ T^{-}_{n} , T^{+}_{n}  \},
	\end{align*}
	and by Lemma \ref{lem:h(n)-prop-diff} the desired result follows.

\end{proof}

\begin{lemma}\label{lem:hn-in-Fn-diff}
	Suppose Assumption \ref{ass:Lepski-undersmooth} holds. For any $n \in \mathbb{N}$, $\mathbf{P}( 	\{\boldsymbol{z} \in \mathbb{Z}^{\infty} \colon h_{n} \notin \mathcal{L}_{n}(\boldsymbol{z}) \}  \cap A_{n}) \leq \mathbf{P}(B_{n}^{C})$.
\end{lemma}

\begin{proof}
	For any $n \in \mathbb{N}$,
	\begin{align*}
	\mathbf{P}( 	\{\boldsymbol{z} \in \mathbb{Z}^{\infty} \colon h_{n} \notin \mathcal{L}_{n}(\boldsymbol{z}) \}  \cap A_{n}) \leq \mathbf{P}( 	\{\boldsymbol{z} \in \mathbb{Z}^{\infty} \colon h_{n} \notin \mathcal{L}_{n}(\boldsymbol{z}) \}   \cap A_{n} \cap B_{n}) + \mathbf{P}(B_{n}^{C}).
	\end{align*}
	
	By definition of $\mathcal{L}_{n}$, 
	\begin{align*}
	\{\boldsymbol{z} \in \mathbb{Z}^{\infty} \colon h_{n} \notin \mathcal{L}_{n}(\boldsymbol{z}) \}  \subseteq C_{n} \equiv \left\{ \boldsymbol{z} \in \mathbb{Z}^{\infty} \colon  \exists k \in \mathbb{K}_{n} \colon k > h_{n}~and~ |\psi_{k}(P_{n}(\boldsymbol{z})) - \psi_{h_{n}}(P_{n}(\boldsymbol{z}))| > 4 \bar{\delta}_{k}(n) \right\},
	\end{align*}
	where $(n,k) \mapsto \bar{\delta}_{k}(n) \equiv  \bar{\delta}_{1,k}(n) +  \bar{\delta}_{2,k}(n)$.
	
	For any $k \in \mathbb{K}_{n}$ such that $k \geq h_{n}$ and any $\boldsymbol{z} \in C_{n} \cap B_{n} \cap A_{n}$ (to ease the notational burden we omit $\boldsymbol{z}$ from the expressions below)   
	\begin{align*}
	|\psi_{k}(P_{n}) - \psi_{h_{n}}(P_{n})| \leq & \left| \psi_{k}(P_{n}) - \psi(P)  - D \psi_{k}(P)[P_{n} - P]     \right| + \left| \psi_{h_{n}}(P_{n}) - \psi(P)  - D \psi_{h_{n}}(P)[P_{n} - P]     \right|\\
	& + \left| D \psi_{k}(P)[P_{n} - P]  -   D \psi_{h_{n}}(P)[P_{n} - P]   \right|  \\
	& +  \left| \psi_{k}(P) - \psi(P)    \right| + \left| \psi_{h_{n}}(P) - \psi(P)    \right|\\
	\leq & |\eta_{k}(P_{n}-P)| + |\eta_{h_{n}}(P_{n}-P)| + \bar{\delta}_{2,k}(n) + \bar{B}_{k}(P) + \bar{B}_{h_{n}}(P)  \\
	\leq &  \bar{\delta}_{1,k}(n)  +  \bar{\delta}_{1,h_{n}}(n)  + \bar{\delta}_{2,k}(n) + \bar{B}_{k}(P) + \bar{B}_{h_{n}}(P)
	\end{align*}
	where the second inequality follows from the definition of $\eta$, the fact that $\boldsymbol{z} \in B_{n}$ and the fact that $k>h_{n}$; the third inequality follows from the fact that $\boldsymbol{z} \in A_{n}$. Thus, 
	\begin{align}\notag
	& \{h_{n} \notin \mathcal{L}_{n} \} \cap A_{n} \cap B_{n}  \\ \label{eqn:hn-in-Fn-diff-1}
	\subseteq & \left\{   \exists k \in \mathbb{K}_{n} \colon k > h_{n}~and~  \bar{\delta}_{1,k}(n)  +  \bar{\delta}_{1,h_{n}}(n)  + \bar{\delta}_{2,k}(n)  + \bar{B}_{k}(P) + \bar{B}_{h_{n}}(P)  > 4  \bar{\delta}_{k}(n)  \right\}.
	\end{align}

	We now derive a series of useful claims.
	
	\medskip 
	
	\textbf{Claim 1:} If there exists $k \in \mathbb{K}_{n}$ such that $k > h_{n}$ and $h_{n} = h^{-}_{n}$, then $k \in \mathcal{G}^{+}_{n}$. \textbf{Proof:} If $h_{n}=h^{-}_{n}$, then $h_{n}$ is the largest element of $\mathcal{G}^{-}_{n}$ and thus $k \notin 	\mathcal{G}^{-}_{n}$, which means that $k \in \mathcal{G}^{+}_{n}$. $\square$
	
	\medskip 
	
	A corollary of this claim is that if there exists $k \in \mathbb{K}_{n}$ such that $k > h_{n}$ and $h_{n} = h^{-}_{n}$, then $\mathcal{G}^{+}_{n}$ is non-empty. From this claim, we derive the following two claims. 
	
	\medskip 
	
	\textbf{Claim 2:} If there exists a $k > h_{n}$, then  $\bar{\delta}_{1,h_{n}}(n) + \bar{B}_{h_{n}}(P) \leq 2 \bar{\delta}_{h^{+}_{n}}(n) $. \textbf{Proof:} If $h_{n} = h^{+}_{n}$, then $\bar{\delta}_{h_{n}}(n) +  \bar{B}_{h_{n}}(P)  \leq \bar{\delta}_{h^{+}_{n}}(n) +  \bar{B}_{h^{+}_{n}}(P)  \leq 2  \bar{\delta}_{h^{+}_{n}}(n) $. If $h_{n} = h^{-}_{n}$, by the previous claim it follows that $\mathcal{G}^{+}_{n}$ is non-empty and thus $h^{+}_{n}$ is well-defined, thus $\bar{\delta}_{h_{n}}(n) +   \bar{B}_{h_{n}}(P) \leq \bar{\delta}_{h^{+}_{n}}(n) +  \bar{B}_{h^{+}_{n}}(P)   \leq 2 \bar{\delta}_{h^{+}_{n}}(n)  $. $\square$
	
	\medskip 
	
	\textbf{Claim 3:} For any $k > h_{n}$, $\bar{\delta}_{1,k}(n) \geq \bar{B}_{k}(P)$. \textbf{Proof:} If $h_{n} = h^{+}_{n}$ then the claim follows because $k \mapsto \bar{\delta}_{k}(n) - \bar{B}_{k}(P)$ is non-decreasing under Assumption \ref{ass:Lepski-undersmooth}(i). If $h_{n}=h^{-}_{n}$, then $k \in \mathcal{G}^{+}_{n}$ by Claim 1 and thus $\bar{\delta}_{k}(n) \geq \bar{B}_{k}(P)$. $\square$

	\medskip 
	
	By Claims 2 and 3, it follows that  if there exists $k \in \mathbb{K}_{n}$ such that  $k \geq h_{n}$, then $ \bar{\delta}_{k}(n) + \bar{B}_{k}(P) +  \bar{\delta}_{h_{n}}(n)  + \bar{B}_{h_{n}}(P) \leq 2 \bar{\delta}_{k}(n)  + 2  \bar{\delta}_{h_{n}}(n)   \leq 4  \bar{\delta}_{k}(n) $ where the last inequality follows from the fact that $k \mapsto \bar{\delta}_{k}(n)$ is non-decreasing by Assumption \ref{ass:Lepski-undersmooth}(i) and the fact that $k  > h^{+}_{n}$ because $k > h_{n}$ and so by Claim 1 $k \in \mathbb{K}_{n}^{+}$ and $h^{+}_{n}$ is minimal in this set.

	Hence applying this result to expression \ref{eqn:hn-in-Fn-diff-1} and since $\bar{\delta}_{k}(n) = \bar{\delta}_{1,k}(n)+\bar{\delta}_{2,k}(n)$, it follows that	\begin{align}
	\{h_{n} \notin \mathcal{L}_{n} \} \cap A_{n} \cap B_{n} \subseteq \left\{   \exists k \in \mathbb{K}_{n} \colon k \geq h_{n}~and~ 4 \bar{\delta}_{k}(n) > 4  \bar{\delta}_{k}(n) \right\},
	\end{align}
	which is empty. Hence, $\mathbf{P}(\{ h_{n} \notin \mathcal{L}_{n} \}  \cap A_{n}) \leq \mathbf{P}(B_{n}^{C})$ as desired.
\end{proof}

\begin{proof}[Proof of Proposition \ref{pro:Lepski-rate}]
	We verify that  $(h_{n})_{n \in \mathbb{N}}$ satisfies both conditions in Lemma \ref{lem:suff-choice-rate-diff}. By Lemma \ref{lem:rate-h(n)-bound-diff} condition 2 in the Lemma \ref{lem:suff-choice-rate-diff} holds with $L=2C_{n}$. To check condition 1 in the Lemma \ref{lem:suff-choice-rate-diff}, observe that 
	\begin{align*}
	\mathbf{P} \left( \mathbb{Z}^{\infty} \setminus  \left\{ 	\{\boldsymbol{z} \in \mathbb{Z}^{\infty} \colon h_{n} \in \mathcal{L}_{n}(\boldsymbol{z}) \}  \cap A_{n} \cap B_{n} \right\} \right) \leq & \mathbf{P} \left(  	\{\boldsymbol{z} \in \mathbb{Z}^{\infty} \colon h_{n} \notin \mathcal{L}_{n}(\boldsymbol{z}) \}   \right) + \mathbf{P} \left( A^{C}_{n} \right) + \mathbf{P} \left( B^{C}_{n} \right) \\
	\leq & \mathbf{P} \left( 	\{\boldsymbol{z} \in \mathbb{Z}^{\infty} \colon h_{n} \notin \mathcal{L}_{n}(\boldsymbol{z}) \}  \cap A_{n} \right) + 2 \mathbf{P} \left( A^{C}_{n} \right) + \mathbf{P} \left( B^{C}_{n} \right)
	\end{align*}
	Thus, by Lemma \ref{lem:hn-in-Fn-diff} 
	\begin{align*}
	\mathbf{P} \left( \mathbb{Z}^{\infty} \setminus  \left\{ 	\{\boldsymbol{z} \in \mathbb{Z}^{\infty} \colon h_{n} \in \mathcal{L}_{n}(\boldsymbol{z}) \}  \cap A_{n} \right\} \right)  \leq 2 \mathbf{P} \left( A^{C}_{n} \right)  + 2 \mathbf{P}(B^{C}_{n}).
	\end{align*}
	Under assumption \ref{ass:Lepski-undersmooth}(i) the first term in the RHS vanishes. Regarding the second term, by the union bound and the Markov inequality
	\begin{align*}
	\mathbf{P}(B^{C}_{n}) \leq & \frac{|\mathbb{K}_{n}|^{2}}{n} \sup_{k' \geq k ~in ~ \mathbb{K}_{n}} \bar{\delta}^{-2}_{2,k'}(n) E_{\mathbf{P}} \left[  \left( D\psi_{k'}(P)[P_{n} - P] -  D\psi_{k}(P)[P_{n} - P] \right)^{2} \right] \\
	=& \frac{|\mathbb{K}_{n}|^{2}}{n} \sup_{k' \geq k ~in ~ \mathbb{K}_{n}} \bar{\delta}^{-2}_{2,k'}(n) E_{\mathbf{P}} \left[ \left( \varphi_{k'}(P)(Z) -  \varphi_{k}(P)(Z) \right)^{2} \right]\\
	= & o(1)
	\end{align*}
	where the last line follows from Assumption \ref{ass:Lepski-undersmooth}(ii).
	
\end{proof}

\subsubsection{Appendix for Example \ref{exa:ipdf2-Lepski}}
\label{app:ipdf2-Lepski}

Next, we provide an explicit characterization of $\eta_{k}(P_{n} - P)$.

\begin{lemma}\label{lem:ipdf2-diff-2}
	For any $P \in \mathcal{M}$ and any $k \in \mathbb{N}$,
	\begin{align*}
	\eta_{k}(P_{n} - P) = O_{P} \left( \frac{\kappa_{k}(0)}{n} +  \frac{2||\kappa||_{L^{2}} \sqrt{k}  \sqrt{||p||_{L^{\infty}}} }{n}   +   \frac{ ||p||_{L^{\infty}}}{n} + \frac{2\sqrt{k} ||\kappa||_{L^{2}} \sqrt{||p||_{L^{\infty}}} }{n^{2}} \right).
	\end{align*}
\end{lemma}

\begin{proof}
	The proof is relegated to the end of this section. 
\end{proof}

\begin{proof}[Proof of Lemma \ref {lem:ipdf2-suff-assumption-bound}]
	Observe that $z \mapsto \varphi_{1/h}(P)(z) \equiv (\kappa_{1/h} \star P)(z) - E_{P}[ (\kappa_{1/h} \star P)(Z) ] $. So for any $h$ and $h'$,
	\begin{align*}
& E_{P} \left[  \left( \varphi_{1/h}(P)(Z) - \varphi_{1/h'}(P)(Z)  \right)^{2} \right] \\
= & E_{P} \left[ \left(   \left( (\kappa_{1/h} \star P)(Z) - E_{P}[ (\kappa_{1/h} \star P)(Z) ] - \{ (\kappa_{1/h'} \star P)(Z) - E_{P}[ (\kappa_{1/h'} \star P)(Z) ] \}  \right) \right)^{2} \right] \\
= & E_{P} \left[ \left(  (\kappa_{1/h} \star P)(Z) - (\kappa_{1/h'} \star P)(Z) - \{  E_{P}[ (\kappa_{1/h} \star P)(Z) ]- E_{P}[ (\kappa_{1/h'} \star P)(Z) ] \}  \right)^{2} \right] \\
\leq & 4 E_{P} \left[ \left(  \int \kappa(u) \{ p(Z+h u) -  p(Z+ h' u) \} du  \right)^{2} \right].
\end{align*}	
	
	By expression (\ref{eqn:ipdf2-smooth}), it follows that $| p(Z+h u) -  p(Z+ h' u) | \leq C(z) (|h|^{\varrho} + |h'|^{\varrho} ) |u|^{\varrho} $. 
	Thus, for any $z$,
	\begin{align*}
	&  \int \kappa(u) \{ p(z+h u) -  p(z+ h' u) \} du  \\
	\leq &  \int \kappa(u) p'(z)(h + h') u du + C(z)( |h|^{m+\varrho} + |h'|^{\varrho} ) \int |\kappa(u)| |u|^{\varrho}  du.
	\end{align*}
	By symmetry of $\kappa$, $\int \kappa(u) u du = 0$. 
	
	Therefore, 
	\begin{align*}
 E_{P} \left[  \left( \varphi_{1/h}(P)(Z) - \varphi_{1/h'}(P)(Z)  \right)^{2} \right] \leq   4 E_{P} \left[ (C(Z))^{2} \right]  \left( ( (h)^{\varrho} + (h')^{\varrho} ) \int |\kappa(u)| |u|^{\varrho}  du   \right)^{2}.
\end{align*}	

%
%

	\bigskip 
	
	By the proof of the Lemma \ref{lem:ipdf2-diff-2}, for any $h > 0$ and any $M>0$, there exists a $N$ such that
	\begin{align*}
	\mathbf{P} \left( \frac{|\eta_{1/h}(P_{n} - P)|}{M \left( \frac{\kappa(0)}{h n} + \frac{1}{n \sqrt{h} } \right) } \geq 1  \right) \leq M^{-1}
	\end{align*}
	for all $n \geq N$. By the union bound 
	\begin{align*}
	\mathbf{P} \left( \sup_{k \in \mathbb{K}_{n} } \frac{|\eta_{k}(P_{n} - P)|}{ \frac{M}{n}  \left( \kappa(0) k+ \sqrt{k}  \right)  } \geq 1  \right) \leq \sum_{k \in \mathbb{K}_{n} }   \mathbf{P} \left(  \frac{|\eta_{k}(P_{n} - P)|}{ \frac{M}{n}  \left( \kappa(0) k+ \sqrt{k}  \right)  } \geq 1  \right)  \leq |\mathbb{K}_{n} | M^{-1}.
	\end{align*}
\end{proof}

\begin{proof}[Proof of Lemma \ref{lem:ipdf2-diff-2}]
	Consider the curve $t \mapsto P + tQ$. It is a valid curve because $\mathbb{D}_{\psi} = ca(\mathbb{R})$. Therefore
	\begin{align*}
	\psi_{k}(P+tQ) - \psi_{k}(P) =& t \left\{  \int (\kappa_{k} \star Q)(x) P(dx) + \int (\kappa_{k} \star P)(x) Q(dx) \right\} \\
	& + t^{2} \int (\kappa_{k} \star Q)(x) Q(dx).
	\end{align*}
	Since $\kappa$ is symmetric, $\int (\kappa_{k} \star P)(x) Q(dx) = \int (\kappa_{k} \star Q)(x) P(dx)$. From this display, $\eta_{k}(tQ) = t^{2}  \int (\kappa_{k} \star Q)(x) Q(dx)$ and $D \psi_{k}(P)[Q]=2\int (\kappa_{k} \star P)(x) Q(dx)$.
	
	The mapping $Q \mapsto 2\int (\kappa_{k} \star P)(x) Q(dx)$ is clearly linear. Also, note that $(\kappa_{k} \star P)(.) = \int \kappa (u) p(\cdot + k u) du $. Hence, for any reals $x$ and $x'$
	\begin{align*}
	|(\kappa_{k} \star P)(x) - (\kappa_{k} \star P)(x')| = \int \kappa (u) \{p(x + u/k) - p(x' + u/k) \} du.
	\end{align*}
	Thus, under the smoothness condition on $p$ in expression (\ref{eqn:ipdf2-smooth}), it follows that $x \mapsto (\kappa_{k} \star P)(x)$ is uniformly continuous and bounded, so the mapping $Q \mapsto  2\int (\kappa_{k} \star P)(x) Q(dx)$ is continuous with respect to the $||.||_{LB}$.

	To establish the rate result, we use the Markov inequality. We also introduce the following notation $\int (\kappa_{k} \star Q)(x) Q(dx) = \langle  \kappa_{h_{k}} \star Q , Q   \rangle  $ where $\langle . , . \rangle$ is the inner produce of the dual $(L^{\infty}(\mathbb{R}),ca(\mathbb{R}))$. 
	
	It follows that
	\begin{align*}
	\sqrt{E \left[ \left( \eta_{k}(P_{n} - P)    \right)^{2}    \right]} = &	\sqrt{E \left[ \left( \langle  \kappa_{k} \star (P_{n}-P) , P_{n}-P  \rangle   \right)^{2}    \right]} \\
	= &	\sqrt{E \left[ \left( \langle  \kappa_{k} \star P_{n} , P_{n} \rangle - 2\langle  \kappa_{k} \star P , P_{n} \rangle    + \langle  \kappa_{k} \star P , P \rangle      \right)^{2}    \right]}
	\end{align*} 
	where the second line follows from symmetry of $\kappa$ which implies $\langle  \kappa_{k} \star P_{n} , P \rangle    = \langle  P_{n} , \kappa_{k} \star P \rangle   $.
	
	We note that
	\begin{align*}
	\langle  \kappa_{k} \star P_{n} , P_{n} \rangle  = &  \frac{\kappa_{k}(0)}{n} + \frac{1}{n^{2}} \sum_{i \ne j}  \kappa_{k}(Z_{i}-Z_{j}) \\
	= & \frac{\kappa_{k}(0)}{n} + \frac{1}{n^{2}} \left(  \sum_{i=1}^{n-1}  \sum_{j = i+1}^{n} \kappa_{k}(Z_{i}-Z_{j}) + \sum_{i=2}^{n}  \sum_{j = 1}^{i-1} \kappa_{k}(Z_{i}-Z_{j}) \right),
	\end{align*}
	also 
	\begin{align*}
	\langle  \kappa_{k} \star P , P_{n} \rangle =& \frac{1}{n} \sum_{i=1}^{n} (\kappa_{k} \star P)(Z_{i}) =   \frac{1}{n} \sum_{i=1}^{n} E_{P}[\kappa_{k}(Z_{i}-Z)] \\
	= & \frac{1}{n} \sum_{i=1}^{n} E_{P}[\kappa_{k}(Z_{i}-Z)]\frac{i}{n} + \frac{1}{n} \sum_{i=1}^{n} E_{P}[\kappa_{k}(Z_{i}-Z)]\frac{n-i}{n}\\
	= & \frac{1}{n} \sum_{i=1}^{n} E_{P}[\kappa_{k}(Z_{i}-Z)]\frac{i}{n} + \frac{1}{n} \sum_{i=1}^{n-1} E_{P}[\kappa_{k}(Z_{i}-Z)]\frac{n-i}{n}\\		
	= & \frac{1}{n^{2}}  \sum_{i=2}^{n} \sum_{j=1}^{i-1} E_{P}[\kappa_{k}(Z_{i}-Z_{j})] + \frac{1}{n^{2}} \sum_{i=1}^{n-1} \sum_{j=i+1}^{n} E_{P}[\kappa_{k}(Z_{i}-Z_{j})]\\
	& + \frac{1}{n^{2}} E_{P}[\kappa_{k}(Z_{1}-Z)].
	\end{align*}
	where the third line follows because $ E_{P}[\kappa_{k}(Z_{n}-Z)]\frac{n-n}{n} = 0$, and the fourth one follows from the fact that by iid-ness, $E_{P}[\kappa_{k}(Z_{i}-Z_{j})] = E_{P}[\kappa_{k}(Z_{i}-Z)]$ for all $j$.
	
	Therefore,
	\begin{align*}
	\langle  \kappa_{k} \star P_{n} , P_{n} \rangle - 2\langle  \kappa_{k} \star P , P_{n} \rangle = & \frac{\kappa_{k}(0)}{n} + \frac{1}{n^{2}}  \left(  \sum_{i=1}^{n-1}  \sum_{j = i+1}^{n} \kappa_{k}(Z_{i}-Z_{j}) + \sum_{i=2}^{n}  \sum_{j = 1}^{i-1} \kappa_{k}(Z_{i}-Z_{j}) \right) \\
	& - \frac{2}{n^{2}}  \sum_{i=2}^{n} \sum_{j=1}^{i-1} E_{P}[\kappa_{k}(Z_{i}-Z_{j})] + \frac{1}{n^{2}} \sum_{i=1}^{n-1} \sum_{j=i+1}^{n} E_{P}[\kappa_{k}(Z_{i}-Z_{j})]\\
	& - \frac{2}{n^{2}} E_{P}[\kappa_{k}(Z_{1}-Z)]\\
	= & \frac{1}{n^{2}} \sum_{i=1}^{n-1}  \sum_{j = i+1}^{n} \{\kappa_{k}(Z_{i}-Z_{j}) - 2 E_{P}[\kappa_{k}(Z_{i}-Z_{j})]\}\\
	& + \frac{1}{n^{2}} \sum_{i=2}^{n}  \sum_{j = 1}^{i-1} \{\kappa_{k}(Z_{i}-Z_{j}) - 2 E_{P}[\kappa_{k}(Z_{i}-Z_{j})]\}\\
	& + \frac{\kappa_{k}(0)}{n} - \frac{2}{n^{2}} E_{P}[\kappa_{k}(Z_{1}-Z)]\\
	= &  \frac{2}{n^{2}} \sum_{i < j} \{\kappa_{k}(Z_{i}-Z_{j}) - 2 E_{P}[\kappa_{k}(Z_{i}-Z_{j})]\} \\
	&+ \frac{\kappa_{k}(0)}{n} - \frac{2}{n^{2}} E_{P}[\kappa_{k}(Z_{1}-Z)],
	\end{align*}
	where the last line follows by symmetry of $\kappa$ since $\kappa(Z_{i} - Z_{j}) = \kappa(Z_{j} - Z_{i})$ for all $i,j$. 
	
	Since $\langle  \kappa_{k} \star P , P \rangle  = E_{P \cdot P}[\kappa_{k}(Z-Z')] = \frac{1}{n^{2}} \sum_{i,j} E_{P \cdot P}[\kappa_{k}(Z-Z')] $, it follows that 
	\begin{align*}
	&\langle  \kappa_{k} \star P_{n} , P_{n} \rangle - 2\langle  \kappa_{k} \star P , P_{n} \rangle    + \langle  \kappa_{k} \star P , P \rangle  \\
	= &  \frac{2}{n^{2}} \sum_{i < j} \bar{\kappa}_{k}(Z_{i}-Z_{j}) + \frac{\kappa_{k}(0)}{n} - \frac{2}{n^{2}} E_{P}[\kappa_{k}(Z_{1}-Z)] + \frac{1}{n} E[\kappa_{k}(Z-Z')]
	\end{align*}
	where $(z,z,') \mapsto \bar{\kappa}_{h}(z-z') \equiv \kappa_{h}(z-z') -  E_{P}[ \kappa_{h}(z-Z) ] - E_{P}[ \kappa_{h}(z'-Z) ]  + E_{P \cdot P}[ \kappa_{h}(Z-Z') ] $.	
	
	Therefore, 
	\begin{align*}
	\sqrt{E \left[ \left( \eta_{k}(P_{n} , P)    \right)^{2}    \right]}  \leq & 2\sqrt{E \left[ \left(  \frac{1}{n^{2}} \sum_{i < j} \bar{\kappa}_{k}(Z_{i}-Z_{j})  \right)^{2}    \right]} + \frac{\kappa_{k}(0)}{n} + \frac{2}{n^{2}} \sqrt{E \left[ \left(  E[\kappa_{k}(Z_{1}-Z)]  \right)^{2}    \right]}\\
	&  + \frac{1}{n} E[\kappa_{k}(Z-Z')].
	\end{align*}
	We now bound each term on the RHS. First note that 	
	\begin{align*}
	\frac{1}{n} E[\kappa_{h}(Z-Z')] = & \frac{1}{n} \int k \kappa(k(z-z')) p(z) p(z') dz dz' \\
	= & \frac{1}{n} \int \kappa(u) p(z' + u/k) p(z') dz' du \leq n^{-1} ||p||_{L^{\infty}},
	\end{align*}
	and 
	\begin{align*}
	\sqrt{E \left[ \left(  E[\kappa_{k}(Z_{1}-Z)]  \right)^{2}    \right]} \leq & \sqrt{E \left[ (\kappa_{k}(Z'-Z))^{2} \right]} \\
	= & \sqrt{\int  (k \kappa(k(z'-z)))^{2} p(z)p(z')dzdz'} \\
	= & \sqrt{k\int  (\kappa(u))^{2} p(z + u/k)p(u)dzdu} \leq k^{1/2} \sqrt{||p||_{L^{\infty}}} ||\kappa||_{L^{2}}.
	\end{align*}
	where the first line follows by Jensen inequality. Finally, by \cite{GineNickl2008} Sec. 2
	\begin{align*}
	\sqrt{E \left[ \left(  \frac{1}{n^{2}} \sum_{i < j} \bar{\kappa}_{k}(Z_{i}-Z_{j})  \right)^{2}    \right]} \leq \frac{2}{\sqrt{n^{2}}} \sqrt{E[(\bar{\kappa}_{k}(Z-Z'))^{2}]} \leq \frac{2 ||\kappa||_{L^{2}} \sqrt{||p||_{L^{2}}}}{n\sqrt{1/k}}. 
	\end{align*}
\end{proof}

\stopcontents[section2]
\end{document}